\documentclass{amsart}
\usepackage{times,mathrsfs,array,amssymb,pb-diagram,multirow}

\newcounter{lemma}
\newtheorem{Theorem}{Theorem}
\newtheorem{Lemma}[lemma]{Lemma}
\newtheorem{Corollary}[lemma]{Corollary}
\newtheorem{Proposition}[lemma]{Proposition}

\theoremstyle{definition}
\newtheorem{Example}[lemma]{Example}
\newtheorem{Definition}[lemma]{Definition}
\newtheorem{Notation}[lemma]{Notation}

\newtheorem{Remark}[lemma]{Remark}

\def\H{\mathbb H}

\def\GL{\mathrm{GL}}

\def\Gal{\mathrm{Gal}}
\def\End{\mathrm{End}}
\def\tr{\operatorname{tr}}
\def\nm{\operatorname{nr}}

\def\sA{\mathscr A}
\def\B{\mathcal B}
\def\C{\mathbb C}

\def\I{\mathcal I}
\def\O{\mathcal O}

\def\Q{\mathbb Q}
\def\R{\mathbb R}
\def\Z{\mathbb Z}
\def\fX{\mathfrak X}
\def\fY{\mathfrak Y}
\def\sA{\mathscr A}
\def\sD{\mathscr D}
\def\sQ{\mathscr Q}
\def\QQ{\mathcal Q}

\def\new{\mathrm{new}}

\def\mod{\  \mathrm{mod}\ }

\def\Im{\mathrm{Im\,}}
\def\gen#1{\langle #1\rangle}
\def\JS#1#2{\left(\frac{#1}{#2}\right)}

\def\SL{\mathrm{SL}}
\def\Sp{\mathrm{Sp}}
\def\CM{\mathrm{CM}}
\def\wt{\widetilde}
\def\M#1#2#3#4{\begin{pmatrix}#1&#2\\#3&#4\end{pmatrix}}
\def\SM#1#2#3#4{\left(\begin{smallmatrix}#1&#2\\#3&#4\end{smallmatrix}
  \right)}
\def\div{\mathrm{div}\,}

\def\->{\rightarrow}
\def\<->{\leftrightarrow}
\def\ol{\overline}
\def\~{\widetilde}
\newcommand{\tabcaption}{\def\@captype{table}\caption}

\begin{document}

\title{Quaternionic loci in Siegel's modular threefold}

\author{Yi-Hsuan Lin}
\address{Department of Applied Mathematics, National Chiao Tung
  University, Hsinchu, Taiwan 300}
\email{philiplin1212.am98g@nctu.edu.tw}

\author{Yifan Yang}
\address{Department of Mathematics, National Taiwan
  University and National Center for Theoretical Sciences, Taipei,
  Taiwan 10617}  
\email{yangyifan@ntu.edu.tw}
\date{\today}
\thanks{The authors were partially supported by Grant
  102-2115-M-009-001-MY4 of the Ministry of Science and Technology, Taiwan (R.O.C.).}

\begin{abstract} Let $\QQ_D$ be the set of moduli points on Siegel's
  modular threefold whose corresponding principally polarized abelian
  surfaces have quaternionic multiplication by a maximal order $\O$ in
  an indefinite quaternion algebra of discriminant $D$ over $\Q$ such
  that the Rosati involution coincides with a positive involution of
  the form $\alpha\mapsto\mu^{-1}\ol\alpha\mu$ on $\O$ for some
  $\mu\in\O$ with $\mu^2+D=0$. In this paper, we first give a formula
  for the number of irreducible components in $\QQ_D$, strengthening
  an earlier result of Rotger. Then for each irreducible component of
  genus $0$, we determine its rational parameterization in terms of a
  Hauptmodul of the associated Shimura curve.
\end{abstract}
\maketitle

\section{Introduction}
Let $A$ be a simple abelian surface over $\C$. According to a
well-known classification of Albert (see \cite[Section 21]{Mumford}),
the endomorphism algebra of $A$ is either $\Q$, a real quadratic
field, an indefinite quaternion algebra over $\Q$, or a CM field of
degree $4$ over $\Q$. In this paper, we are concerned with the third
case, i.e., abelian surfaces with quaternionic multiplication.

Let $\O$ be a maximal order in an indefinite quaternion algebra $\B_D$
of discriminant $D$ over $\Q$. We let $\QQ_D$ be the set of all
isomorphism classes of principally polarized abelian surfaces
$(A,\rho)$ with $\O\subset\End(A)$ such that the Rosati involution
coincides with a positive involution of the form
$\alpha\mapsto\mu^{-1}\ol\alpha\mu$ on $\O$ for some $\mu\in\O$ with
$\mu^2+D=0$, where $\ol\alpha$ denotes the quaternion conjugate of
$\alpha$. (Any abelian surface $A$ with $\O\subset\End(A)$ has such a
principal polarization.) We regard $\QQ_D$ as a subset of Siegel's
modular threefold $\sA_2:=\Sp(4,\Z)\backslash\H_2$ and call it the
\emph{quaternionic locus} of discriminant $D$. Here $\H_2$ denotes the
Siegel upper half-space of degree $2$.

In \cite{Rotger-Igusa}, it is shown that $\QQ_D$ consists of a finite
number of irreducible components, each of which is the image of the
Shimura curve $X_0^D(1)$ under some natural map
$\psi:X_0^D(1)\to\sA_2$. Such a map $\psi$ factors through a subgroup
$W$ of the Atkin-Lehner group $W_D$ and the resulting map from the
quotient curve $X_0^D(1)/W$ to $\sA_2$ is generically injective,
although it happens quite often that the image has
self-intersection. For convenience, we will also call an irreducible
component in $\QQ_D$ a \emph{Shimura curve}. To avoid confusion, we
shall use the Fraktur font, such as $\fX$, to denote a Shimura curve
on $\sA_2$, while the regular Roman font, such as $X$, will be used to
denote a Shimura curve in the usual sense. Also, the term
\emph{CM-point} could mean a CM-point on a Shimura curve $X$ in the
usual sense or the corresponding point on $\fX$, depending on the
context.

\begin{Remark} Note that an abelian surface $A$ with quaternionic
  multiplication by $\O$ may have a principal polarization $\rho$ such
  that $(A,\rho)\notin\QQ_D$, i.e., the restriction of the Rosati
  involution to $\O$ does not coincide with any
  positive involution of the form $\alpha\mapsto\mu^{-1}\ol\alpha\mu$
  for some $\mu\in\O$ with $\mu^2+D=0$. When this happens, $\End(A)$
  is strictly bigger than $\O$. In such a case, the abelian 
  surface is isogenous to the product of two elliptic curves with
  complex multiplication by the same imaginary quadratic field. In
  Example \ref{example: QM not on QD}, we will give such an example.
\end{Remark}

Let $r_D$ be the number of Shimura curves in the quaternion
locus $\QQ_D$. Rotger \cite{Rotger2002,Rotger-Crelle} showed that
\begin{equation} \label{equation: Rotger rD}
\begin{cases}
r_D=\displaystyle\frac{1}{2^k} \~{h}(D), & \text{ if }\B_D\text{ does
  not admit a twisting},\\
\displaystyle\frac{1}{2^k} \~{h}(D)<r_D\leq \frac{1}{2^{k-1}} \~{h}(D), & \text{ otherwise,}\\
\end{cases}
\end{equation}
where, using the terminology in \cite{Rotger2002,Rotger-Crelle}, we say
$\B_D$ admits a twisting if there is a positive divisor
$m$ of $D$ such that
$$
  \left(\frac{-D,m}\Q\right)\simeq \B_D
$$
and $\~h(D)$ is defined by
$$
\~{h}(D)=\begin{cases}
h(-4D)+h(-D), & \text{ if } D\equiv 3\mod 4,\\
h(-4D), & \text{ otherwise}\\
\end{cases}
$$
with $h(d)$ being the class number of the quadratic order of
discriminant $d$. Our first main result in this paper is a refinement
of the result of Rotger.

\begin{Theorem} %[Corollary \ref{Cor:First Main Result}]
\label{theorem: main 1}
Let $k$ be the number of prime divisors of $D$. Then
$$
r_D=\frac{1}{2^k}\~{h}(D)+\frac{1}c\#\left\{m|D:\left(\frac{-D,m}{\Q}\right)\simeq \B_D\right\},
$$
where 
$$
c=\begin{cases}
4, & \text{ if } D \text{ is even},\\
2, & \text{ if } D \text{ is odd}.
\end{cases}
$$
\end{Theorem}

\begin{Example} If $D=2p$, $p$ an odd prime, then the discriminant of
  the quaternion algebra $\JS{-D,2}\Q\simeq\JS{-D,p}\Q$ is either $1$
  or $D$, depending on whether $p$ is congruent to $\pm1$ or $\pm 3$
  modulo $8$. Therefore, from Theorem \ref{theorem: main 1}, we obtain
  $$
    r_D=\frac14h(-8p)+\begin{cases}
    0, &\text{if }p\equiv\pm 1\mod 8, \\
    1/2, &\text{if }p\equiv\pm 3\mod 8. \end{cases}.
  $$
  Similarly, if $D=pq$ is a product of two odd primes, then
  $$
    r_D=\frac14\~h(D)+\begin{cases}
    0, &\text{if }p,q\equiv 3\mod 4, \\
    \displaystyle\frac12\left(1-\JS qp\right),
    &\text{else}. \end{cases}
  $$
  Using these formulas, We find the following values of $r_D$ for the
  first few $D$.
  $$ \extrarowheight3pt
  \begin{array}{c|cccccccccccc} \hline\hline
  D & 6 & 10 & 14 & 15 & 21 & 22 & 26 & 33 & 34 & 35 & 39 & 46 \\ \hline
 r_D& 1 & 1  &  1 & 2  & 1  &  1 & 2  & 1  & 1  & 3  & 2 & 1 \\ 
  \hline\hline
  \end{array}
  $$
\end{Example}

In fact, Theorem \ref{theorem: main 1} is a corollary to Theorem
\ref{theorem: correspondences} below. In Theorem \ref{theorem:
  correspondences}, we will see that to each Shimura curve
in $\QQ_D$, we may associate a positive definite integral binary
quadratic form $ax^2+bxy+cy^2$ with properties
\begin{enumerate}
\item $b^2-4ac=-16D$,
\item $a\equiv 0,1\mod 4$, and
\item $\JS{-D,a}\Q\simeq\B_D$,
\end{enumerate}
and the number $r_D$ is precisely equal to the number
of $\mathrm{GL}(2,\Z)$-equivalence classes of such quadratic forms.
In other words, the problem of counting the number of Shimura curves in
$\QQ_D$ is reduced to that of counting the number of equivalence
classes of certain quadratic forms with prescribed properties, which
we can solve using theory of quadratic forms.
% See Appendix A for
%the list of all quadratic forms such that $D<100$ or the
%corresponding Shimura curves $\fX$ have genus $0$.
These quadratic forms are initially defined in terms of singular
relations satisfied by Shimura curves and were studied in
\cite{Hashimoto,Runge}. They have the property that a
Shimura curve $\fX$ lies on the Humbert surface $H_n$ of discriminant
$n$ if and only if its associated quadratic form primitively
represents $n$.
%, but they can also be defined purely algebraically in terms of quaternion
%algebras. 

Having determined $r_D$, our next goal is to find a rational
parameterization for each irreducible component of genus zero in terms
of a Hauptmodul of its associated Shimura curve. To
state the problem more concretely, we recall that the Igusa invariants
$J_m$, $m=2,4,6,10$, of a curve of genus $2$ over $\C$ are equal to
the values of certain (meromorphic) Siegel modular forms of weight $m$
at the moduli point corresponding to its Jacobian with the canonical
principal polarization (see
\cite{Igusa-1962}). Now let $\fX$ be an irreducible component in
$\QQ_D$ and $X=X_0^D(1)/W$, $W<W_D$, be the Shimura curve associated
to it. It can be shown that the restriction of a Siegel modular form
of weight $m$ along an irreducible component $\fX$ is a
modular form of weight $2m$ on the associated Shimura curve. In
particular, if $\fX$ is of genus $0$, then $(J_m^{k/2}/J_k^{m/2})|_\fX$
is equal to a rational function in a Hauptmodul of the Shimura
curve. Our goal here is to find these rational functions. This will
provide a simple method to determine whether a curve of genus $2$ has
quaternionic multiplication by a maximal order in the quaternion
algebra. As far as we know, this has been done only for the cases
$D=6$ and $D=10$.

In \cite{Baba-Granath}, Baba and Granath showed that if $j$ is the
Hauptmodul of $X_0^6(1)/W_6$ such that it takes values $0$, $\infty$
and $-16/27$ at the CM-points of discriminants $-4$, $-3$, and $-24$,
respectively, then an isomorphism from $X_0^6(1)/W_6$ to the
quaternionic locus $\QQ_6$ in terms of Igusa's coordinates
$[J_2,J_4,J_6,J_{10}]$ is given by
\begin{equation} \label{equation: Baba}
  \tau\longmapsto [12(j(\tau)+1),6(j(\tau)^2+j(\tau)+1),
  4(j(\tau)^3-2j(\tau)^2+1),j(\tau)^3],
\end{equation}
in the weighted projective space $\mathbb{WP}[1,2,3,5]$.
In particular, the Jacobian of an algebraic curve of genus $2$ over
$\C$ lies in $\B_6$ if and only if its Igusa
invariants $[J_2,J_4,J_6,J_{10}]$ are equal to
$[12(j+1),6(j^2+j+1),4(j^3-2j^2+1),j^3]$ in $\mathbb{WP}[1,2,3,5]$ for
some $j\neq 0,\infty$. (Note that points on $J_{10}=0$ correspond to
abelian surfaces that are products of elliptic curves with product
principal polarization. Such abelian surfaces cannot be Jacobians of
curves of genus $2$ with the canonical principal polarization.) Baba
and Granath \cite{Baba-Granath} also had an 
analogous result for the case $D=10$. Their method relied on
an earlier work of Hashimoto and Murabayashi
\cite{Hashimoto-Murabayashi}, in which explicit families of curves of
genus $2$ with quaternionic multiplication by maximal orders in $\B_6$
and $\B_{10}$ were obtained using the fact that the Shimura curves of
discriminants $6$ and $10$ lie on the intersection of Humbert surfaces
of discriminants $5$ and $8$. In this paper, we adopt a very different
approach, which will be explained in Section \ref{section: method}.

\begin{Theorem}
  The tables in Appendix \ref{appendix: parameterizations}
  give a complete list of rational parameterizations of irreducible
  components of genus zero in quaternionic loci in terms of
  Hauptmoduls of their associated Shimura curves.
\end{Theorem}

In principle, we can use our method to determine modular
parameterizations for many other curves of positive genus, but since
there are too many of them, we decide to restrict our attention to the
case of genus zero. Our method is straightforward, although some
computation is quite involved.

Note that recently, Elkies and Kumar
\cite{Elkies-Kumar} computed equations of many Hilbert modular
surfaces and rational parameterizations of Humbert surfaces. As
Shimura curves lie on the intersections of Humbert sufaces, it will be
an interesting problem to see whether one can use their results to
study quaternionic loci.

The rest of the paper is organized as follows. In Section
\ref{section: QM}, we review properties of principal polarizations of
abelian surfaces with quaternionic multiplication. A majority of the
material is taken from
\cite{Rotger2002,Rotger-Crelle,Rotger-TAMS,Rotger-Igusa}. In Section
\ref{section: quadratic forms}, we introduce certain quadratic forms
coming from singular relations satisfied by Shimura curves and study
their properties, setting the stage for a proof of Theorem
\ref{theorem: main 1}, given in Section \ref{section: proof of Theorem
  1}. In Section \ref{section: method}, we describe our
method for determining modular parameterizations of Shimura curves in
$\sA_2$ and give some examples. The main ingredients, ternary
quadratic forms associated to CM-points and the method of Borcherds
forms and the method of explicit Shimizu lifting for evaluating
modular functions, are discussed in Sections \ref{section: singular
  relations} and \ref{section: evaluations}, respectively. In
addition, as an auxiliary tool, we introduce modular curves on $\sA_2$
in Section \ref{section: modular curves}.

The computational results are given in appendices. In Appendix
\ref{appendix: list of Shimura curves}, we list all Shimura curves
with discriminant $D<100$ and all Shimura curves of genus $0$ on
$\sA_2$, characterized by their associated binary quadratic forms.
In Appendix \ref{appendix: parameterizations}, we list rational
parameterizations for all Shimura curves of genus $0$. In Appendix
\ref{appendix: Mestre}, we give Mestre obstructions for Shimura
curves with $D<80$. Note that the Mestre obstruction for a moduli
point in $\sA_2\backslash H_4$ with Igusa invariants
$J=[J_2,J_4,J_6,J_{10}]$ is a quaternion algebra over $\Q(J)$ such
that there is a curve of genus $2$ over a field $K$ with Igusa
invariants $J$ if and only if $K$ splits the quaternion algebra.
Our computation suggests that when a Shimura curve $\fX$ in $\QQ_D$ is
isomorphic to $X_0^D(1)/w_D$, there are only a finite number of
isomorphism classes of algebraic curves of genus $2$ over $\R$ such
that their Jacobians lie on $\fX$. For example, for the case $D=14$,
there are exactly four such curves (see Example \ref{example: Mestre
  14}).

\section{Abelian surfaces with QM and their principal
    polarizations}
\label{section: QM}
To facilitate subsequent discussion, we shall review some properties
of abelian surfaces with quaternionic multiplication (QM) and their
principal polarizations in this section.

Let $A=\C^2/\Lambda$ be the complex torus associated to a lattice
$\Lambda$ of rank $4$ in $\C^2$. If $A$ has quaternionic
multiplication by an order $\O$ in an indefinite quaternionic algebra
$\B$ of discriminant $D>1$ over $\Q$, then $\Lambda$ is
naturally a left $\O$-module and $\Lambda\otimes_\Z\Q$ is a left
$\B$-module. Since $\B$ is a division algebra of dimension $4$
over $\Q$, this $\B$-module must be free of rank $1$, i.e., we have
$\Lambda\otimes\Q=\B.v$ for some $v\in\C^2$ and consequently
$\Lambda=\I.v$ for some left ideal $\I$ of $\O$. Assume that $\O$ is
an Eichler order. By a well-known result of Eichler, $\I$ is
necessarily a principal ideal and thus we may 
as well assume that $\Lambda=\O.v$. Now for each $\alpha\in\B$,
there exists a unique $\SM abcd\in M(2,\R)$ such that $\alpha.v=\SM
abcd v$. The map $\alpha\mapsto\SM abcd$ defines an embedding $\phi$
from $\B$ into $M(2,\R)$. In summary, we find that if $\C^2/\Lambda$
has quaternionic multiplication by an Eichler order $\O$, then
$\Lambda=\phi(\O)v$ for some embedding $\phi:\B\to M(2,\R)$ and some
vector $v\in\C^2$.

Now any two different embeddings $\phi_1$ and $\phi_2$ of $\B$
into $M(2,\R)$ are related by conjugation by a matrix in
$\mathrm{GL}(2,\R)$, say, $\phi_2=\gamma^{-1}\phi_1\gamma$. Then
$$
  \C^2/\phi_2(\O)v\simeq\C^2/\phi_1(\O)v', \qquad v'=\gamma v.
$$
Thus, if we let $\phi$ be a fixed embedding of $\B$ into $M(2,\R)$,
then every complex torus with quaternionic multiplication by $\O$ is
isomorphic to $\C^2/\phi(\O)v$ for some $v=\left(\begin{smallmatrix}
z_1\\z_2\end{smallmatrix}\right)\in\C^2$. Switching $z_1$ and $z_2$ if
necessary, we may assume that $\Im(z_1/z_2)>0$. Setting $z=z_1/z_2$,
we find that $\C^2/\phi(\O)v$ is isomorphic to $\C^2/\phi(\O)v_z$ with
$v_z=\left(\begin{smallmatrix}z\\1\end{smallmatrix}\right)$, $z\in\H$.
Also, it is easy to see that if $\gamma\in\phi(\O^1)$, then
$\C^2/\phi(\O)v_z\simeq \C^2/\phi(\O)v_{\gamma z}$, where $\O^1$ is
the norm-one group of $\O$. Thus, there is a well-defined map from the
Shimura curve $\phi(\O)\backslash\H$ to the set of isomorphism classes
of complex tori of dimension $2$ with quaternionic multiplication by
$\O$.

From now on, we assume that $\O$ is a maximal order. We fix an
embedding $\phi$ of $\B$ into $M(2,\R)$ and for $z\in\H$, we let
$$
  \Lambda_z=\phi(\O)v_z, \qquad A_z=\C^2/\Lambda_z,
$$
where $v_z=\left(\begin{smallmatrix}z\\1\end{smallmatrix}\right)$.
We now describe its principal polarizations.

Let $\mu$ be a pure quaternion in $\B$ such that $\nm(\mu)>0$ and
$E_\mu$ be the symplectic form on $\C^2$ defined by setting
\begin{equation} \label{equation: Emu 0}
  E_\mu(\phi(\alpha)v_z,\phi(\beta)v_z)=\tr(\mu^{-1}\alpha\ol\beta)
\end{equation}
for $\alpha,\beta\in\B$ and extending $\R$-linearly to $\C^2$.
It can be shown that there exists an integer $k$ such that $kE_\mu$
is a Riemann form on $A_z$. In particular, $A_z$ can be polarized.
Conversely, if $A_z$ is a simple abelian surface (i.e., $z$ is not a
CM-point on the Shimura curve) and $E$ is a Riemann
form on $A_z$, then $E=E_\mu$ for some pure quaternion $\mu$ with
positive norm (see, e.g., \cite[Proposition 4.6]{Shimura-CM1} or
\cite[Theorem 4.3, Chapter IX]{Lang}). Note that the requirement that
$E$ takes integer values on $\Lambda_z\times\Lambda_z$ implies that
$$
  \mu^{-1}\in\O^\vee:=\{\alpha\in\B:\tr(\alpha\ol\beta)\in\Z
  \text{ for all }\beta\in\O\},
$$
i.e., $\mu$ is contained in the different of $\O$.
For the case of a maximal order $\O$, this means that $\mu$ is an
element in $\O$ whose trace is $0$ and whose norm is a positive
integer multiple of $D$. In fact, if $E_\mu$ defines a principal
polarization, then the norm of $\mu$ must be precisely $D$. 

Note that if we pick a pure quaternion $\mu$ with a positive norm
from $\B$, the Hermitian form associated to $E_\mu$
could be positive definite or negative definite. For the purpose of
practical computation later on, we need a simple criterion when the
Hermitian form is positive definite.

\begin{Lemma} Let $\mu$ be a nonzero element of trace $0$ and positive
  norm in $\B$. Let $H_\mu:\C^2\times\C^2\to\C$ be the Hermitian form
  defined by $H_\mu(v_1,v_2)=E_\mu(iv_1,v_2)$ for
  $v_1,v_2\in\C^2$. Assume that $\phi(\mu)=\SM abc{-a}$. Then $H_\mu$
  is positive definite if and only if $c>0$.
\end{Lemma}

\begin{proof}
  Assume that $\mu$ is a pure quaternion with a positive norm.
  The definiteness of the Hermitian form $H_\mu$ is known
  (for instance, see the proof of \cite[Theorem 4.3, Chapter
  IX]{Lang}). Thus, we only need to know the sign of $H_\mu(v,v)$ at
  one particular nonzero $v\in\C^2$. Here we choose
  $$
    v=\M 1000v_z.
  $$
  Assume that $z=x+iy$. Then we have
  $$
    iv_z=\frac 1y\M x{-x^2-y^2}1{-x}v_z,
  $$
  and
  $$
    iv=\M1000(iv_z)=\frac1y\M1000\M x{-x^2-y^2}1{-x}v_z
  $$
  Therefore, if $\phi(\mu)=\SM abc{-a}$, then
  \begin{equation*}
  \begin{split}
    H_\mu(v,v)&=E_\mu(iv,v)=\frac1{y\nm(\mu)}\tr\left(
    \M{-a}{-b}{-c}a\M1000\M{x}{-x^2-y^2}{1}{-x}\M0001\right) \\
  &=\frac{c(x^2+y^2)}{y\nm(\mu)}.
  \end{split}
  \end{equation*}
  We conclude that $H_\mu$ is positive definite if and only if $c>0$.
\end{proof}

In view of this lemma, it is natural to introduce the following
definition.

\begin{Definition} An element $\mu$ of trace $0$ and positive norm in $\B$
  is said to be \emph{positive} (with respect to $\phi$) if the
  $(2,1)$-entry of $\phi(\mu)$ is positive. If $\mu$ is not positive,
  then we say it is \emph{negative}.
\end{Definition}

\begin{Notation} We let $\sD(\O)$ denote the set of all elements of
  trace $0$ and norm $D$ in $\O$ and set
  $\sD^+(\O):=\{\mu\in\sD(\O):\mu\text{ positive}\}$.

  For $\mu\in\sD^+(\O)$, let $\rho_\mu$ denote the
  principal polarization of $A_z$ corresponding to $E_\mu$, and set
  $$
    \fX_\mu:=\{(A_z,\rho_\mu):~z\in X_0^D(1)\}.
  $$
\end{Notation}

\begin{Lemma} \label{lemma: positive}
  Let $\mu$ be an element of trace $0$ and norm $D$ in $\B$. Assume
  that $\mu$ is positive. Then for $\gamma\in\B^\times$, the element
  $\gamma^{-1}\mu\gamma$ is positive if $\nm(\gamma)>0$ and negative
  if $\nm(\gamma)<0$.
\end{Lemma}

\begin{proof} Assume that $\phi(\mu)=\SM abc{-a}$ and
  $\phi(\gamma)=\SM xyzw$. We have
  $$
    \phi(\gamma^{-1}\mu\gamma)=\M\ast\ast
    {(cx^2-2axz-bz^2)/\nm(\gamma)}\ast.
  $$
  The quadratic form $cx^2-2axz-bz^2$ has discriminant $4(a^2+bc)=-4D$
  and hence is positive definite. Then the lemma follows.
\end{proof}

Of course, two elements in $\sD^+(\O)$ may define isomorphic principal
polarizations.

\begin{Lemma}[{\cite[Theorem 2.2]{Rotger-Crelle}}] Let $\mu_1$ and
  $\mu_2$ be two elements in $\sD^+(\O)$. Then
  $\rho_{\mu_1}\simeq\rho_{\mu_2}$ if and only if there exists an
  element $\alpha$ in the group $\O^\times$ of units in $\O$ such that
  $\mu_1=\ol\alpha\mu_2\alpha$.
\end{Lemma}

Following \cite{Rotger-Crelle}, we call $\ol\alpha\mu\alpha$ the
\emph{Pollak conjugation} of $\mu$ by $\alpha$ (cf. \cite{Pollak}).

\begin{Definition}
We let the equivalence relation $\sim_p$ on $\sD^+(\O)$ be defined
by $\mu_1\sim_p\mu_2$ if and only if there exists a unit $\alpha$ in
$\O$ such that $\mu_1=\ol\alpha\mu_2\alpha$. Also, we let
$$
  \~\sD(\O):=\sD^+(\O)/\sim_p
$$
denote the set of equivalence classes under $\sim_p$.
\end{Definition}

\begin{Corollary}[{\cite[Theorem 1.1 and Corollary 3.7]{Rotger-Crelle}}]
  Let $A$ be an abelian surface such that $\End(A)\simeq\O$. Then the
  number $\pi(A)$ of isomorphism classes of principal polarizations on
  $A$ is equal to the cardinality of $\~\sD(\O)$. Consequently, we have
  $$
    \pi(A)=\begin{cases}
    (h(-D)+h(-4D))/2, &\text{if }D\equiv 3\mod 4, \\
    h(-4D)/2, &\text{else}. \end{cases}
  $$
\end{Corollary}

Note that it is possible that two elements $\mu_1$ and $\mu_2$ in
$\sD^+(\O)$ define two inequivalent principal polarizations, but
$\fX_{\mu_1}=\fX_{\mu_2}$. That is, it may happen that for each
$z_1\in X_0^D(1)$, there is another point $z_2\in X_0^D(1)$ such that
$(A_{z_1},\rho_{\mu_1})$ and $(A_{z_2},\rho_{\mu_2})$ are the same
point in $\sA_2$. For example, we can easily check that if $\omega\in
N_\B^+(\O)$, the normalizer of $\O$ in $\B$ with positive norm, then
$(A_z,\rho_\mu)$ and $(A_{\phi(\omega)z},\rho_{\omega\mu\omega^{-1}})$
are isomorphic principally polarized abelian surfaces. In particular,
we have $\fX_\mu=\fX_{\omega\mu\omega^{-1}}$. In fact, the converse
statement is also true.

\begin{Proposition}[{\cite[Theorem 6.4]{Rotger-Igusa}}]
  \label{proposition: fX1=fX2}
  Let $\mu_1,\mu_2\in\sD^+(\O)$. Then $\fX_{\mu_1}=\fX_{\mu_2}$ if and
  only if there exists $\omega\in N_\B^+(\O)$ such that the
  equivalence class of $\omega\mu_1\omega^{-1}$ in $\~\sD(\O)$
  is the same as that of $\mu_2$.
\end{Proposition}

It follows from the proposition that the number $r_D$ of irreducible
components is equal to the number of orbits of $\~\sD^+(\O)$
under the action of the Atkin-Lehner group $W_D=N_\B^+(\O)/\Q^\times\O^1$,
where $\O^1$ denotes the norm-one group of $\O$. Therefore, to obtain
a formula for $r_D$, one would need to study the subgroup of $W_D$ that
fixes an equivalence class in $\~\sD(\O)$.

\begin{Definition}[{\cite[Definition 4.2]{Rotger-TAMS}}]
\label{definition: stable}
  For $\mu\in\sD^+(\O)$, the \emph{stable group} $W_\mu$ is defined to
  be
  $$
    W_\mu:=\{\omega\in W_D:\omega\mu\omega^{-1}\sim_p\mu\}.
  $$
\end{Definition}

It is clear that $W_\mu$ contains at least the identity
and $\mu$. Rotger \cite[Corollary 4.5]{Rotger-TAMS} showed that
$|W_\mu|$ can only be $2$ or $4$ and gave a complete description of
$W_\mu$.

\begin{Definition}[{\cite[Page 1542]{Rotger-TAMS}}] Let $\mu\in\sD^+(\O)$.
  We say $\mu$ is \emph{twisting} if there exists a positive divisor
  $m$ of $D$ and an element $\chi$ in $\O$ such that
  \begin{equation} \label{equation: twisting}
   \chi^2=m, \qquad \mu\chi=-\chi\mu.
  \end{equation}
\end{Definition}

Note that the conditions in \eqref{equation: twisting} implies that
the quaternion algebra $\JS{-D,m}\Q$ is isomorphic to $\B$.
Conversely, it is obvious that if $\JS{-D,m}\Q\simeq\B$ for some
positive divisor $m$ of $D$, then $\sD^+(\O)$ contains a twisting
element $\mu$. In such a case, we say the quaternion algebra $\B$ is
twisting.

\begin{Proposition}[{\cite[Corollary 4.5]{Rotger-TAMS}}]
  \label{proposition: Rotger Wmu}
  Let $\mu\in\sD^+(\O)$. If there exists a positive divisor $m$ of $D$
  and an element $\chi$ in $\O$ with properties \eqref{equation:
    twisting}, then $W_\mu=\gen{w_m,w_D}$; otherwise, $W_\mu=\gen{w_D}$.
\end{Proposition}

Then the result \eqref{equation: Rotger rD} of Rotger about $r_D$
follows immediately from this lemma.
We summarise the discussion in this section in the following
proposition.

\begin{Proposition}[{\cite[Proposition 4.3]{Rotger-Igusa}}]
  We have
  $$
    \QQ_D=\bigcup_{\mu}\fX_\mu,
  $$
  where $\mu$ runs through a set of representatives of the orbits of
  $\~\sD(\O)$ under the action of the Atkin-Lehner group
  $W_D$.
\end{Proposition}

\section{Quadratic forms associated to Shimura curves}
\label{section: quadratic forms}
Let all the notations, such as $D$, $\B$, $\O$, $\phi$, $\sD^+(\O)$,
$\~\sD$, $E_\mu$, $\rho_\mu$, $\fX_\mu$, and etc., be defined as in
the previous section. In order to prove Theorem \ref{theorem: main 1},
we shall introduce certain quadratic forms associated to Shimura curves.
We first review the notion of singular relations.

Let $(A,\rho)$ be a
principally polarized abelian surface over $\C$. In terms of the
Riemann form $E$ of $(A,\rho)$, the Rosati involution
$\dagger:\End(A)\to\End(A)$ is defined by $f\mapsto f^\dagger$, where
$f^\dagger$ is the endomorphism characterized by the property
$E(f\omega_1,\omega_2)=E(\omega_1,f^\dagger\omega_2)$ for all periods
$\omega_1$ and $\omega_2$ of $A$. An endomorphism $f$ of $A$ is said to be
symmetric with respect to the Rosati involution or Rosati invariant if $f^\dagger=f$.

Now suppose that the normalized period matrix of $(A,\rho)$ is
$\tau=\SM{\tau_1}{\tau_2}{\tau_2}{\tau_3}\in \sA_2$.
An endomorphism $f$ of $A$ can be represented by a matrix
$R_f\in M(2,\C)$ such that
$$
  R_f(\tau,1)=(\tau, 1)M_f=(\tau\alpha+\beta,\tau\gamma+\delta)
$$
for some $M_f=\SM\alpha\gamma\beta\delta\in M(4,\Z)$.
Then $f$ is symmetric with respect to the Rosati
involution if and only if the matrix $M_f$ satisfies
$$
  M_f^t\M{0}{-1}{1}{0}=\M{0}{-1}{1}{0}M_f,
$$
that is, if and only if $M_f$ is of the form
\begin{equation} \label{equation: Rosati matrix}
  M_f=\begin{pmatrix}a_1&a_2&0&b\\a_3&a_4&-b&0\\
  0&c&a_1&a_3\\-c&0&a_2&a_4\end{pmatrix}.
\end{equation}
Furthermore, the relation $(\tau\gamma+\delta)\tau=\tau\alpha+\beta$ is
equivalent to the relation
$$
  a_2\tau_1+(a_4-a_1)\tau_2-a_3\tau_3+b(\tau_2^2-\tau_1\tau_3)+c=0.
$$
Clearly, the relation is nontrivial if and only if $f$ is not a
multiplication-by-$n$ endomorphism.

Conversely, if the period matrix $\tau$ satisfies
\begin{equation} \label{equation: singular relation}
  a\tau_1+b\tau_2+c\tau_3+d(\tau_2^2-\tau_1\tau_3)+e=0
\end{equation}
for some integers $a,b,c,d,e\in\Z$, then the corresponding abelian
surface has an endomorphism $f$ whose matrix $M_f$ is
\begin{equation} \label{equation: Rosati matrix 2}
  \begin{pmatrix}0&a&0&d\\-c&b&-d&0\\
  0&e&0&-c\\-e&0&a&b\end{pmatrix}
\end{equation}
which is a matrix of the form \eqref{equation: Rosati matrix}.
The relation in \eqref{equation: singular relation} is called a
\emph{singular relation}. It is \emph{primitive} if
$\gcd(a,b,c,d,e)=1$. Note that the matrix in
\eqref{equation: Rosati matrix 2} satisfies
$M^2-bM+ac+de=0$. The discriminant $\Delta=\Delta(a,b,c,d,e)$ of the
singular relation in \eqref{equation: singular relation} is defined to
be $b^2-4ac-4de$, the discriminant of the quadratic polynomial
$x^2-bx+ac+de$. Using the condition that $\Im\tau$ is positive
definite, we can deduce that $\Delta>0$. In other words, if we let
$L$ be the set of singular relations satisfied by a given
$\tau\in\H_2$, then $\Delta: L\to\Z$ is a
positive definite quadratic form on $L$. In the following, we let
$\gen{\cdot,\cdot}_\Delta$ denote the bilinear form associated to this
quadratic form, i.e.,
\begin{equation} \label{equation: bilinear Delta}
  \gen{\ell_1,\ell_2}_\Delta:=\frac12\left(\Delta(\ell_1)
 +\Delta(\ell_2)-\Delta(\ell_1+\ell_2)\right).
\end{equation}

\begin{Remark} From the definition of $\Delta$, it is plain that
$\Delta(\ell)\equiv 0,1\mod 4$ for all $\ell$. Moreover, we can check
by explicit computation that
\begin{equation} \label{equation: parity}
  \gen{\ell_1,\ell_2}_\Delta^2\equiv \Delta(\ell_1)\Delta(\ell_2)\mod 4.
\end{equation}
\end{Remark}

We now consider the case of a Shimura curve $\fX_\mu$ in $\QQ_D$.

\begin{Lemma}[{\cite[Page 113]{Shimura-CM1}}]
  Let $(A_z,\rho_\mu)$ be a principally polarized abelian surface in
  $\fX_\mu$. The Rosati involution $\dagger:\End(A_z)\to\End(A_z)$
  with respect to $\rho_\mu$ coincides with
  $\alpha\mapsto\mu^{-1}\ol\alpha\mu$ on $\O$, where $\ol\alpha$ is the
  quaternionic conjugation of $\alpha$.
\end{Lemma}

In other words, $\alpha\in\O$ is Rosati-invariant with respect to
$\rho_\mu$ if and only if $\alpha\in\mu^\perp$, where
$$
  \mu^\perp:=\{\beta\in\O:\tr(\beta\ol\mu)=0\}.
$$
Let $L$ be the set of singular relations arising from elements in
$\mu^\perp$. Note that since $\Z\subset\mu^\perp$ yields trivial
relations, we have $L\simeq \mu^\perp/\Z$. To study properties of $L$,
we shall identify $L$ with a certain subset of $\mu^\perp$.

For $\mu\in\sD^+(\O)$, let
$$
  \gen{1,\mu}^\perp:=\{\alpha\in\O:\tr(\alpha)=~\tr(\alpha\mu)=0\}
 =\{\alpha\in\O:~\tr(\alpha)=0,~\alpha\mu+\mu\alpha=0\}
$$
be the set of elements orthogonal to $1$ and $\mu$ in $\O$, and let
\begin{equation} \label{equation: Lmu}
  L_\mu:=(\Z+2\O)\cap\gen{1,\mu}^\perp.
\end{equation}
Under the negative trace form
$\gen{\cdot,\cdot}:(\alpha,\beta)\mapsto-\tr(\alpha\ol\beta)$, $L_\mu$
is a positive definite lattice of rank $2$.

\begin{Lemma} Let $L$ and $L_\mu$ be given above. Then $L\simeq L_\mu$
  as lattices.
\end{Lemma}

\begin{proof} Recall that if $f$ is a Rosati-invariant endomorphism,
  then it satisfies $f^2-tf+n=0$ for some integers $t,n$ and the
  discriminant of the corresponding singular relation is $t^2-4n$.
  Thus, the quadratic form on $L\simeq\mu^\perp/\Z$ is given by
  $\alpha+\Z\mapsto\operatorname{disc}(\alpha)=\tr(\alpha)^2-4\nm(\alpha)$.
 
  From the definition of $L_\mu$ it is easy to see that $\nm(\alpha)\equiv
  0,3\mod 4$. Also, we have $\alpha/2\in\O$ if
  $\nm(\alpha)\equiv 0\mod 4$ and $(1+\alpha)/2\in\O$ if
  $\nm(\alpha)\equiv 3\mod 4$. Define $L_\mu\to\mu^\perp/\Z$ by
  $$
    \alpha\longmapsto\begin{cases}
    \alpha/2+\Z, &\text{if }\nm(\alpha)\equiv 0\mod 4,\\
    (1+\alpha)/2+\Z, &\text{if }\nm(\alpha)\equiv 3\mod 4. \end{cases}
  $$
  It is straightforward to check that this is an isometry between
  the two lattices.
\end{proof}

We have the following properties of $L_\mu$.

\begin{Lemma} \label{lemma: Runge}
\begin{enumerate}
\item[(1)] With respect to the negative trace form $\gen{\cdot,\cdot}:
  (\alpha,\beta)\mapsto-\tr(\alpha\overline\beta)$, $L_\mu$ is a
  positive definite lattice of discriminant $16D$.
\item[(2)] For all $\alpha\in L_\mu\backslash\{0\}$, we have
  $$
    \JS{-D,-\nm(\alpha)}\Q\simeq\B.
  $$
%\item[(3)] For all $\alpha\in L_\mu$, we have
%  $\nm(\alpha)\equiv 0,3\mod 4$. Moreover, if $\nm(\alpha)\equiv0\mod
%  4$, then $\alpha/2\in\O$, and if $\nm(\alpha)\equiv3\mod 4$, then
%  $(1+\alpha)/2\in\O$.
\item[(3)] Let $\{\alpha',\beta'\}$ be a $\Z$-basis of $L_\mu$. Set
  $$
    \alpha=\begin{cases}
    \alpha'/2, &\text{if }\nm(\alpha)\equiv 0\mod 4, \\
    (1+\alpha')/2, &\text{if }\nm(\alpha)\equiv 3\mod 4, \end{cases}
  $$
  and similarly for $\beta$. Then
  $\O=\Z+\Z\alpha+\Z\beta+\Z\alpha\beta$. Also,
  $\alpha'\beta'-\beta'\alpha'$ is equal to $4\mu$ or $-4\mu$.
\end{enumerate}
\end{Lemma}

\begin{proof} Part (1) can be proved easily by local consideration.
  It can be checked that if $p$ is a prime not dividing $2D$, 
  then $p$ does not divide the discriminant of $L_\mu$.

  If $p$ is an odd prime dividing $D$, then
  $$
    \O\otimes_\Z\Z_p\simeq\Z_p+\Z_p\mu+\Z_p\epsilon+\Z_p\mu\epsilon
  $$
  for some $\epsilon\in\O$ with $\epsilon^2\in\Z_p^\times$ and
  $\mu\epsilon=-\epsilon\mu$. Then $L_\mu\otimes_\Z\Z_p$ is spanned by
  $\epsilon$ and $\mu\epsilon$. From this, we see that the $p$-adic
  valuation of the discriminant of $L_\mu$ is $1$. We now consider the
  case $p=2$.

  If $2\nmid D$, then $\O\otimes_\Z\Z_2\simeq M(2,\Z_2)$. Let
  $\SM{c_1}{c_2}{c_3}{-c_1}$ be the image of $\mu$ under this
  isomorphism. Then an element in $\Z_2+2M(2,\Z_2)$ is orthogonal to
  $1$ and $\mu$ if and only if it is of the form $\SM a{2b}{2c}{-a}$
  for some $a,b,c\in\Z_2$ and satisfies $ac_1+bc_3+cc_2=0$. If $c_1$
  is odd, we find that
  $$
    \M{-c_3}{2c_1}0{c_3}, \qquad\M{-c_2}0{2c_1}{c_2}
  $$
  form a $\Z_2$-basis for $L_\mu\otimes_\Z\Z_2$. The Gram matrix of
  $L_\mu\otimes_\Z\Z_2$ with respect to this basis is
  $$
    \M{-2c_3^2}{-2c_2c_3-4c_1^2}{-2c_2c_3-4c_1^2}{-2c_2^2}.
  $$
  Its determinant is $-16c_1^2(c_2c_3+c_1^2)=16c_1^2D$. Hence, the
  $2$-adic valuation of the discriminant of $L_\mu$ is $4$.
  Similarly, when $c_1$ is even, $c_2$ and $c_3$ must be odd and we
  find that
  $$
    \M{-c_2}0{2c_1}{c_2}, \qquad \M0{-2c_2}{2c_3}0
  $$
  form a basis for $L_\mu\otimes_\Z\Z_2$. The Gram matrix is
  $$
    \M{-2c_2^2}{4c_1c_2}{4c_1c_2}{8c_2c_3}.
  $$
  Its determinant is $16c_2^2(-c_2c_3-c_1^2)=16c_2^2D$. Again, we find
  that the $2$-adic valuation of the discriminant of $L_\mu$ is $4$.

  Now assume that $2|D$. We have
  $\B\otimes_\Q\Q_2\simeq\JS{-1,-1}{\Q_2}$ and
  $\O\otimes_\Z\Z_2\simeq\Z_2+\Z_2I+\Z_2J+\Z_2(1+I+J+IJ)/2$, where
  $I^2=J^2=-1$ and $IJ=-JI$. Let $c_1I+c_2J+c_3IJ$ be the image of
  $\mu$ under the isomorphism. Since $2|D$ but $4\nmid D$, exactly two
  of $c_1,c_2,c_3$ are odd. Without loss of generality, we assume that
  $c_1$ and $c_2$ are odd and $c_3$ is even. The $\Z_2$-module
  $\Z_2+2\O\otimes_\Z\Z_2$ has a basis $\{1,2I,2J,I+J+IJ\}$. Then an
  element in $\Z_2+2\O\otimes_\Z\Z_2$ is in $L_\mu\otimes_\Z\Z_2$ if
  and only if it is of the form $a(2I)+b(2J)+c(I+J+IJ)$ with
  $$
    ac_1+bc_2+c\frac{c_1+c_2+c_3}2=0.
  $$
  We find that
  $$
    2c_2I-2c_1J, \qquad (c_2+c_3)I-c_1J-c_1IJ
  $$
  forms a $\Z_2$-basis for $L_\mu\otimes_\Z\Z_2$. The Gram matrix
  with respect to this basis is
  $$
    \M{8(c_1^2+c_2^2)}{4c_1^2+4c_2(c_2+c_3)}
    {4c_1^2+4c_2(c_2+c_3)}{4c_1^2+2(c_2+c_3)^2},
  $$
  whose determinant is $16c_1^2(c_1^2+c_2^2+c_3^2)=16c_1^2D$. We
  conclude that the $2$-adic valuation of the discriminant of
  $L_\mu$ is $5$. This completes the proof of Part (1). We next
  consider Part (2).

  Assume that $\alpha$ is a nonzero element in $L_\mu$. Then
  $\mu$ and $\alpha$ generate the quaternion algebra $\B$. Since
  $\mu^2=-D$ and $\alpha^2=-\nm(\alpha)$, we must have
  $$
    \JS{-D,-\nm(\alpha)}\Q\simeq\B.
  $$

%  We now prove Part (3).
%  Assume that $\alpha\in L_\mu$. Then either $\alpha\in 2\O$ or
%  $\alpha\in 1+2\O$. It is clear that the first case occurs if and
%  only if $\nm(\alpha)\equiv 0\mod 4$, while the second case occurs if
%  only if $\nm(\alpha)\equiv 3\mod 4$. This is the content of Part
%  (3).

  For Part (3), we first show that $\Z+\Z\alpha+\Z\beta+\Z\alpha\beta$
  is an order in $\B$. Let $a=\nm(\alpha')$,
  $b=\tr(\alpha'\overline{\beta'})$, and $c=\nm(\beta')$. From Part
  (1), we know that $b^2-4ac=-16D$. We check directly
  that
  \begin{equation} \label{equation: ab+ba}
    \alpha\beta+\beta\alpha=\begin{cases}
    -b/4, &\text{if }a,c\equiv0\mod 4, \\
    \beta-b/4, &\text{if }a\equiv3\mod 4\text{ and }c\equiv0\mod 4, \\
    \alpha-b/4, &\text{if }a\equiv0\mod 4\text{ and }c\equiv 3\mod
    4,\\
    \alpha+\beta-(2+b)/4, &\text{if }a,c\equiv3\mod 4. \end{cases}
  \end{equation}
  Note that in the last case, because $16|(b^2-4ac)$, we must have
  $b\equiv 2\mod 4$. From this identity, it is easy to see that
  $\Z+\Z\alpha+\Z\beta+\Z\alpha\beta$ is an order and that its Gram
  matrix with respect to the basis $\{1,\alpha,\beta,\alpha\beta\}$ is
  $$
    \begin{pmatrix}
    2&0&0&-b/4\\0&a/2&b/4&0\\0&b/4&c/2&0\\-b/4&0&0&ac/8\end{pmatrix},
    \qquad
    \begin{pmatrix}
    2&1&0&-b/4\\1&(1+a)/2&b/4&0\\0&b/4&c/2&c/4\\-b/4&0&c/4&
    (a+1)c/8 \end{pmatrix}
  $$
  $$
    \begin{pmatrix}
    2&0&1&-b/4\\0&a/2&b/4&a/4\\1&b/4&(1+c)/2&0\\-b/4&a/4&0&
    a(c+1)/8\end{pmatrix},
  $$
  $$
    \begin{pmatrix}
    2&1&1&(2-b)/4\\1&(1+a)/2&(2+b)/4&(1+a)/4\\
    1&(2+b)/4&(1+c)/2&(1+c)/4\\(2-b)/4&(1+a)/4&(1+c)/4&(1+a)(1+c)/8
    \end{pmatrix}
  $$
  in the four cases, respectively. The determinants of the four
  matrices above are all $(b^2-4ac)^2/256=D^2$. Therefore,
  $\Z+\Z\alpha+\Z\beta+\Z\alpha\beta$ is a maximal order and equals to
  $\O$.

  For the last assertion, we check that $\alpha'\beta'-\beta'\alpha'$
  commutes with $\mu$ and hence is contained in $\Q(\mu)$. Moreover,
  using $\alpha'\beta'+\beta'\alpha'=-b$, we find
  $$
    (\alpha'\beta'-\beta'\alpha')^2=b^2-4ac=-16D.
  $$
  Therefore, $\alpha'\beta'-\beta'\alpha'=\pm 4\mu$. This completes
  the proof of the lemma.
\end{proof}

The lemma shows that to each $\mu$ in $\sD^+(\O)$, we may associate a
$\GL(2,\Z)$-equivalence class of positive definite integral quadratic
form $Q_\mu(x,y)$ of discriminant $-16D$ given by
$$
  Q_\mu(x,y)=-\nm(\alpha x+\beta y),
$$
where $\{\alpha,\beta\}$ is a $\Z$-basis of $L_\mu$, with properties
that
\begin{enumerate}
\item the discriminant of $Q_\mu$ is $-16D$, and
\item if $a$ is an integer represented by $Q_\mu$, then
$$
  a\equiv 0,1\mod 4, \qquad\JS{-D,a}\Q\simeq\B.
$$
\end{enumerate}

\begin{Notation} \label{notation: sQ}
We let $\sQ_D$ denote the set of all positive definite binary integral
quadratic forms of discriminant $-16D$ having the two properties
above. Also, if $Q$ is a binary quadratic form, then we let
$[Q]_{\GL}$ and $[Q]_\SL$ denote the $\GL(2,\Z)$-equivalence class and
the $\SL(2,\Z)$-equivalence class of $Q$, respectively.
\end{Notation}

The next lemma shows that $\fX_\mu\mapsto[Q_\mu]_\GL$ is a
well-defined map from the set of Shimura curves in $\QQ_D$ to the set
$\sQ_D/\GL(2,\Z)$.

\begin{Lemma} \label{lemma: QQ to sQ}
  Let $\mu_1$ and $\mu_2$ be two elements in $\sD^+(\O)$.
  Then $\fX_{\mu_1}=\fX_{\mu_2}$ if and only if
  $[Q_{\mu_1}]_\GL=[Q_{\mu_2}]_\GL$.
\end{Lemma}

\begin{proof} Assume that $\mu_1$ and $\mu_2$ are elements of
  $\sD^+(\O)$ such that $\fX_{\mu_1}=\fX_{\mu_2}$. By Proposition
  \ref{proposition: fX1=fX2}, we have
  $\mu_2=\overline\epsilon\omega^{-1}\mu_1\omega\epsilon$ for some
  $\omega\in N_\B^+(\O)$ and $\epsilon\in\O^\times$. We can easily check
  that the function
  $\alpha\mapsto\overline\epsilon\omega^{-1}\alpha\omega\epsilon$ defines
  an isometry from the lattice $L_{\mu_1}$ onto the lattice
  $L_{\mu_2}$. Thus, $Q_{\mu_1}$ is $\GL(2,\Z)$-equivalent to
  $Q_{\mu_2}$.

  Conversely, assume that $\mu_1$ and $\mu_2$ are elements of
  $\sD^+(\O)$ such that $Q_{\mu_1}$ and $Q_{\mu_2}$ are
  $\GL(2,\Z)$-equivalent. There exist $\Z$-bases
  $\{\alpha_1',\beta_1'\}$ and $\{\alpha_2',\beta_2'\}$ for
  $L_{\mu_1}$ and $L_{\mu_2}$, respectively, such that
  \begin{equation} \label{equation: theorem 1 1}
     \nm(x\alpha_1'+y\beta_1')=\nm(x\alpha_2'+y\beta_2')
  \end{equation}
  for all $x,y\in\Z$. Set
  $$
    \alpha_j=\begin{cases}\alpha_j'/2, &\text{if }\nm(\alpha_j')\equiv
      0\mod 4, \\
    (1+\alpha_j')/2, &\text{if }\nm(\alpha_j')\equiv3\mod
    4, \end{cases}
  $$
  and similarly for $\beta_j$. By Lemma \ref{lemma: Runge},
  $\{1,\alpha_1,\beta_1,\alpha_1\beta_1\}$ and
  $\{1,\alpha_2,\beta_2,\alpha_2\beta_2\}$ are both $\Z$-bases for
  $\O$. Then \eqref{equation: theorem 1 1} implies that the function
  $f:a_0+a_1\alpha_1+a_2\beta_1+a_3\alpha_1\beta_1\mapsto
   a_0+a_1\alpha_2+a_2\beta_2+a_3\alpha_2\beta_2$, $a_j\in\Q$, is an
   isomorphism of the quaternion algebra that leaves $\O$ invariant.
  By the Noether-Skolem theorem, there exists an element $\gamma$ in
  $N_\B(\O)$ such that $f(\alpha)=\gamma^{-1}\alpha\gamma$ for all
  $\alpha\in\B$. Then by Part (3) of Lemma \ref{lemma: Runge},
  we have $\gamma^{-1}\mu_1\gamma=f(\mu_1)=\mu_2$ or
  $\gamma^{-1}\mu_1\gamma=-\mu_2$. By Lemma \ref{lemma: positive}, in
  the first case, we have $\gamma\in N_\B^+(\O)$. In the second
  case, we have $\nm(\gamma)<0$ and we write $\gamma=\omega\epsilon$
  for some $\omega\in N_\B^+(\O)$ and $\epsilon\in\O^\times$ with
  $\nm(\epsilon)=-1$. Then
  $\mu_2=\overline\epsilon\omega^{-1}\mu_1\omega\epsilon$. By
  Proposition \ref{proposition: fX1=fX2}, this implies that
  $\fX_{\mu_1}=\fX_{\mu_2}$, and the proof of the lemma is complete.
\end{proof}

The next lemma shows that conversely each $[Q]_\GL$ in
$\sQ_D/\GL(2,\Z)$ appears as the quadratic form associated to some
Shimura curve in $\QQ_D$.

\begin{Lemma} \label{lemma: sQ to QQ}
  Let $Q$ be an element in $\sQ_D$. Then there
  exists a Shimura curve $\fX$ in $\QQ_D$ whose associated
  $\GL(2,\Z)$-equivalence class of quadratic form is $[Q]_\GL$.
\end{Lemma}

\begin{proof}
  Let $ax^2+bxy+cy^2$ be a quadratic form in $\sQ_D$. Define a
  quaternion algebra $\B'=\Q+\Q I+\Q J+\Q IJ$ with the multiplication rule
  $$
    I^2=a, \quad J^2=c, \quad IJ+JI=b. 
  $$
  We check directly that $(IJ-JI)^2=b^2-4ac=-16D$ and
  $I(IJ-JI)=-(IJ-JI)I$. Thus, the quaternion algebra is isomorphic to
  $\JS{-16D,a}\Q$, which by assumption is isomorphic to $\B$.
  Set
  $$
    \alpha=\begin{cases}
    I/2, &\text{if }a\equiv 0\mod 4, \\
    (1+I)/2, &\text{if }a\equiv 1\mod 4, \end{cases} \qquad
    \beta=\begin{cases}
    J/2, &\text{if }c\equiv 0\mod 4, \\
    (1+J)/2, &\text{if }c\equiv 1\mod 4. \end{cases}
  $$
  By Lemma \ref{lemma: Runge}, $\O'=\Z+\Z\alpha+\Z\beta+\Z\alpha\beta$ is
  a maximal order in the quaternion algebra $\B'\simeq\B$. Set
  $\mu'=(IJ-JI)/4=\alpha\beta-\beta\alpha$, which is an element of
  trace $0$ and norm $D$ in $\O'$. Then $L_{\mu'}=\Z I+\Z J$ and
  $-\nm(xI+yJ)=ax^2+bxy+cy^2$. Since all maximal orders in an
  indefinite quaternion algebra over $\Q$ are isomorphic, this means
  that there exists an element $\mu$ of trace $0$ and norm $D$ in $\O$
  such that the quadratic form associated to $\fX_\mu$ (replacing
  $\mu$ by $-\mu$ if necessary, we may assume that $\mu\in\sD^+(\O)$)
  is in the class of $ax^2+bxy+cy^2$. This proves the lemma.
\end{proof}

\begin{Remark} \label{remark: Runge}
%  The proof of the lemma can also be used to show that if $\tau$ is a
%  point in $\H_2$ satisfying two singular relations $\ell_1$ and
%  $\ell_2$ with $(\gen{\ell_i,\ell_j}_\Delta)=\SM abbc$ for some
%  $Q(x,y)=ax^2+2bxy+cy^2\in\QQ_D$, then $\tau$ lie on the Shimura
%  curve $\fX$ whose quadratic form is $Q$. This slightly generalizes
%  Theorem 10 of \cite{Runge} in the case $Q\in\QQ_D$.
%
  Note that Theorem 10 of \cite{Runge} states that if $\SM
  abbc$ is a positive definite matrix in $M(2,\Z)$ satisfying
  $a,c\equiv 0,1\mod 4$, $4|(b^2-ac)$, $\gcd(a,b,c)=1$, and
  $(b^2-ac)/4$ is squarefree, then any pair of two primitive singular
  relations $\ell_1,\ell_2$ satisfying
  $(\gen{\ell_i,\ell_j}_\Delta)=\SM abbc$ defines the same locus in
  $\sA_2$. In Lemma \ref{lemma: primitive}, we find that when $D\equiv
  1,2\mod 4$, every quadratic form $ax^2+2bxy+cy^2$ in $\sQ_D$
  satisfies Runge's conditions and Lemma \ref{lemma: QQ to sQ}
  would follows immediately from Runge's theorem. However, when
  $D\equiv 3\mod 4$, there are quadratic forms $ax^2+2bxy+cy^2$ in
  $\sQ_D$ with $\gcd(a,b,c)=2$, which does not meet the conditions in
  Runge's theorem.
\end{Remark}

Summarizing the discussion, we have the following theorem.

\begin{Theorem} \label{theorem: correspondences}
  There are one-to-one correspondences between
  \begin{enumerate}
  \item[(1)] the set $\{\fX_\mu:\mu\in\sD^+(\O)\}$ of irreducible
    components in $\QQ_D$,
  \item[(2)] the set of orbits of $\wt\sD(\O)$ under the action of the
    Atkin-Lehner group, and
  \item[(3)] the set $\sQ_D/\GL(2,\Z)$.
  \end{enumerate}
  The correspondences between the three sets are given by
  $$
    \fX_\mu\longleftrightarrow\mu\longleftrightarrow[Q_\mu]_\GL.
  $$
  In particular, we have $r_D=|\sQ_D/\GL(2,\Z)|$.
\end{Theorem}

\begin{proof} The correspondence between the first set and the second
  set is the content of Proposition \ref{proposition: fX1=fX2}, while
  the correspondence between the first set and the third set is proved
  in Lemmas \ref{lemma: QQ to sQ} and \ref{lemma: sQ to QQ}.
\end{proof}

In view of this theorem, the problem of determining the number of
Shimura curves reduces to that of counting certain classes of
quadratic forms.

\section{Proof of Theorem \ref{theorem: main 1}}
\label{section: proof of Theorem 1}
In this section, we shall prove Theorem \ref{theorem: main 1}.
We first prove a lemma.

\begin{Lemma} \label{lemma: primitive}
  If $D\equiv 1,2\mod 4$, then every quadratic form in $\sQ_D$ is
  primitive. If $D\equiv 3\mod 4$, then every quadratic form in
  $\sQ_D$ is either primitive or is equal to $4Q'$ for some (primitive)
  quadratic form $Q'$ of discriminant $-D$.
\end{Lemma}

\begin{proof} If $Q(x,y)=ax^2+bxy+cy^2\in\sQ_D$ is not primitive,
  then we have $Q=2Q'$ for some quadratic form
  $Q'$ of discriminant $-4D$, i.e., $2|a,b,c$. However, because every
  integer represented by $Q$ is congruent to $0$ or $1$ modulo $4$, we
  must have $4|a,c$ and hence $4|b$. If $D\equiv 1,2\mod 4$, this
  cannot happen and therefore $Q$ is primitive. If $D\equiv 3\mod 4$,
  then the argument above shows that either $Q$ is primitive or is
  equal to $4Q'$ for some (primitive) quadratic form $Q'$ of
  discriminant $-D$.
\end{proof}

Now for a discriminant $d$, we let $C_d$ denote the class group of
primitive binary quadratic forms of discriminant $d$ and
$h(d)=|C_d|$ be the class number. Assume that the prime factorization
of $D$ is $D=p_1\ldots p_k$. We define characters
$\chi_{-4}$, and $\chi_{p_j}$ of $C_{-16D}$ by
$$
  \chi_{-4}:Q\mapsto\JS{-4}a, \qquad \chi_{p_j}:Q\mapsto\JS a{p_j},
  ~\quad p_j\text{ odd prime divisor of }D,
$$
where $a$ is any positive integer represented by $Q$ with $(a,2D)=1$.
For even $D$, we define additionally $\chi_2$ by
$$
  \chi_2:Q\mapsto\JS8a.
$$
Then these characters generate the group of genus characters of
$C_{-16D}$ with a single relation that the product of these characters
is the trivial character (see \cite[Theorem 3.15]{Cox}). It is not
difficult to verify that the characters $\chi_{p_j}$ can
also be defined by
$$
  \chi_{p_j}(Q)=(-D,a)_{p_j},
$$
where $(-D,a)_{p_j}$ is the Hilbert symbol and $a$ is any positive
integer represented by $Q$ (without the restriction that $(a,2D)=1$).
For the case $D\equiv 1,2\mod 4$, this means that the set $\sQ_D$,
which we recall is defined to be the set of positive definite binary
quadratic forms of discriminant $-16D$ satisfying
\begin{equation} \label{equation: Q conditions}
  a\equiv 0,1\mod 4, \qquad \JS{-D,a}\Q\simeq\B,
\end{equation}
actually forms a single genus in $C_{-16D}$.
%(In Lemma \ref{lemma:
%  primitive}, we have shown that quadratic forms in $\sQ_D$ are all
%primitive in the case $D\equiv 1,2\mod 4$.)
Likewise, in the case
$D\equiv 3\mod 4$, a similar argument and Lemma \ref{lemma: primitive}
imply that $\sQ_D$ is the disjoint union of one genus in $C_{-16D}$
and $\{4Q:~Q\in \mathcal G\}$ for some genus $\mathcal G$ in
$C_{-D}$. Now $C_{-16D}$ has $2^k$ genera, while in
the case $D\equiv 3\mod 4$, $C_{-D}$ has $2^{k-1}$ genera. From this,
we easily obtain the following formula for $|\sQ_D/\SL(2,\Z)|$.

\begin{Lemma} \label{lemma: mod SL}
Assume that the prime factorization of $D$ is $D=p_1\ldots p_k$.
Let $\mathcal G$ be the genus of $C_{-16D}$ characterized by the properties
$$
  \chi_{-4}(Q)=1, \qquad
  \chi_{p_j}(Q)=-1,
$$
for all $p_j$ and all $Q\in\mathcal G$. If $D\equiv 1,2\mod 4$, then
$$
  \sQ_D/\SL(2,\Z)=\mathcal G.
$$
In particular, we have
$$
  |\sQ_D/\SL(2,\Z)|=\frac1{2^k}h(-16D)=\frac1{2^{k-1}}h(-4D).
$$

Also, assume that $D\equiv 3\mod 4$. We let $\mathcal G'$ be the genus
of $C_{-D}$ such that 
$$
  \chi_{p_j}(Q)=-1
$$
for all $p_j$ and all $Q\in\mathcal G'$. Then
$$
  \sQ_D/\SL(2,\Z)=\mathcal G\cup\{4Q:~Q\in\mathcal G'\}.
$$
Consequently,
$$
  |\sQ_D/\SL(2,\Z)|=\frac1{2^k}h(-16D)+\frac1{2^{k-1}}h(-D)
 =\frac1{2^{k-1}}(h(-4D)+h(-D)).
$$
\end{Lemma}

Since $\SM1001$ and $\SM{-1}001$ form a complete set of coset
representatives of $\SL(2,\Z)$ in $\GL(2,\Z)$, to obtain a formula for
$r_D=|\sQ_D/\GL(2,\Z)|$, we should study how $\SM{-1}001$ acts on $\sQ_D$.

Recall that a quadratic form
$Q(x,y)$ is said to be \emph{ambiguous} if $Q(-x,y)$ is
$\SL(2,\Z)$-equivalent to itself. If a class contains an ambiguous
form, then every form in the class is ambiguous and we say the class
is ambiguous. In fact, since $Q(-x,y)$ is the inverse of $Q(x,y)$, the
ambiguous classes in $C_d$ are precisely the classes of order $1$ or
$2$ in $C_d$. Thus, the set $A_d$ of ambiguous classes in $C_d$
forms a subgroup of $C_d$ of order $2^r$, where $2^r$ is the number
of genera of $C_d$. In terms of reduced forms, a class is ambiguous
if and only if its reduced form is $ax^2+cy^2$, $ax^2+ax+cy^2$, or
$ax^2+bxy+ay^2$.

\begin{Notation}
  For a discriminant $d$, we let $A_d$ denote the set of ambiguous
  classes in $C_d$. Also, we let
  $$
    S_{-16D}=A_{-16D}\cap \sQ_D/\SL(2,\Z)
   =\{[Q]_\SL\in A_{-16D}: Q\in\sQ_D\}.
  $$
  For $D\equiv 3\mod 4$, also let
  $$
    S_{-D}=\{[Q]_\SL\in A_{-D}: 4Q\in \sQ_D\}.
  $$
\end{Notation}

From Lemma \ref{lemma: mod SL}, we easily obtain the following lemma.

\begin{Lemma} \label{lemma: rD mid-step}
  Let $k$ denote the number of prime divisors
  of $D$. Then we for $D\equiv 1,2\mod 4$, we have
  $$
%    \#\{\fX_\mu:\mu\in\sD^+(\O)\}
   r_D=\frac1{2^{k}}h(-4D)+\frac12|S_{-16D}|,
  $$
  and for $D\equiv 3\mod 4$, we have
  $$
%    \#\{\fX_\mu:\mu\in\sD^+(\O)\}
   r_D=\frac1{2^k}(h(-D)+h(-4D))
   +\frac12(|S_{-D}|+|S_{-16D}|).
  $$
\end{Lemma}

As the final preparation for our proof of Theorem \ref{theorem: main
  1}, we shall prove the following proposition, which relates the
notion of twisting elements and that of ambiguous classes.

\begin{Proposition} \label{proposition: ambiguous twisting}
  Let $\mu$ be an element in $\sD^+(\O)$. Then the
  quadratic form $Q_\mu$ is ambiguous if and only if $(\O,\mu)$ admits
  a twisting, that is, if and only if there exists an element $\chi$
  in $\O$ and a positive divisor $m$ of $D$ such that $\chi^2=m$ and
  $\mu\chi=-\chi\mu$.

  Moreover, if $Q_\mu$ is ambiguous and $m$ is as above, then $Q_\mu$
  primitively represents $m$ or $4m$ and also $D/m$ or $4D/m$.
\end{Proposition}

\begin{proof} Assume that there exists an element $\chi$ in $\O$ and a
  positive divisor $m$ of $D$ such that $\chi^2=m$ and
  $\mu\chi=-\chi\mu$. Such an element $\chi$ is in
  $\gen{1,\mu}^\perp$. Thus, the quadratic form $Q_\mu$ primitively
  represents
  $$
    \begin{cases}
    4m, &\text{if }m\equiv2,3\mod 4, \\
    m\text{ or }4m, &\text{if }m\equiv 1\mod 4. \end{cases}
  $$
  In other words, the quadratic form $Q_\mu$ is $\GL(2,\Z)$-equivalent
  to $4mx^2+bxy+cy^2$ or $mx^2+bxy+cy^2$ for some $b,c\in\Z$. It is
  easy to see that $(4m)|b$ in the first case and $(2m)|b$ in the
  second case. It follows that $Q_\mu$ is equivalent to $4mx^2+(D/m)y^2$
  or $4mx^2+4mxy+(D/m+m)y^2$ in the first case and equivalent to
  $mx^2+(4D/m)y^2$ in the second case. In either case, we
  find that $Q_\mu$ is an ambiguous quadratic form primitively
  representing $m$ or $4m$ and also $D/m$ or $4D/m$. (For the case
  $Q_\mu(x,y)=4mx^2+4mxy+(D/m+m)y^2$, a solution of $Q_\mu(x,y)=4D/m$
  is $(x,y)=(1,-2)$.)

  Conversely, assume that $Q_\mu$ is an ambiguous quadratic form,
  i.e., it is equivalent to $ax^2+cy^2$, $ax^2+axy+cy^2$, or
  $ax^2+bxy+ay^2$. Note that
  $ax^2+bx(x+y)+a(x+y)^2=(2a+b)x^2+(2a+b)xy+ay^2$. Therefore, we may
  assume that $Q_\mu$ is of the form $ax^2+cy^2$ or $ax^2+axy+cy^2$.
  In the first case, we find that $a|4D$. In the second case, we have
  $a(a-4c)=-16D$ and it is easy to see that we also have $a|4D$.
  In other words, the lattice $L_\mu=(\Z+2\O)\cap\gen{1,\mu}^\perp$
  contains an element $\alpha$ such that $\alpha^2=a$ with $a|4D$.
  If $a$ is odd, then we must have $a|D$ and hence $(\O,\mu)$ admits a
  twisting with $\chi=\alpha$ and $m=a$. If $4|a$, then we must have
  $\alpha\in2\O$ and hence $(\O,\mu)$ admits a twisting with
  $\chi=\alpha/2$ and $m=a/4$. This completes the proof of the
  proposition.
\end{proof}

Finally, let us prove Theorem \ref{theorem: main 1}.

\begin{proof}[Proof of Theorem \ref{theorem: main 1}]
  According to Theorem \ref{theorem: correspondences} and the
  proof of Proposition \ref{proposition: ambiguous twisting}, for each
  ambiguous class $[Q]$ in $\sQ_D$, there exists an element
  $\mu\in\sD^+(\O)$ and a positive divisor $m$ of $D$ with
  $\JS{-D,m}\Q\simeq\B$ such that $[Q]=[Q_\mu]$ and $Q_\mu$
  primitively represents
  $$
    \begin{cases}
    4m, &\text{if }m\equiv 2,3\mod 4, \\
    m\text{ or }4m, &\text{if }m\equiv 1\mod 4. \end{cases}
  $$
  Thus, to prove the theorem, we need to count for each positive
  divisor $m$ of $D$ satisfying $\JS{-D,m}\Q\simeq\B$, how many
  ambiguous classes in $\sQ_D$ primitively representing $4m$ or $m$
  (for the case $m\equiv 1\mod 4$) there are.

  Consider the case $2|D$ first. Let $m$ be a positive divisor
  of $D$ satisfying $\JS{-D,m}\Q\simeq\B$. Assume that $m\equiv
  1\mod 4$ first. If a primitive ambiguous quadratic form $Q$ of
  discriminant $-16D$ primitively represents $m$, then it must be
  equivalent to $mx^2+(4D/m)y^2$ (see the proof of Proposition
  \ref{proposition: ambiguous twisting}), which is clearly in
  $S_{-16D}$. If $Q$ primitively represents $4m$, then it is
  equivalent to $4mx^2+(D/m)y^2$ or $4mx^2+4mxy+(D/m+m)y^2$, but
  neither of them satisfies the first condition in \eqref{equation: Q
    conditions}. For the case $m\equiv 2\mod 4$, if $Q$ primitively
  representing $4m$, 
  then it is equivalent to $4mx^2+(D/m)y^2$ or
  $4mx^2+4mxy+(D/m+m)y^2$, exactly one of which is contained in
  $S_{-16D}$ since exactly one of $D/m$ and $D/m+m$ is congruent to
  $1$ modulo $4$. For the case $m\equiv 3\mod 4$, again, if $Q$
  primitively representing $4m$, then it is equivalent to
  $4mx^2+(D/m)y^2$ or $4mx^2+4mxy+(D/m+m)y^2$. The first one cannot be
  in $S_{-16D}$ since $D/m\equiv 2\mod 4$. The second one is in
  $S_{-16D}$. In summary, we find that for each positive 
  divisor $m$ of $D$ satisfying $\JS{-D,m}\Q\simeq\B$, there exists
  exactly one element $[Q_m]$ in $S_{-16D}$ primitively representing
  $4m$ or $m$. Now observe that the correspondence $m\to[Q_m]$ is
  $2$-to-$1$ as $[Q_m]=[Q_{D/m}]$, by Proposition \ref{proposition:
  ambiguous twisting}, and no other divisors $m'$ of $D$ will satisfy
  $[Q_{m'}]=[Q_m]$ (this can be explained in a purely algebraic term,
  but can also be seen from Propositions \ref{proposition: fX1=fX2} and
  \ref{proposition: Rotger Wmu}). Thus, 
  $$
    |S_{-16D}|=\frac12\#\left\{m|D:\JS{-D,m}\Q\simeq\B\right\},
  $$
  and hence by Lemma \ref{lemma: rD mid-step},
  $$
   r_D=\frac1{2^k}h(-4D)+\frac14\#\left\{m|D:\JS{-D,m}\Q\simeq\B\right\}.
  $$
  This proves the theorem for the case $2|D$.

  For the case $D\equiv 1\mod 4$, in order for a divisor $m$ of $D$ to
  satisfy $\JS{-D,m}\Q\simeq\B$, $m$ must be congruent to $1$ modulo
  $4$. For such an $m$, there are two elements $[mx^2+(4D/m)y^2]$ and
  $[4mx^2+(D/m)y^2]$ in $S_{-16D}$. For the case $D\equiv 3\mod 4$, if
  $m\equiv 1\mod 4$, then $[4mx^2+4mxy+(D/m+m)y^2]\in S_{-D}$ and
  $[mx^2+(4D/m)y^2]\in S_{-16D}$, and if $m\equiv 3\mod 4$, then
  $[4mx^2+4mxy+(D/m+m)y^2]\in S_{-D}$ and $[4mx^2+(D/m)y^2]\in
  S_{-16D}$. Thus, for the case of odd $D$, each divisor $m|D$
  satisfying $\JS{-D,m}\Q\simeq\B$ yields two elements $[Q_m]$ and
  $[Q_m']$ in $S_{-16D}$ or $S_{-D}$. Again, because
  $\{[Q_{D/m}],[Q_{D/m}']\}=\{[Q_m],[Q_m']\}$ and no other divisors
  of $D$ has this property, we find that, by Lemma \ref{lemma: rD
    mid-step},
  $$
    r_D=\frac1{2^k}h(-4D)+\frac12\#\left\{m|D:\JS{-D,m}\Q\simeq\B\right\},
  $$
  and the proof of the theorem is complete.
\end{proof}

\section{Modular parameterization of Shimura curves on $\sA_2$}
\label{section: method}
In this section, we will describe our method for finding modular
parameterizations of Shimura curves $\fX$ on $\sA_2$ and then give
some examples.

Let us first state our objective. Let all notations $D$, $\B$, $\O$,
$\phi$, $\sD^+(\O)$, $\mu$, $E_\mu$, $\rho_\mu$, $\fX_\mu$, etc. be
the same as before. Let $\{\alpha_1,\ldots,\alpha_4\}$ be a symplectic basis of
$\O$ with respect to $\mu$. Observe that if $\gamma\in\O^1$, then we
have
$$
  E_\mu(\phi(\alpha\gamma)v_z,\phi(\beta\gamma)v_z)=\tr(\mu^{-1}\alpha\gamma
  \ol\gamma\ol\beta)=\tr(\mu^{-1}\alpha\ol\beta)=E_\mu(\phi(\alpha)
  v_z,\phi(\beta)v_z).
$$
Thus, if we let $M_\gamma$ be the matrix in $M(4,\Z)$ such that
$$
  (\alpha_1\gamma,\ldots,\alpha_4\gamma)=(\alpha_1,\ldots,\alpha_4)M_\gamma,
$$
then $\gamma\mapsto M_\gamma^t$ defines an embedding of $\O^1$ into
$\Sp(4,\Z)$.

Now for $z\in\H_1$, set $\tau_1=(\phi(\alpha_1)v_z,\phi(\alpha_2)v_z)$,
$\tau_2=(\phi(\alpha_3)v_z,\phi(\alpha_4)v_z)$, and let
$\tau_z=\tau_2^{-1}\tau_1$ be a normalized period matrix for
$(A_z,\rho_\mu)$. Given a Siegel modular form $f$ of weight 
$k$ on $\Sp(4,\Z)$, define $\wt f:\H_1\to\C$ by
$$
  \wt f(z)=f(\tau_z).
$$
For $\gamma\in\O^1$, write $\phi(\gamma)=\SM abcd$ and $M_\gamma^t=\SM
ABCD$. We may check that $\det(C\tau_z+D)=(cz+d)^2$, and hence
$$
  \wt f(\phi(\gamma)z)=f(\tau_{\phi(\gamma)z})=f(M_\gamma^t\tau_z)
 =\det(C\tau_z+D)^kf(\tau_z)=(cz+d)^{2k}\wt f(z).
$$
In other words, the restriction of a Siegel modular form of weight $k$
along $\fX_\mu$ yields a modular form of weight $2k$ on $X_0^D(1)$. In
fact, we can further show that if $w_m\in W_\mu$, then $\wt f|w_m=\wt f$.

Now recall that the Igusa invariants $J_k$, $k=2,4,6,8,10$, are
weighted projective invariants of a genus $2$ curve introduced in
\cite{Igusa-1960} to study isomorphism classes of genus $2$
curves. When we consider curves over a field of characteristic
different from $2$, the invariant $J_8$ is redundant and we only need
to consider $J_2$, $J_4$, $J_6$, and $J_{10}$. Furthermore, in the
case of genus $2$ curves $C$ over $\C$, the values of $J_k$ are equal
to the values, in the weighted projective space, of certain
meromorphic Siegel modular forms at the 
moduli point in $\sA_2$ corresponding to the Jacobian of $C$ (see
\cite{Igusa-1967}). Now if $\fX_\mu$ is a Shimura curve in $\QQ_D$,
then $(J_k^{m/2}/J_m^{k/2})|_{\fX_\mu}$ are modular functions on
$X_0^D(1)/W_\mu$. When $X_0^D(1)/W_\mu$ has genus $0$ and $j$ is a
Hauptmodul on $X_0^D(1)/W_\mu$, they are rational functions in
$j$. Our goal in this section is to find those rational
functions. Equivalently, we wish to find polynomials $p_k(j)$ such
that
$[J_2,J_4,J_6,J_{10}]|_\fX=[p_2(j),p_4(j),p_6(j),p_{10}(j)]$. Such an 
expression will enable us to determine whether a curve
of genus $2$ has quaternionic multiplication by $\O$ just by
computing its Igusa invariants and checking whether they are equal to
$[p_2(j),p_4(j),p_6(j),p_{10}(j)]$ for some $j$.
Note that $J_2$, $J_4$, and $J_6$ are meromorphic Siegel modular
forms. In practice, it is more preferable to work with holomorphic
Siegel modular forms because the degrees of the rational functions
appearing in computation will be smaller.

Recall that Igusa \cite{Igusa-1967} showed that the ring of Siegel
modular forms of even weights on $\Sp(4,\Z)$ is generated by four
modular forms $\psi_4$, $\psi_6$, $\chi_{10}$, and $\chi_{12}$, where
$\psi_4$ and $\psi_6$ are Eisenstein series, and $\chi_{10}$ and
$\chi_{12}$ are certain cusp forms. The relations between these
modular forms and $J_k$ are
\begin{equation} \label{equation: psi to J}
\begin{split}
  J_2&=-3\frac{\chi_{12}}{\chi_{10}}, \\
  J_4&=\frac1{24}(J_2^2-\psi_4), \\
  J_6&=\frac1{216}(-J_2^3+36J_2J_4+\psi_6), \\
  J_{10}&=-4\chi_{10}.
\end{split}
\end{equation}
Moreover, from our computation, it appears that the following choice
\begin{equation} \label{equation: ss}
  s_2=\psi_4, \quad s_3=\psi_6, \quad s_5=2^{12}3^5\chi_{10}, \quad
  s_6=2^{12}3^6\chi_{12},
\end{equation}
of modular forms will make the final results simpler. These will be
the modular forms we will be using.

Before we describe our method, let us mention some ideas and
observations.
\begin{enumerate}
\item In addition to Shimura curves, there are many curves on $\sA_2$
  that are isomorphic to non-compact modular curves $Y_0(N)/W$ for
  some $N$ and some subgroup $W$ of the Atkin-Lehner group. For such
  curves, it is quite easy and straightforward to find their modular
  parameterization using Fourier expansions of modular
  functions. This will be explained in more details in Section
  \ref{section: modular curves}.
\item Suppose that $\fX$ and $\fX'$ are two different Shimura curves
  (or modular curves) on $\sA_2$. The intersection points of $\fX$ and
  $\fX'$, being moduli points with endomorphism algebras bigger than a
  quaternion algebra over $\Q$, must be CM-points. In order to
  determine which CM-points lie on the intersections, we shall study
  properties of singular relations satisfied by CM-points. This is
  done in Section \ref{section: singular relations}. We will see that
  to each CM-point we may associate a positive definite ternary
  quadratic form coming from its singular relations such that it is an
  intersection point of $\fX$ and $\fX'$ if and only if the ternary
  quadratic form primitively represents the quadratic forms
  associated to both $\fX$ and $\fX'$.
\item We employ two methods to find values of modular functions at
  CM-points. The first one is the method of Borcherds forms, and the
  second one is the method of explicit Shimizu liftings. We will
  describe the two methods in Section \ref{section: evaluations}. Note
  that the method of Borcherds forms only gives us the minimal
  polynomial of the value of a modular function at a CM-point. The
  method of explicit Shimizu liftings enables us to pin down the exact
  value.
\item The divisor of $\chi_{10}$ is $2H_1$, where $H_1$ is the Humbert
  surface of discriminant $1$. Thus, the divisor of $s_5|_\fX$ is a
  sum over CM-points whose ternary quadratic forms represent $1$ and
  the numerator of $(s_5/s_2s_3)|_\fX$ is a square.
\item There is a Siegel modular form of weight $60$ on $\Sp(4,\Z)$
  with divisor $2H_4$. This modular form is given in
  \cite[Page 849]{Igusa-1967}. In terms of our $s_k$, it is
\begin{equation} \label{equation: H4}
\begin{split}
&27000s_6^2s_5^3s_3-97200s_5^4s_3^2s_2^2-37125s_6s_5^4s_2^2-71928s_6^2s_5^2s_2^4\\
&\quad-48600s_6s_5^3s_3^3-33750s_5^5s_3s_2+48600s_6^3s_5^2s_2+34992s_6s_5^2s_2^7\\
&\quad+23328s_6^3s_3^2s_2^3-46656s_5^3s_3^3s_2^3+23328s_5^3s_3s_2^6+23328s_6^4s_3^2\\
&\quad+23328s_6^4s_2^3-11664s_6^3s_3^4-11664s_6^3s_2^6+3888s_5^4s_2^5+23328s_5^3s_3^5\\
&\quad-11664s_6^5-3125s_5^6-184680s_6s_5^3s_3s_2^3+34992s_6s_5^2s_3^4s_2\\
&\quad-69984s_6s_5^2s_3^2s_2^4+46656s_6^3s_5s_3s_2^2-68040s_6^2s_5^2s_3^2s_2.
\end{split}
\end{equation}
If we plug our formulas for $[s_2,s_3,s_5,s_6]|_\fX$ into this
expression, the result should be a square (except for factors
corresponding to elliptic points). This provides a test for the
correctness of our computation.
\end{enumerate}

Now let $\fX$ be an irreducible component of genus zero in $\QQ_D$ and $j$
be a Hauptmodul of its associated Shimura curve $X$. From the degree
formula for the divisor of a modular form on $X$, we can deduce 
an upperbound $n$ for the degree of the rational function $R(x)$ such
that $R(j)=(s_m^k/s_k^m)|_\fX$. Assume that
$R(x)=(a_nx^n+\cdots+a_0)/(b_nx^n+\cdots+b_0)$. Observe that any point
where the values of $j$ and $(s_m^k/s_k^{m})|_\fX$ are known will
give rise to a linear relation among the coefficients $a_j$ and $b_j$.
If we have sufficiently many such points, there is a good chance that
their corresponding linear relations will uniquely determine the
coefficients. To find such points, we consider intersection of
$\fX$ with other Shimura curves $\fX'$ whose modular parameterizations
are already known. Here let us give a simple example
demonstrating how to determine intersection of two Shimura curves
using quadratic forms.

\begin{Example} \label{example: D=6 1}
  Let $\fX_6$ be the unique Shimura curve in $\QQ_6$.
  The quadratic form associated to $\fX_6$ is $5x^2+2xy+5y^2$. Since
  $5x^2+2xy+5y^2$ is ambiguous, $\fX_6$ is isomorphic to
  $X_0^6(1)/W_6$. Let
  $\fY_1$ be the image of the map $Y_0(1)\mapsto\sA_2$ defined by
  $z\mapsto\SM z00z$ as in Example \ref{example: YN}. The quadratic form
  associated to $\mathfrak Y$ is $x^2+4y^2$. If $\tau$ lies on the
  intersection of $\fX_6$ and $\fY_1$, then it is necessarily a
  CM-point of some discriminant $d$ and satisfies a singular
  relation of discriminant $1$, in addition to the singular relations
  satisfied by $\fX_6$. Thus, the lattice of singular relations for
  $\tau$ must have a sublattice whose Gram matrix is equivalent to
  $$
    \begin{pmatrix}5&1&u\\1&5&v\\u&v&1\end{pmatrix}
  $$
  for some integers $u$ and $v$ satisfying the conditions that $u$ and
  $v$ are odd, by \eqref{equation: parity}, and the determinant of the
  matrix, by Lemma \ref{lemma: CM Gram}, is $4r^2|d|$ for some
  integers $r$. We find that the only $(u,v)$ satisfying these
  conditions are $(u,v)=\pm(1,1)$ with $(d,r)=(-4,1)$ and
  $(u,v)=\pm(1,-1)$ with $(d,r)=(-3,1)$. By Proposition
  \ref{proposition: existence of CM} below, these are indeed the Gram
  matrices for CM-points of discriminant $-3$ and $-4$. Furthermore,
  these matrices presents the quadratic form $x^2+4y^2$. Thus, the
  CM-points of discriminant $-3$ and $-4$ lie on the intersection of
  $\fX_6$ and $\fY_1$. Note that the divisor
  of $s_5$ is $2H_1$. Thus, the divisor of $s_5|_{\fX_6}$ should be a
  sum over CM-points of $\fX_6$ whose ternary quadratic forms
  represent $1$, i.e., a sum over the CM-points of discriminant $-3$
  and $-4$. Indeed, as a modular form of weight $k$ on $X_0^6(1)/W_6$
  has $k/24$ zeros, we have
  $$
    \div s_5|_{\fX_6}=\frac13P_{-3}+\frac12P_{-4},
  $$
  where $P_{-3}$ and $P_{-4}$ denote the CM-points of discriminant
  $-3$ and $-4$, respectively.
\end{Example}

In general, assume that $\fX\in\QQ_D$ and $\fX'\in\QQ_{D'}$ are two
different Shimura curves ($D$ and $D'$ need not be different). Suppose
that the quadratic forms associated to $\fX$ and $\fX'$ are
$Q(x,y)=ax^2+2bxy+cy^2$ and $Q'(x,y)=a'x^2+2b'xy+c'y^2$, respectively.
To determine the intersection of $\fX$ and $\fX'$, we let $n$ be the
smallest integer represented by $Q'$, but not by $Q$, and search for
matrices of the form
$$
  M=\begin{pmatrix}a&b&u\\b&c&v\\u&v&n\end{pmatrix}
$$
satisfying the conditions
\begin{enumerate}
\item $u\equiv an,v\equiv cn\mod 2$ (due to \eqref{equation: parity}),
\item $M$ is positive definite (which limits $u$ and $v$ to
a finite number of possibilities),
\item $M$ represents $Q'$, and
\item $\det M=4|d|$ for some negative discriminant $d$ such that
  $\Q(\sqrt d)$ can be embedded in both $\B_D$ and $\B_{D'}$.
\end{enumerate}
Then according to Proposition \ref{proposition: existence of CM}, some
of the CM-points of discriminant $d/r^2$ will lie on the intersection
of $\fX$ and $\fX'$. Note that in practice we may need to work out
explicit optimal embeddings to see how CM-points on $X_0^D(1)$ and
$X_0^{D'}(1)$ are mapped to points on $\fX$ and $\fX'$.

We now work out some examples of modular parameterizations of Shimura
curves.

\begin{Example} \label{example: fX6}
  This example is a continuation of Example \ref{example: D=6 1}.
  Let $\fX_6$ be the unique Shimura curve of discriminant $6$ in
  $\sA_2$. Its associated quadratic form is $5x^2+2xy+5y^2$. Let
  $j_6$ be the
  Hauptmodul of $X=X_0^6(1)/W_6$ characterised by the properties that it
  takes values $0$, $\infty$ and $-16/27$ at the CM-points $P_{-4}$,
  $P_{-3}$, $P_{-24}$ of discriminants $-4$, $-3$, and $-24$,
  respectively.

  A modular form of weight $k$ on $X$ has $k/24$ zeros in $X$.
  Thus,
  $$
    \div s_2=\frac13P_{-3}, \quad
    \div s_3=\frac12P_{-4}, \quad
    \div s_5=\frac13P_{-3}+\frac12P_{-4},
  $$
  and
  $$
    \frac{s_3^2}{s_2^3}=c_1j_6, \quad
    \frac{s_5}{s_2s_3}=c_2, \quad
    \frac{s_6}{s_2^3}=c_3j_6+c_4
  $$
  for some constants $c_1,\ldots,c_4$. To determine the constants,
  we consider the intersection of $\fX$ with modular curves.

  Let $\fY_1$ be the modular curve isomorphic to $Y_0(1)$ appearing in
  Example \ref{example: D=6 1}, in which we find that $\fX$ and
  $\fY_1$ intersect at the CM-points of discriminant $-3$ and $-4$.
  Its modular parameterization is
  $$
    [s_2,s_3,s_5,s_6]|_{\fY_1}=[j^2,j^2(j-1728),0,2^{12}3^6j^4],
  $$
  computed in Example \ref{example: X0(1)} below, where $j$ is the
  elliptic $j$-function. From this, we see that the value of
  $(s_6/s_2^3)|_{\fY_1}$ at the CM-point of discriminant $-4$ is
  $$
    \frac{2^{12}3^6j^4}{j^6}\Big|_{j=1728}=1.
  $$
  Therefore, the constant $c_4$ is $1$. To determine the other
  constants, we consider the intersection of $\fX$ with other modular
  curves.

  Let $\fY_5'$ be the curve defined by the image of the map
  $Y_0(5)/w_5\to\sA_2$ defined by $z\mapsto\SM{3z}{5z}{5z}{10z}$ as
  described in Example \ref{example: YN'}. Its modular
  parameterization is given by
  $$
    [s_2,s_3,s_5,s_6]|_{\fY'}=[(j_5-12)^2,~j_5(j_5^2-36j_5-432),~-2^{10}3^5j_5,
   ~2^{10}3^5(j_5^2+12j_5+12)],
  $$
  where
  $$
    j_5=\frac{\eta(z)^{6}}{\eta(5z)^{6}}+22+125
    \frac{\eta(5z)^{6}}{\eta(z)^{6}}
  $$
  Now the quadratic form associated to $\fY'$ is
  $4x^2+5y^2$. Considering matrices of the form
  $$
    M=\begin{pmatrix}5&1&u\\1&5&v\\u&v&4\end{pmatrix}
  $$
  satisfying the conditions that
  \begin{enumerate}
  \item $u,v\equiv 0\mod 2$,
  \item $M$ is positive definite, and
  \item $\det M=4|d|$ for some negative discriminant $d$ such that
    $\Q(\sqrt d)$ can be embedded in $\B_6$ and $\JS{d}5\neq -1$,
  \end{enumerate}
  we find that $\fX$ and $\fY_5'$ intersect at the CM-points of
  discriminant $-4$, $-19$, and $-24$. The values of $j_5$ at these
  CM-points are $0$, $36$, and $-6\pm6\sqrt{-3}$, respectively. Thus,
  the values of $(s_3^2/s_2^3)|_{\fY'_5}$, $(s_5/s_2s_3)|_{\fY'_5}$,
  and $(s_6/s_2^3)|_{\fY'_5}$ at the CM-point of discriminant
  $-19$ are $81/64$, $1$, and $145/64$, respectively. Since the value
  of $j_6$ at the CM-point of discriminant $-19$ is $81/64$, we find
  that $c_1=c_2=c_3=1$. Hence along $\fX$, we have
  $$
    [s_2,s_3,s_5,s_6]|_\fX=[j_6,~j_6^2,~j_6^3,~j_6^3(j_6+1)].
  $$
  We can deduce from \eqref{equation: psi to J} and \eqref{equation:
    ss} that our result agrees with \eqref{equation: Baba}.
\end{Example}

We now give an example where the Shimura curve $\fX$ is not isomorphic
to the quotient of $X_0^D(1)$ by the full Atkin-Lehner group.

\begin{Example} \label{example: 14}
  Let $D=14$ and $\fX$ be the unique Shimura curve in
  $\QQ_{14}$. Its quadratic form is $5x^2+4xy+12y^2$. Since the
  quadratic form is not ambiguous, $\fX$ is isomorphic to
  $X=X_0^{14}(1)/w_{14}$.

  As in Examples \ref{example: D=14 order} and \ref{example: D=14
    -11}, we represent $\B_{14}$ by $\JS{-14,5}\Q$ and let $\O$ be the
  maximal order spanned by
  $$
    e_1=1, \quad e_2=\frac{1+J}2, \quad e_3=\frac{I+IJ}2, \quad
    e_4=\frac{56J+IJ}5.
  $$
  Fix an embedding $\phi:\B_{14}\hookrightarrow M(2,R)$ by
  $$
    I\longmapsto\M0{-1}{14}0, \quad J\longmapsto\M{\sqrt 5}00{-\sqrt 5},
  $$
  and choose $\mu=I$. In Example \ref{example: D=14 -11}, we find that
  $\sqrt{-11}\mapsto I+(J+IJ)/5$ and $\sqrt{-11}\mapsto 2I+3J$ define
  the two inequivalent optimal embeddings of discriminant $-11$ into
  $\O$ modulo $\O^1+w_{14}$. We let $z_{-11}$ and $z_{-11}'$ be the
  CM-points corresponding to these two optimal embeddings,
  respectively. Let $j=j_{14}$ be the Hauptmodul of $X$ that takes
  values $0$, $\infty$, and $1$ at the CM-point $z_{-8}$ of
  discriminant $-8$, the CM-point of discriminant $-4$, and $z_{-11}$,
  respectively. Since a modular form of weight $k$ on
  $X_0^{14}(1)/w_{14}$ has $k/4$ zeros, we have
  $$
    \frac{s_3^2}{s_2^3}\Big|_\fX=c_1\frac{f_3(j)^2}{f_2(j)^3}, \quad
    \frac{s_5}{s_2s_3}\Big|_\fX=c_2\frac{f_5(j)}{f_2(j)f_3(j)},
    \quad
    \frac{s_6}{s_2^3}\Big|_\fX=c_3\frac{f_6(j)}{f_2(j)^3}
  $$
  for some polynomial $f_k$, $k=2,3,5,6$, with $\deg f_k\le k$ and
  some constants $c_1$, $c_2$, and $c_3$.

  Following the proof of Proposition \ref{proposition: existence of
    CM}, we replace the quadratic form $5x^2+4xy+12y^2$ by the
  equivalent form $5x^2+224xy+2520y^2$. Considering matrices of the
  form
  $$
    \begin{pmatrix}5&112&u\\112&2520&v\\u&v&n\end{pmatrix},
  $$
  we find the following intersection information.
  $$ \extrarowheight2pt
  \begin{array}{cc|ccc|cc} \hline\hline
  n & (u,v)  & d   & \text{optimal embedding} & j=j_{14} & \text{curve}
    & j'\\ \hline
  1 & (1,20) & -4  & [1,6/5,1/5]  & \infty & \fY_1 & 1728\\
    & (1,22) & -11 & [1,1/5,1/5]  & 1 & \fY_1 & -32768 \\
    & (1,24) & -8  & [1,4/5,-1/5] & 0 & \fY_1 & 8000 \\
  4 & (0,2)  & -51 & [2,1,0] & (2+\sqrt{-3})/3 & \fY_5' &
                                                          -12-48\sqrt{-3} \\
    &        &     & [3,1,1] & (2-\sqrt{-3})/3 &        &
                                                          -12+48\sqrt{-3} \\
    & (0,4)  & -36 & [2,2,0] & 2\sqrt{-3}/3 & \fY_5' & 30+22\sqrt{-3} \\
    &        &     & [3,2,1] & -2\sqrt{-3}/3 &       & 30-22\sqrt{-3} \\
  5 & (1,18) & -43 & [3,17/5,-3/5] & -5/9 & \fX_6 & 21^4/10^6 \\
    & (1,22) & -67 & [4,27/5,2/5] & -35/9 & \fX_6 & 231^4/20^6 \\
    & (1,26) & -51 & [3,19/5,-1/5] & (-2+\sqrt{-3})/3 & \fX_6 &
                                                                -7^4/12^3
    \\
    &        &     & [4,19/5,-6/5] & (-2-\sqrt{-3})/3 &       & \\ \hline\hline
%  8 & (0,6)  & -67 & [3,13/5,3/5] & 35/9  & \fX_{10} & 341^2/91^2 \\ \hline\hline
  \end{array}
  $$
  Here, for a CM-point of discriminant $d$, we let $[b_1,b_2,b_3]$
  denote the corresponding optimal embedding $\sqrt d\mapsto
  b_1I+b_2J+b_3IJ$. The values of $j$ at the CM-points are
  determined using the method of Borcherds forms and the method of
  explicit Shimizu liftings (see Example \ref{example: Shimizu 14}).
  The second to the last column indicates which Shimura curve the
  CM-point lies on. Here $\fY_1$, $\fY_5'$, and $\fX_6$
  are the modular curve or the Shimura curves in Examples
  \ref{example: X0(1)}, \ref{example: X0(5)}, and \ref{example: fX6},
  respectively. The Hauptmodul $j'$ for these 
  curves are specified as in the examples. Note that there is only one
  CM-point of discriminant $-51$ on $X_0^6(1)/W_6$. Thus, $\fX$ has a
  self-intersection at this point.

  Using modular parameterizations of $\fY_1$, $\fY_5'$, and $\fX_6$,
  it is straightforward to deduce from the above table that
  \begin{equation*}
  \begin{split}
    \frac{s_3^2}{s_2^3}\Big|_\fX
  &=\frac{49(27j^2+36j+14)}{(3j+5)^6}, \\
    \frac{s_5}{s_2s_3}\Big|_\fX
  &=\frac{243j^2(j-1)^2}{(3j+5)^2(27j^2+36j+14)}, \\
    \frac{s_6}{s_2^3}\Big|_\fX
  &=\frac{243(3j^6-6j^5+4j^4+4j^3+j^2-6j+3)}{(3j+5)^6}.
  \end{split}
  \end{equation*}
  Equivalently, we have
  \begin{equation*}
  \begin{split}
    [s_2,s_3,s_5,s_6]|_\fX
  &=[(3j+5)^2,~7(27j^2+36j+14),~243j^2(j-1)^2, \\
  & \qquad 243(3j^6-6j^5+4j^4+4j^3+j^2-6j+3)].
  \end{split}
  \end{equation*}
  Here we remark that the polynomial $f_5$ is easy to determine. Since
  $\div s_5=2H_1$, we have
  $$
    \div_Xs_5|_\fX=z_{-4}+2z_{-8}+2z_{-11}
  $$
  (note that $z_{-4}$ is an elliptic point of order $2$) and hence
  $$
    f_5(j)=j^2(j-1)^2.
  $$
  Note that if we substitute $s_k$ in \eqref{equation: H4} by the
  expressions above, we get
  $$
    3^{30}j^2(j-1)^2(j+1)^6(9j-5)^2(j^2-5)^2(3j^2+4)^2(9j^2-12j+7)^2(16j^4-13j^2+8),
  $$
  which is indeed a square, except for the factor $16j^4-13j^2+8$
  corresponding to the CM-points of discriminant $-56$, which are
  elliptic points of order $2$ on $X$.
\end{Example}

We now give an application of our modular parameterization.

\begin{Example}
  In \cite{Gonzalez-Guardia}, Gonz\'alez and Gu\`ardia provided a method
  to determine all the isomorphism classes of principal polarizations
  of the modular abelian surfaces $A_f$ with quaternionic multiplication
  attached a newform $f$ without complex multiplication. As an example
  of $A_f$ with two principal polarizations, they considered the
  newform
  $$
    f=q+\sqrt 7q^3-3q^5+4q^9-\sqrt 7q^{11}-4q^{13}-3\sqrt 7q^{15}+\cdots
  $$
  in $S_2(1568)$. They showed that the modular abelian surface $A_f$
  has quaternionic multiplication by a maximal order in $\B_{14}$ and
  has a principal polarization with which it becomes the Jacobian of
  the curve
  \begin{equation*}
  \begin{split}
   C:y^2&=\frac{1372-539i}5\Bigg(x^6+\frac{332+208i}{181}x^5
   +\frac{1173+2148i}{1267}x^4+\frac{376+8060i}{8869}x^3 \\
   &\qquad\qquad -\frac{705-1992i}{8869}x^2-\frac{1228-1612i}{62083}x
   -\frac{607-492i}{434581}\Bigg).
  \end{split}
  \end{equation*}
  The absolute Igusa invariants for $C$ are
  \begin{equation*}
  \begin{split}
    i_1(C)&=\frac{(1+i)^{14}(-7+8i)^5(28+5i)^5}{(2+i)^{12}}, \\
    i_2(C)&=\frac{(1+i)^{10}(3+10i)^2(7-8i)^3(28+5i)^3}{(2+i)^8}, \\
    i_3(C)&=\frac{(1+i)^{12}(-2+3i)(8+7i)^2(28+5i)^2(320+1383i)}{(2+i)^8}.
  \end{split}
  \end{equation*}
  To check that this curve indeed has quaternionic multiplication by a
  maximal order in $\B_{14}$, we note that in terms of $s_k$, the
  absolute Igusa invariants are
  $$
    i_1=\frac{2^{13}3^5s_6^5}{s_5^6}, \quad
    i_2=\frac{2^93^5s_2s_6^3}{s_5^4}, \quad
    i_3=\frac{2^73^4(4s_2s_6+s_3s_5)}{s_5^4}.
  $$
  By the previous example, these invariants along the unique Shimura
  curve $\fX$ in $\QQ_{14}$ are
  \begin{equation} \label{equation: Gonzalez}
  \begin{split}
    i_1|_\fX=\frac{2^{13}g(j)^5}{j^{12}(j-1)^{12}}, \quad
    i_2|_\fX=\frac{2^9g(j)^3(3j+5)^2}{j^8(j-1)^8}, \quad
    i_3|_\fX=\frac{2^7g(j)^2h(j)}{j^8(j-1)^8},
  \end{split}
  \end{equation}
  where $g(j)=3j^6-6j^5+4j^4+4j^3+j^2-6j+3$ and
  $h(j)=36j^8+48j^7-29j^6-34j^5+233j^4+120j^3-138j^2-80j+100$. Indeed,
  the values of $i_k|_\fX$ at $j=-i/2$ are equal to $i_k(C)$. We
  conclude that $C$ has quaternionic multiplication by a maximal order
  in $\B_{14}$.

  Note that the Mestre obstruction for $\fX$ is
  $\JS{-14,-3(16j^4-13j^2+8)}{\Q(j)}$ (see Appendix \ref{appendix:
    Mestre}). For the case $j=-i/2$, the Mestre obstruction is
  $\JS{-14,-147/4}{\Q(i)}\simeq M(2,\Q(i))$.
  This explains why a curve of genus $2$ over $\Q(i)$ with absolute
  Igusa invariants given by \eqref{equation: Gonzalez} exists.
\end{Example}

Finally, let us give an example of a principally polarized abelian
surface with quaternionic multiplication by a maximal order in $\B_D$,
but not in the quaternionic locus $\QQ_D$.

\begin{Example} \label{example: QM not on QD}
  Let $D=15$. There are two Shimura curves $\fX_{15}$ and $\fX_{15}'$ of
  discriminant $15$ on $\sA_2$ with quadratic forms $5x^2+12y^2$ and
  $8x^2+4xy+8y^2$, respectively.
% Both forms are ambiguous, so $\fX_{15}$
%  and $\fX_{15}'$ are isomorphic to $X_0^{15}(1)/W_{15}$.

  Using the standard recipe, we find that the CM-point of discriminant
  $-67$ of $\fX_{15}$ also lies on $\fX_6$, the unique Shimura curve
  in $\QQ_6$. The Gram matrix of its lattice of singular relations is
  $$
    \begin{pmatrix}5&0&1\\0&12&2\\1&2&5\end{pmatrix}.
  $$
  This Gram matrix does not represent the quadratic form
  $8x^2+4xy+8y^2$. Thus, the CM-point of discriminant $-67$ on
  $\fX_{15}'$ is different from that on $\fX_{15}$. However, as
  complex tori (without considering polarizations), the CM-point of
  discriminant $-67$ on $\fX_{15}$ is isomorphic to that on
  $\fX_{15}'$. This means that the CM-point of discriminant $-67$ on
  $\fX_{15}'$ has quaternionic multiplication by the maximal order in
  $\B_6$, but it does not belong to the quaternionic locus $\QQ_6$.
\end{Example}

\section{Modular curves on $\sA_2$}
\label{section: modular curves}
Let $Y_0(N)$ be the non-compact modular curve of level $N$. We can map
$Y_0(N)$ to $\sA_2$ by the same procedure as Shimura curves.

Let $\O_N=\SM\Z\Z{N\Z}\Z$ denote the standard Eichler order of level
$N$ in $M(2,\Q)$. For $z\in\H$, let 
$v_z=\left(\begin{smallmatrix}z\\1\end{smallmatrix}\right)$ and
$\Lambda_z=\O_Nv_z$. Choose  an element $\mu$ of trace $0$ and
determinant $N$ with a positive $(2,1)$-entry in $\O_N$. Then the
Riemann form $E_\mu:(\alpha v_z,\beta
v_z)\mapsto\tr(\mu^{-1}\alpha\ol\beta)$, $\alpha,\beta\in\O_N$ gives
rises to a principal polarization of $\C^2/\Lambda_z$. In this way, we
may map $Y_0(N)$ into $\sA_2$ holomorphically. When $N$ is large,
there are many inequivalent choices of $\mu$. Here we only consider
the following two types of $\mu$.
The resulting modular curves lie on the Humbert surfaces of
discriminant $1$ and $4$, respectively.

\begin{Example} \label{example: YN}
Choose $\mu=\SM0{-1}N0$. A symplectic basis with
respect to $\mu$ is
$$
  \alpha_1=\M1000, \quad \alpha_2=\M00N0, \quad \alpha_3=\M0100, \quad \alpha_4=\M0001.
$$
Then the normalized period matrix is $\SM z00{Nz}$. The lattice of
singular relations for non-CM-points is spanned by
$(0,1,0,0,0)$ and $(N,0,-1,0,0)$. Thus, the quadratic form associated
to this curve is $x^2+4Ny^2$.

Note that
$$
  \M{-1/Nz}00{-1/z}\sim_{\Sp}\M{Nz}00z\sim_\Sp\M z00{Nz}.
$$
Thus the map $Y_0(N)\to\sA_2$ defined by $z\mapsto\tau_z=\SM z00{Nz}$
factors through the Atkin-Lehner involution $w_N$.
\end{Example}

\begin{Example} \label{example: YN'}
In this example, we assume that $N$ is odd. We choose $\mu=\SM
N{-(N+1)/2}{2N}{-N}$. A symplectic basis for $\O_N$ is
$$
  \M{(N+1)/2}0N0, \quad \M N0{2N}0, \quad
  \M0100, \quad \M0001,
$$
so that the normalized period matrix is $\tau_z=\SM{(N+1)z/2}{Nz}{Nz}{2Nz}$.
The lattice of singular relations for non-CM-points is spanned by
$(0,2,-1,0,0)$ and $(N,-(N+1)/2,0,0,0)$.
Thus, the quadratic form associated to this curve is
$4x^2+2(N-1)xy+(N+1)^2y^2/4$, which is equivalent to
$$
  \begin{cases}
  4x^2+Ny^2, &\text{if }N\equiv 1\mod 4, \\
  4x^2+4xy+(N+1)y^2, &\text{if }N\equiv 3\mod 4.\end{cases}
$$
Note that if we change $z$ to $-1/Nz$ in the period matrix, we get
$$
  \M{-(N+1)/2Nz}{-1/z}{-1/z}{-2/z}\sim_\Sp
  \M{2Nz}{-Nz}{-Nz}{(N+1)z/2}\sim_\Sp\tau_z,
$$
where the first equivalence is given by the action of $\SM0{-1}10$.
Thus, the map $Y_0(N)\to\sA_2$ given by $z\mapsto\tau_z$ factors
through the Atkin-Lehner involution $w_N$.
\end{Example}

\begin{Notation} \label{notation: YN}
  We let $\fY_N$ and $\fY_N'$ denote the curves in the two examples
  above, respectively.
\end{Notation}

Now if $f$ is a Siegel modular form of weight $k$, then $f|_\fY$ or
$f|_{\fY'}$ is is a modular form of weight $2k$ on $X_0(N)/w_N$. For
the Siegel modular forms $\psi_4$, $\psi_6$, $\chi_{10}$ and
$\chi_{12}$ that generate the graded ring of Siegel 
modular forms of even weights on $\Sp(4,\Z)$, one can use Igusa's
formulas \cite[Page 848]{Igusa-1967} relating the four modular forms to
theta functions with characteristics to find $q$-expansions of the
four modular forms along $X_0(N)/w_N$. (The reader who has difficulty
understanding the formula for $\psi_6$ may find \cite{Streng}
helpful.) Then from the $q$-expansions,
it is straightforward to find the parameterization of the absolute
invariants in terms of modular functions on $X_0(N)/w_N$.

\begin{Example} \label{example: X0(1)}
  For $N=1$, computing $\psi_4$, $\psi_6$, $\chi_{10}$, $\chi_{12}$
  along $\fY_1$ using Igusa's formulas, we find
  \begin{equation*}
  \begin{split}
  \psi_4(\tau_z)&=E_4(z)^2, \\
  \psi_6(\tau_z)&=E_6(z)^2, \\
  \chi_{10}(\tau_z)&=0, \\
  \chi_{12}(\tau_z)&=\Delta(z)^2.
  \end{split}
  \end{equation*}
  Thus, we have
  $$
    [\psi_4,\psi_6,\chi_{10},\chi_{12}]
   =[j^2,j^2(j-1728),0,j^4],
  $$
  or equivalently,
  $$
    [s_2,s_3,s_5,s_6]=[j^2,j^2(j-1728),0,2^{12}3^6j^4]
  $$
  in the weighted projective space of weight $(2,3,5,6)$, where
  $j=j(z)$ is the elliptic $j$-function.
\end{Example}

\begin{Example} \label{example: X0(5)}
  Let $N=5$. Computing $\psi_4$, $\psi_6$, $\chi_{10}$, $\chi_{12}$
  along $\fY=\fY_5'$ using Igusa's formulas, we obtain
  \begin{equation*}
  \begin{split}
    \psi_4|_{\fY}&=(f_1-12f_2)^2, \\
    \psi_6|_{\fY}&=f_1(f_1^2-36f_1f_2-432f_2^2), \\
    \chi_{10}|_\fY&=-\frac14f_1f_2^4, \\
    \chi_{12}|_\fY&=\frac1{12}f_2^4(f_1^2+12f_1f_2+12f_2^2),
  \end{split}
  \end{equation*}
  where $f_1(z)=(E_2(z)-5E_2(5z))^2$, $f_2(z)=\eta(z)^6\eta(5z)^6$,
  and $E_2(z)$ is Eisenstein series of weight $2$ on $\SL(2,\Z)$.
  Thus, choosing the Hauptmodul $j_5$ for $X_0(5)/w_5$ to be
  $$
    j_5(z)=\frac{f_1(z)}{f_2(z)}=\frac{\eta(z)^6}{\eta(5z)^6}
   +22+125\frac{\eta(5z)^6}{\eta(z)^6},
  $$
  we have
  $$
    [\psi_4,\psi_6,\chi_{10},\chi_{12}]|_\fY
   =[(j_5-12)^2,~j_5(j_5^2-36j_5-432),~-j_5/4,~j_5^2/12+j_5+1],
  $$
  or equivalently,
  $$
    [s_2,s_3,s_5,s_6]|_\fY=[(j_5-12)^2,j_5(j_2^2-36j_5-432),
    -2^{10}3^5j_5^2,2^{10}3^5(j_5^2+12j_5+12)].
  $$
\end{Example}

\section{Singular relations satisfied by CM-points}
\label{section: singular relations}
In this section, we shall study singular relations satisfied by CM-points
on a Shimura curve $\fX$. All the notations such as $D$, $\B$, $\O$,
$\phi$, $\sD^+(\O)$, $E_\mu$, $\rho_\mu$, $\fX_\mu$, etc. carry the
same meaning as before.

\begin{Lemma} \label{lemma: explicit order}
  Let $Q(x,y)$ be a quadratic form in $\sQ_D$.
\begin{enumerate}
\item Assume that $Q$ is primitive. Let $p\nmid D$ be a prime
  (congruent to $1$ modulo $4$) represented by $Q$ so that
  $\JS{-D,p}\Q\simeq\B$. Let $s$ be an even integer such 
  that $s^2D+1\equiv 0\mod p$ and set $t=(s^2D+1)/p$. Then the
  $\Z$-module spanned by
  \begin{equation} \label{equation: basis for O}
    e_1=1, \quad e_2=\frac{1+J}2, \quad
    e_3=\frac{I+IJ}2, \quad e_4=\frac{sDJ+IJ}p
  \end{equation}
  is a maximal order in $\JS{-D,p}\Q$. Moreover, with the choice
  $\mu=I$, the quadratic form $Q_\mu$ is $px^2+4sDxy+4tDy^2$ (which
  is necessarily $\GL(2,\Z)$-equivalent to $Q$ since $p$ is a prime).
\item Assume that $D\equiv 3\mod 4$ and $Q=4Q'$ for some quadratic
  form $Q'$ of discriminant $-D$. Let $p$ be a prime congruent to $1$
  modulo $4$ represented by $Q'$ (so that
  $\JS{-D,p}\Q\simeq\B$). Let $s$ be an odd integer such that 
  $s^2D+1\equiv 0\mod 4p$ and set $t=(s^2+D)/4p$. Then the $\Z$-module
  spanned by
  $$
    e_1=1, \quad e_2=\frac{1+I}2, \quad e_3=J, \quad
    e_4=\frac{sDJ+IJ}{2p}
  $$
  is a maximal order in $\JS{-D,p}\Q$. Moreover, with the choice of
  $\mu=I$, the quadratic form $Q_\mu$ is $4px^2+4sDxy+4tDy^2$ (which
  is $\GL(2,\Z)$-equivalent to $Q$).
\end{enumerate}
\end{Lemma}

\begin{proof}
  Part (1) is proved in \cite[Theorems 2.2 and 5.1]{Hashimoto}.
  The proof of Part (2) is similar. It is straight to verify that the
  $\Z$-module $\O$ spanned by $e_1,\ldots,e_4$ is a maximal order and
  that $2J$ and $(sJ+IJ)/p$ form a basis for $\gen{1,\mu}^\perp\cap(\Z+2\O)$
  and hence the associated quadratic form is
  $$
    -\nm\left(2xJ+y\frac{sDJ+IJ}p\right)=4px^2+4sDxy+4tDy^2.
  $$
  This proves the lemma.
\end{proof}

\begin{Example} \label{example: D=14 order}
  Let $D=14$. There is a unique Shimura curve $\fX$ in $\QQ_{14}$. The
  associated quadratic form is $5x^2+4xy+12y^2$. The quadratic form
  clearly represents the prime $5$. The integer $s=4$ is
  an even solution of the congruence $s^2D+1\equiv 0\mod 5$. Thus, by
  Part (1) of the lemma, if we represent $\B_{14}$ by $\JS{-14,5}\Q$, then
  $$
    e_1=1, \quad e_2=\frac{1+J}2, \quad e_3=\frac{I+IJ}2, \quad e_4=\frac{56J+IJ}5
  $$
  span a maximal order in $\JS{-14,5}\Q$.
\end{Example}

\begin{Lemma} \label{lemma: CM Gram}
  Let $\fX_\mu$ be a Shimura curve in $\QQ_D$.
  Let $z$ be a CM-point of discriminant $d$ on $X_0^D(1)$ and
  $(A_z,\rho_\mu)$ be the corresponding principally
  polarized abelian surface in $\fX_\mu$. Let $\mathcal L_z$ be the
  lattice of singular relations satisfied by $(A_z,\rho_\mu)$. Then
  the rank of $\mathcal L_z$ is $3$, and
  $$
    \det(\gen{\ell_i,\ell_j}_\Delta)=4|d|
  $$
  for a basis $\{\ell_1,\ell_2,\ell_3\}$ of $\mathcal L_z$, where
  $\gen{\cdot,\cdot}_\Delta$ is the bilinear form in \eqref{equation:
    bilinear Delta}.
\end{Lemma}

\begin{proof}
  We first consider the case where the quadratic form $Q_\mu$
  associated to $\fX_\mu$ is primitive of discriminant $-16D$. Let
  $e_1,\ldots,e_4$ be given as in Part (1) of
  Lemma \ref{lemma: explicit order}. By the same lemma, we may assume
  that $\mu=I$. Then a symplectic basis for $\O$ with respect to
  $E_\mu$ is
  $$
    \alpha_1=e_3-\frac{p-1}2e_4,\quad
    \alpha_2=-sDe_1-e_4, \quad \alpha_3=e_1, \quad \alpha_4=e_2
  $$
  (see Page 537 of \cite{Hashimoto}). Let $\phi:\B_D\to M(2,\R)$ be
  the embedding defined by
  $$
    \phi(I)=\M0{-1}D0, \qquad \phi(J)=\M{\sqrt p}00{-\sqrt p}.
  $$
  Then the complex torus $\C^2/\phi(\O)v_z$ with principal
  polarization given by $\rho_\mu$ has a normalized period matrix
  $$
    \tau_z=\frac1{pz}
    \M{-\overline\epsilon^2+\frac{(p-1)sD}2z+D\epsilon^2z^2}
    {\overline\epsilon-(p-1)sDz-D\epsilon z^2}
    {\overline\epsilon-(p-1)sDz-D\epsilon z^2}
    {-1-2sDz+Dz^2},
  $$
  where $\epsilon=(1+\sqrt p)/2$ and $\ol\epsilon=(1-\sqrt p)/2$
  (Theorem 3.5 of \cite{Hashimoto}). Write $\tau_z$ as
  $\SM{\tau_1}{\tau_2}{\tau_2}{\tau_3}$. We have
  \begin{equation} \label{equation: 5 times 3}
    \begin{pmatrix}\tau_1\\\tau_2\\\tau_3\\\tau_2^2-\tau_1\tau_3
    \\1\end{pmatrix}
   =\frac1p\begin{pmatrix}
    \overline\epsilon^2 &{(p-1)sD}/2 & \epsilon^2 \\
   -\ol\epsilon &-(p-1)sD&-\epsilon \\
    1 & -2sD & 1 \\
    2sD\ol\epsilon &(p-1)s^2D^2+D & 2sD\epsilon \\
    0 & p & 0\end{pmatrix}
    \begin{pmatrix}-1/z \\ 1 \\ Dz\end{pmatrix}.
  \end{equation}
  Let
  $$
    \beta_1=2e_2-e_1=J, \quad \beta_2=e_3,\quad \beta_3=e_4,
  $$
  which form a basis for the $\Z$-module of elements of trace $0$ in
  $\O$.
  If $z$ is a CM-point of discriminant of $d$, then it is fixed by
  $\phi(\alpha)$ for some $\alpha=b_1\beta_1+b_2\beta_2+b_3\beta_3$
  with $\nm(\alpha)=|d|$ and
  $$
    \gcd(b_1,b_2,b_3)=\begin{cases}
    1, &\text{if }d\equiv 1\mod 4, \\
    2, &\text{if }d\equiv 0\mod 4.\end{cases}
  $$
% Note that if $d$ is odd, then $(1+\alpha)/2\in\O$, which implies that
%  $b_1$ is odd and $b_2$ and $b_3$ are even. 
  Write $\alpha$ as
  $\alpha=c_1I+c_2J+c_3IJ$. Changing $\alpha$ to $-\alpha$ if
  necessary, we may assume that $c_1+c_3\sqrt p>0$. Then we have
  $$
    z=\frac{c_2\sqrt p+\sqrt{d}}{D(c_1+c_3\sqrt p)}, \qquad
    -\frac1z=\frac{-c_2\sqrt p+\sqrt{d}}{c_1-c_3\sqrt p},
  $$
  and
  $$
    \begin{pmatrix}-1/z \\ 1 \\ Dz\end{pmatrix}
   =\begin{pmatrix}
    \gamma & \delta \\ 0 & 1 \\\ol\gamma & \ol\delta\end{pmatrix}
    \begin{pmatrix}\sqrt{d} \\ 1\end{pmatrix}, \qquad
    \gamma=\frac1{c_1-c_3\sqrt p}, \quad
    \delta=\frac{-c_2\sqrt p}{c_1-c_3\sqrt p}.
  $$
  Then from \eqref{equation: 5 times 3}, we have that
  $(a_1,\ldots,a_5)$ is a singular relation for
  $\SM{\tau_1}{\tau_2}{\tau_2}{\tau_3}$ if and only if it is in the
  nullspace of
  \begin{equation} \label{equation: 5 times 2}
    \frac1p\begin{pmatrix}
    \overline\epsilon^2 &{(p-1)sD}/2 & \epsilon^2 \\
   -\ol\epsilon &-(p-1)sD&-\epsilon \\
    1 & -2sD & 1 \\
    2sD\ol\epsilon &(p-1)s^2D^2+D & 2sD\epsilon \\
    0 & p & 0\end{pmatrix}
    \begin{pmatrix}
    \gamma & \delta \\ 0 & 1 \\\ol\gamma & \ol\delta\end{pmatrix}
  \end{equation}
  Noticing that $b_1,b_2,b_3$ and $c_1,c_2,c_3$ are related by
  $$
    c_1=\frac{b_2}2, \qquad
    c_2=b_1+\frac{sDb_3}p, \qquad
    c_3=\frac{b_2}2+\frac{b_3}p,
  $$
  we check directly that the nullspace of \eqref{equation: 5 times 2}
  contains
  $$
    (1,1,(1-p)/4,0,0), \quad
    (0,2sD,0,1,D(s^2D-t)), \quad
    (0,b_2,b_3+b_2(1-p)/2,0,2b_1).
  $$
  If $d\equiv1\mod 4$, then we must have $(1+\alpha)/2\in\O$, i.e.,
  $2|b_2,b_3$ and $2\nmid b_1$. As $\gcd(b_1,b_2,b_3)=1$, we find that
  \begin{equation} \label{equation: singular basis}
    (1,1,(1-p)/4,0,0), \quad
    (0,2sD,0,1,D(s^2D-t)), \quad
    (0,b_2/2,b_3/2+b_2(1-p)/4,0,b_1)
  \end{equation}
  form a basis for the $\Z$-module of the singular relations. If
  $d\equiv 0\mod 4$, then $\gcd(b_1,b_2,b_3)=2$. Moreover, $b_2/2$ and
  $b_3/2$ cannot both be even. (If $4|b_2,b_3$, then $b_1/2$
  is odd. This implies that $(1+\alpha/2)/2\in\O$, and the CM-point has
  an odd discriminant.) Therefore, the singular relations in
  \eqref{equation: singular basis} also form a $\Z$-basis for the
  lattice of singular relations. The Gram
  matrix with respect to this basis is
  \begin{equation} \label{equation: CM Gram}
    \begin{pmatrix}
    p & 2sD & -b_2p/2-b_3 \\
    2sD & 4tD & 2b_1-b_2sD \\
    -b_2p/2-b_3 & 2b_1-b_2sD & b_2^2/4 \end{pmatrix}.
  \end{equation}
  We check that its determinant is
  $$
    D((1-p)b_2^2-4b_2b_3-4b_3^2t-8sb_1b_3)-4pb_1^2,
  $$
  which is precisely $4\nm(\alpha)=4|d|$. This proves the lemma for
  the case $Q_\mu$ is primitive.

  The case where $D\equiv 3\mod 4$ and $Q_\mu=4Q'$ can be proved in
  the same way. Here we omit most of the details and only mention that
  if we let $\O$ be the maximal order in Part (2) of Lemma \ref{lemma:
    explicit order}, then with respect to the choice $\mu=I$, a
  symplectic basis is
  $$
    \alpha_1=\frac{1+I}2, \quad \alpha_2=-\frac{sDJ+IJ}{2p}, \quad
    \alpha_3=1, \quad \alpha_4=J.
  $$
  Then the normalized period matrix obtained using this basis is
  $$
    \tau_z=\frac1{4z}\M{Dz^2+2z-1}{-Dz^2/\sqrt p-1/\sqrt p}
    {-Dz^2/\sqrt p-1/\sqrt p}{Dz^2/p-2sDz/p-1/p}.
  $$
  For a non-CM-point $z$, the lattice of singular relations is spanned
  by
  $$
    (1,0,-p,0,-(1+sD)/2), \quad
    (0,0,(1-sD)/2,1,-tD).
  $$
  For a CM-point $z$ of discriminant $d$ whose corresponding optimal
  embedding is given by $\sqrt d\mapsto b_1J+b_2I+b_3(sDJ+IJ)/2p$, the
  lattice is spanned by the two singular relations above and an
  additional relation
  $$
    (0,-b_2,b_3/2,0,-b_1/2)
  $$
  ($b_1$ and $b_3$ are necessarily even and the parity of $b_2$
  depends on that of $d$). The Gram matrix with respect to this basis
  is
  $$
    \begin{pmatrix}4p&2sD&-b_3\\2sD&4tD&b_1\\-b_3&b_1&b_2^2\end{pmatrix},
  $$
  whose determinant is precisely $4|d|$. This completes the proof.
\end{proof}

Note that the proof of the lemma shows that if $\fX$ is a Shimura
curve with quadratic form $ax^2+2bxy+cy^2$, then the Gram matrix for a
CM-point on $\fX$ will be equivalent to a matrix of the form
\begin{equation} \label{equation: matrix form}
  \begin{pmatrix}a&b&u\\b&c&v\\u&v&n\end{pmatrix}
\end{equation}
for some integers $u$, $v$, $n$. In the following discussion, we will
consider the converse problem. That is, given a matrix of the form
\eqref{equation: matrix form}, does there exist a CM-point whose Gram
matrix is equivalent to the given matrix?

We recall that by \eqref{equation: parity}, the integers $u$, $v$, and
$n$ in \eqref{equation: matrix form} must satisfy $u\equiv an,v\equiv
cn\mod 2$ if it is the Gram matrix for some CM-point.

\begin{Proposition} \label{proposition: existence of CM}
  Let $\fX_\mu$ be a Shimura curve in $\QQ_D$ with
  quadratic form $Q_\mu(x,y)=ax^2+2bxy+cy^2$. Suppose that
  $$
    M=\begin{pmatrix}a&b&u\\b&c&v\\u&v&n\end{pmatrix}
  $$
  satisfy
  \begin{enumerate}
  \item $n\equiv 0,1\mod 4$, $u\equiv an\mod 2$, $v\equiv bn\mod
    2$, and
  \item $\det M=4|d|$ for some negative discriminant $d$ such that
    $\Q(\sqrt d)$ can be embedded in $\B_D$.
  \end{enumerate}
  Let $d_0$ be the discriminant of the quadratic field $\Q(\sqrt d)$
  and set $f=\sqrt{d/d_0}$. Then for each divisor positive $r$ of $f$
  with $(r,D)=1$, there exists a CM-point of discriminant 
  $r^2d_0$ on $\fX_\mu$ such that its lattice of singular relations
  contains a sublattice of index $f/r$ whose Gram matrix is equivalent
  to $M$. Moreover, if a CM-point of discriminant $r^2d_0$ on
  $\fX_\mu$ has this property, then any CM-point in the same
  $\Gal(L/K)$-orbit also has this property, where $K=\Q(\sqrt d)$ and $L$
  is the ring class field of the quadratic order of discriminant
  $r^2d_0$.
\end{Proposition}

In particular, when $M$ is the Gram matrix in \eqref{equation: CM Gram}
for the lattice of singular relations of a CM-point, the proposition
yields the following corollary.

\begin{Corollary} Let $d$ be a negative discriminant such that
  CM-points of discriminant $d$ exist on a Shimura curve
  $\fX\in\QQ_D$. Let $K=\Q(\sqrt d)$ and $L$ be the ring class field
  of the quadratic order of discriminant $d$. Then all points in a
  $\Gal(L/K)$-orbit of CM-points of discriminant $d$ have isomorphic
  lattices of singular relations.
\end{Corollary}

  Here we will prove Proposition \ref{proposition: existence of CM}
  only for the case where $Q_\mu$ is primitive; the proof of the other case
  $Q_\mu=4Q$ is similar. Before we go into the proof, let us briefly
  discuss CM-points and their corresponding optimal embeddings.

  Let $\phi:B\hookrightarrow M(2,\R)$ be a fixed embedding of $\B=\B_D$
  into $M(2,\R)$. Suppose that $K$ is an imaginary quadratic field that
  can be embedded into $\B$. Notice that if $\psi:K\hookrightarrow B$
  is an embedding, then so is $\ol\psi:a\to\ol{\psi(a)}$, where
  $\ol{\psi(a)}$ is the quaternionic conjugate of $\psi(a)$. The two
  embeddings $\psi$ and $\ol\psi$ have the same fixed point in the
  upper half-plane. To remove the ambiguity when we talk about the
  embedding corresponding to a CM-point, we observe that if $a\in
  K\backslash\Q$, then the eigenvalues of $\phi(\psi(a))$ are $a$ and
  $\ol a$. I.e., if $z_\psi$ is the CM-point fixed by $\phi(\psi(K))$,
  then
  $$
    \phi(\psi(a))\begin{pmatrix}z_\psi\\1\end{pmatrix}
   =a\begin{pmatrix}z_\psi\\1\end{pmatrix} \text{ or }
    \phi(\psi(a))\begin{pmatrix}z_\psi\\1\end{pmatrix}
   =\ol a\begin{pmatrix}z_\psi\\1\end{pmatrix}.
  $$

\begin{Definition}
  We say $\psi$ is \emph{normalized} (with respect to $\phi$) if the
  first case holds.
\end{Definition}

\begin{Remark} \label{remark: normalized embedding}
  Suppose that $\psi:\Q(\sqrt d)\hookrightarrow \B$ is a normalized embedding
  with respect to $\phi$ with corresponding CM-point $z_\psi$. Then
  from the relation
  $$
    \phi(\psi(\sqrt d))\begin{pmatrix}z_\psi\\1\end{pmatrix}
   =\sqrt d\begin{pmatrix}z_\psi\\1\end{pmatrix},
  $$
  we see that the $(2,1)$-entry of $\phi(\psi(\sqrt d))$ is
  necessarily positive. In other words, normalized embeddings $\psi$ with
  respect to $\phi$ are characterized by the property that the
  $(2,1)$-entry of $\phi(\psi(\sqrt d))$ is positive. We then can
  easily check that if $\psi$ is a normalized embedding, then for
  $\gamma\in\B^\times$, $\gamma\psi\gamma^{-1}$ 
  is a normalized embedding if and only if $\nm(\gamma)>0$ (cf. Lemma
  \ref{lemma: positive}).
\end{Remark}

\begin{Notation} Suppose that the normalized embedding for a CM-point
  $z$ of discriminant $d$ is $\psi$. We set
  $$
    \beta_z=\psi(\sqrt d).
  $$
\end{Notation}

  Now assume that $Q_\mu$ is a primitive quadratic form.
  Let $p$ be an odd prime presented by $Q_\mu$. By Part (1) of
  Lemma \ref{lemma: explicit order}, we may assume that the quaternion
  algebra $\B_D$ is $\JS{-D,p}\Q$, that the maximal order $\O$ is
  spanned by the four elements in \eqref{equation: basis for O}, that
  $\mu=I$, and that $Q_\mu=px^2+4sDxy+4ty^2$, where $s$ is an even
  integer such that $s^2D+1\equiv 0\mod p$ and $t=(s^2D+1)/p$. 

\begin{Lemma} \label{lemma: orbit condition}
  Let $p$, $s$, $t$, and $\O$ be as above. Let
  $e_1,\ldots,e_4$ be the basis for $\O$ given in \eqref{equation:
    basis for O} and
  $$
    \beta_1=2e_2-e_1=J, \quad \beta_2=e_3, \quad \beta_3=e_4.
  $$
  Let $z$ be a CM-point of discriminant $d$ on $X_0^D(1)$ with
  $\beta_z=b_1\beta_1+b_2\beta_2+b_3\beta_3$. We have the following
  properties.
  \begin{enumerate}
  \item We have
    $$
      b_1^2p\equiv d\mod 4D.
    $$
  \item Let $K=\Q(\sqrt d)$ and $L$ be the ring class field of the
    quadratic order of discriminant $d$. Let $z'$ be another CM-point
    of discriminant $d$ on $X_0^D(1)$ with
    $\beta_{z'}=b_1'\beta_1+b_2'\beta_2+b_3'\beta_3$. Then $z$ and
    $z'$ lie in the same $\Gal(L/K)$-orbit if and only if
    $$
      b_1\equiv b_1'\mod 2D.
    $$
  \item There is a one-to-one and onto correspondence between the set
    of $\Gal(L/K)$-orbits of CM-points of discriminant $d$ on
    $X_0^D(1)$ and the set
  $$
    \{r\mod 2D: pr^2\equiv d\mod 4D\}.
  $$ 
  \end{enumerate}
\end{Lemma}

\begin{proof} According to the definition $\beta_z$, we have
$$
  d=-\nm(\beta_z)=b_1^2p+D((p-1)b_2^2/4+b_2b_3+b_3^2t+2b_1b_3s).
$$
If $d\equiv 1\mod 4$, then $(1+\beta_z)/2$ must
be contained in $\O$. Thus, $b_1$ is odd and $b_2$ and $b_3$ are
even. If $d\equiv 0\mod 4$, then $2|b_j$ for all $j$. In either
case, we find that $b_1^2p\equiv d\mod 4D$. This proves Part (1).

To prove Part (2), we let $G$ be the class group of the quadratic
order of discriminant $d$, which acts on the set $\CM(d)$ of CM-points
of discriminant $d$. Through the Shimura reciprocity, we know that 
the $G$-orbits are the same as the $\Gal(L/K)$-orbits. Now recall that
two CM-points are in the same $G$-orbit if and only if the local
embeddings defined by $\beta_z$ and $\beta_{z'}$ are
$\O_q^\times$-equivalent at each place $q$, where $\O_q=\O\otimes\Z_q$
(see \cite[Chapter III]{Vigneras}). When $q\nmid D$, there is only one
equivalence class of local optimal embeddings. Thus, we only need to
consider the places dividing $2D$.

  Consider the case of odd prime $q|D$ first. In such a case,
  $1$, $I$, $J$, and $IJ$ form a $\Z_q$-basis for $\O_q$.
  A local optimal embedding is determined by the image of $\sqrt d$
  under the embedding. Let $c_1I+c_2J+c_3IJ$ be the image of $\sqrt d$.
  Observe that
  $$
    \nm(c_1I+c_2J+c_3IJ)=Dc_1^2-pc_2^2-pDc_3^2.
  $$
  Thus, $c_2$ is a solution of $x^2p\equiv d\mod q$. Now a direct
  computation shows that if $\gamma=d_1+d_2I+d_3J+d_4IJ\in\O_q$ and
  $c_1'I+c_2'J+c_3'IJ=\gamma(c_1I+c_2J+c_3IJ)\gamma^{-1}$, then
  \begin{equation} \label{equation: optimal lemma 1}
    c_2'-c_2=\frac{2D(c_1d_0d_3-c_3d_0d_1+c_1d_1d_2-c_2d_1^2-c_3d_2d_3p+c_2d_3^2p)}
    {\nm(\gamma)}
  \end{equation}
  Therefore, if $\gamma\in\O_q^\times$, then $c_2'\equiv c_2\mod
  q$. Since the number of local optimal embedding is either $1$ or $2$
  depending on whether $q|d$ or not, which coincides with the number of
  solutions of $x^2p\equiv d\mod q$, the equivalence class of a local
  optimal embedding is completely determined by the residue class of
  $c_2$ modulo $q$.

  Now assume that $z$ and $z'$ are two CM-points of discriminant $d$
  with $\beta_z$ and $\beta_{z'}$ given as in the statement. Notice
  that
  $b_1\beta_1+b_2\beta_2+b_3\beta_3=b_2I/2+(b_1+b_2sD/p)J+(b_2/2+b_3/p)IJ$.
  Therefore, for odd primes $q|D$, the local optimal embedding induced
  by $\beta_z$ and $\beta_{z'}$ are equivalent if and only if
  $b_1\equiv b_1'\mod q$.

  We next consider the case of $q=2$. If $2\nmid D$, then from Part
  (1), we know that $b_1\equiv d\equiv b_1'\mod 2$ and we are done.
  If $2|D$, then $1$, $(1+J)/2$, $(I+IJ)/2$, and $IJ$ form a
  $\Z_2$-basis for $\O_2$, and $J$, $(I+IJ)/2$, and $IJ$ form a
  $\Z_2$-basis for the set of elements of trace zero in $\O_2$. Assume
  that $d\equiv 1\mod 4$ and $\beta=c_1J+c_2(I+IJ)/2+c_3IJ$,
  $c_j\in\Z_2$, is the image of $\sqrt d$ under an optimal embedding
  of discriminant $d$ into $\O_2$. As before, $c_1$ must be odd,
  $c_2,c_3$ must be even, and by considering the equality
  $d=-\nm(\beta)$, we deduce that $c_1$ is a solution of $x^2p\equiv
  d\mod 8$. Now let $\gamma=d_0+d_1(1+J)/2+d_2(I+IJ)/2+d_3IJ$ be an
  element of $\O_2^\times$. A direct computation shows that if we write
  $\gamma\beta\gamma^{-1}$ as $c_1'J+c_2'(I+IJ)/2+c_3'IJ$,
  $c_j'\in\Z_2$, then
  \begin{equation*}
  \begin{split}
    c_1'-c_1&=\frac{D(p-1)}{4\nm(\gamma)}
    \left(-d_1(c_2+2c_3)(d_2+2d_3)+2c_1(d_2+2d_3)^2\right) \\
  &\qquad+\frac D{\nm(\gamma)}
   (2c_1d_3^2+c_2d_0d_3-c_3d_1d_2-c_3d_0d_2+2c_1d_2d_3-c_3d_1d_3)
  \end{split}
  \end{equation*}
  Since $2|c_2,c_3$, we must have $c_1'\equiv c_1\mod 4$. Thus, the
  equivalence class of the optimal embedding of $z$ is determined
  completely by the residue class of $c_1$ modulo $4$.

  When $d\equiv 0\mod 4$, there is only one equivalence class of
  optimal embedding of discriminant $d$ into $\O_2$. Also, a direct
  computation shows that if $c_1J+c_2(I+IJ)/2+c_3IJ$ is the image of
  $\sqrt d$ under an optimal embedding, then $2|c_1$ and $c_1/2\equiv
  d/4\mod 2$.

  Now assume that $z$ and $z'$ are two CM-points of discriminant $d$
  with $\beta_z$ and $\beta_{z'}$ given as in the statement. Notice
  that
  $b_1\beta_1+b_2\beta_2+b_3\beta_3=(b_1+b_3sD/p)J+b_2(I+J)/2+b_3IJ/p$.
  Therefore, the local optimal embedding into $\O_2$ defined by
  by $\beta_z$ and $\beta_{z'}$ are equivalent if and only if
  $b_1'\equiv b_1\mod 4$. This proves Part (2).

  Finally, to prove Part (3), we observe that
  $$
    \#\{r\mod 2D:pr^2\equiv d\mod 4D\}
   =\prod_{q|D}\left(1+\JS{pd}q\right)
   =\prod_{q|D}\left(1-\JS dq\right),
  $$
  which is the same as the number of $\Gal(L/K)$-orbits. Therefore,
  from Parts (1) and (2), we see that the map
  $\beta_z=b_1\beta_1+b_2\beta_2+b_3\beta_3\mapsto b_1\mod 2D$ defines
  a one-to-one and onto correspondence between the two sets. This
  completes the proof of the lemma.
\end{proof}

We are now ready to prove Proposition \ref{proposition: existence of
  CM} (for the case $Q_\mu$ is a primitive quadratic form).

\begin{proof}[Proof of Proposition \ref{proposition: existence of CM}]
  Let $p$, $s$, $t$, $\O$, $e_j$, and $\beta_j$ be given as in the
  above lemma. We may assume that $Q_\mu=px^2+4sDxy+4tDy^2$, so that
  $$
    M=\begin{pmatrix}p&2sD&u\\2sD&4tD&v\\u&v&n\end{pmatrix},
  $$
  where $u\equiv n,v\equiv 0\mod 2$. 

  Let $z$ be a CM-point of discriminant $r^2d_0$ with
  $\beta_z=b_1\beta_1+b_2\beta_2+b_3\beta_3$ and $\tau_z$ be the
  corresponding point on $\fX_\mu$. By \eqref{equation: CM
    Gram}, the Gram matrix of the lattice of singular relations of
  $\tau_z$ is equivalent to
  $$
   A=\begin{pmatrix}
    p & 2sD & -b_2p/2-b_3 \\
    2sD & 4tD & 2b_1-b_2sD \\
    -b_2p/2-b_3 & 2b_1-b_2sD & b_2^2/4 \end{pmatrix}.
  $$
  Consider the matrix
  $$
    U=\begin{pmatrix}1&0&0\\0&1&0\\x&y&r'\end{pmatrix}, \quad r'=f/r,
  $$
  with
  \begin{equation*}
  \begin{split}
    x&=t(u+r'(b_2p/2+b_3))-\frac s2(v-r'(2b_1-b_2sD)), \\
    y&=-\frac s2(u+r'(b_2p/2+b_3))+\frac p{4D}(v-r'(2b_1-b_2sD)).
  \end{split}
  \end{equation*}
  We have
  $$
    UAU^t=\begin{pmatrix}p&2tD&u\\2tD&4sD&v\\u&v&n'\end{pmatrix}
  $$
  for some rational number $n'$. Since
  $\det(UAU^t)=(r')^2(4r^2|d_0|)=\det M$, the number $n'$ has to be
  $n$ and we have $UAU^t=M$. Thus, to prove the proposition, one only
  needs to show that there exists a CM-point $z$ with
  $\beta_z=b_1\beta_1+b_2\beta_2+b_3\beta_3$ such that $x$ and $y$ are
  integers. Since $b_2$ and $b_3$ are always even and $s$ is assumed
  to be even, $x$ is necessarily an integer and $y\equiv
  p(v-2r'b_1)/4D\mod 1$. The proof reduces to showing
  that there exists a CM-point satisfying
  $$
    b_1r'\equiv(v/2)\mod 2D.
  $$

  Recall that by Lemma \ref{lemma: orbit condition}, $b_1\mod 2D$ runs
  through all possible solutions of the congruence equation
  \begin{equation} \label{equation: CM proposition temp}
    pb_1^2\equiv r^2d_0\mod 4D.
  \end{equation}
  On the other hand, we observe that
  \begin{equation} \label{equation: f2d0}
    4f^2|d_0|=\det M=-pv^2+4D(n+suv-tu^2).
  \end{equation}
  Assume that $q$ is an odd prime divisor of $D$. If $q|r'$, then
  $q|f$ and $q|v$. Thus, any solution $b_1$ of $pb_1^2\equiv
  r^2d_0\mod 4D$ automatically satisfies $b_1r'\equiv 0\equiv(v/2)\mod
  q$. If $q\nmid r'$, then $q\nmid f$, and $v$ satisfies
  $$
    p(v/2)^2\equiv f^2d_0\mod q.
  $$
  On the other hand, by \eqref{equation: CM proposition temp}, we have
  $$
    p(b_1r')^2\equiv r^2(r')^2d_0=f^2d_0\mod q
  $$
  Thus, there are $b_1$ such that $b_1r'\equiv v/2\mod q$.

  Now consider the case $q=2$. Since $t$ is odd, $s$ is even, and
  $n\equiv u\mod 2$, we have $2|(n+suv-tu^2)$. It follows that, by
  \eqref{equation: f2d0},
  $$
    \begin{cases}
    4|v, &\text{if }f^2d_0\text{ is even}, \\
    2\|v, &\text{if }f^2d_0\text{ is odd}. \end{cases}
  $$
  In fact, because $p$ is congruent to $1$ modulo $4$, the relation
  $s^2D+1=tp$ implies that $t\equiv 1\mod 4$ and we have
  \begin{equation} \label{equation: CM proposition temp 2}
     4|(n+4suv-tu^2).
  \end{equation}
  In the case $f^2d_0$ is odd, all $d_0$, $f$, $r$, $r'$, $v/2$, and
  $b_1$ are odd. Regardless of whether $2$ divides $D$ or not, we can
  always find $b_1$ satisfying \eqref{equation: CM proposition temp}
  such that $b_1r'\equiv(v/2)\mod 2^{1+v_2(D)}$, where $v_2(D)$ is the
  $2$-adic valuation of $D$. Finally, assume that $f^2d_0$ is even. If
  $D$ is odd, we only need $b_1r'\equiv v/2\mod 2$, which clearly
  holds since the condition $f^2d_0\equiv 0\mod 4$ implies that
  $b_1r'$ and $v/2$ are both even. (If $b_1$ is odd, then by
  \eqref{equation: CM proposition temp}, $r^2d_0$ is odd, but then
  $r'$ is even.) If $D$ is even, by \eqref{equation: f2d0} and
  \eqref{equation: CM proposition temp 2}, we have 
  $$
    f^2d_0\equiv p(v/2)^2 \mod 8.
  $$
  Therefore,
  $$
    v/2\equiv \begin{cases}
    0\mod 4, &\text{if }8|f^2d_0, \\
    2\mod 4, &\text{if }4\|f^2d_0. \end{cases}
  $$
  If $4\|f^2d_0$, then either $d_0\equiv 1\mod 4$ and $2\|r'$ or
  $4\|d_0$ and $2\nmid r'$. In either case, we find that $b_1r'\equiv
  2\equiv v/2\mod 4$. If $8|f^2d_0$, there are three cases to
  consider, namely, either $8|d_0$, or $4|d_0$ and $2|r'$, or
  $d_0\equiv 1\mod 4$ and $4|r'$. In each case, we can also verify
  that $b_1r'\equiv 0\equiv v/2\mod 4$. This completes the proof that
  there exists a CM-point such that $b_1r'\equiv(v/2)\mod 2D$ and the
  lemma follows.
\end{proof}

\begin{Remark} \label{remark: Atkin-Lehner}
  It can also be shown that if we identify the Galois
  orbits of CM-points of discriminant $d$ with the set $\{r\mod 2D:
  pr^2\equiv d\mod 4D\}$ as described in the lemma, then the action of
  the Atkin-Lehner involution $w_q$ for a prime divisor $q$ of $D$ on
  the set maps $r\mod 2D$ to the $r'\mod 2D$ with
  $$
    r'\equiv\begin{cases}
    -r\mod q', &\text{if }q'=q, \\
    r\mod q', &\text{if }q'\neq q, \end{cases}
  $$
  for a prime divisor $q'$ of $D$. In particular, the Atkin-Lehner
  involution $w_D$ maps $r\mod 2D$ to $-r\mod 2D$.
\end{Remark}

\begin{Example} \label{example: D=14 -11}
  Let $D=14$. Let $\fX$ with quadratic form
  $5x^2+4xy+12y^2$ be the unique Shimura curve in $\QQ_{14}$.
  Let $e_1,\ldots,e_4$ be the elements in
  Example \ref{example: D=14 order} that span a maximal order $\O$ in
  $\JS{-14,5}\Q$. Let
  $$
    \beta_1=2e_2-e_1=J, \quad \beta_2=e_3, \quad \beta_3=e_4
  $$
  be a basis for the set of elements of trace $0$ in $\O$. A quick
  search yields the following elements of trace $0$ and norm $11$ in
  $\O$:
  $$
    \pm 2I\pm 3J, \quad \pm 3I\pm 3J\pm IJ, \quad
    \pm I\pm\left(\frac{J+IJ}5\right),
    \pm 2I\pm\left(\frac{J-4IJ}5\right),\ldots.
  $$
  With respect to the choice
  $$
    \phi:~I\mapsto\M0{-1}{14}0, \quad J\mapsto\M{\sqrt 5}00{-\sqrt5},
  $$
  the ones that define normalized embeddings of $\Q(\sqrt{-11})$ into
  $\B_{14}$ are those having a positive coefficient for $I$. To
  determine which of them are equivalent to which of them, we express
  them in terms of $\beta_j$ and apply Lemma \ref{lemma: orbit
    condition}. Writing $[b_1,b_2,b_3]$ for
  $b_1\beta_1+b_2\beta_2+b_3\beta_3$, we find that the normalized
  optimal embeddings for the four CM-points of discriminant $-11$ on
  $X_0^{14}(1)$ are
  \begin{equation*}
  \begin{split}
    2I+3J=[115, 4, -10], \quad &2I-3J=[109, 4, -10], \\
    I+\frac{J+IJ}5=[45, 2, -4], \quad &I-\frac{J+IJ}5=[67, 2, -6].
  \end{split}
  \end{equation*}
  Also, according to Remark \ref{remark: Atkin-Lehner}, the actions of
  the Atkin-Lehner involutions $w_2$ and $w_7$ on $2I+3J$ are
  $$
    w_2:2I+3J\longmapsto I+\frac{J+IJ}5, \quad
    w_7:2I+3J\longmapsto I-\frac{J+IJ}5.
  $$

  Now consider the matrix
  $$
    M=\begin{pmatrix}5&2&1\\2&12&0\\1&0&1\end{pmatrix}
  $$
  of determinant $44$. According to Proposition \ref{proposition:
    existence of CM}, there is a CM-point of discriminant $-11$ such
  that the Gram matrix of its lattice of singular relations is
  equivalent to $M$. To find such a CM-point, we first replace the
  quadratic form $5x^2+4xy+12y^2$ by $5x^2+56sxy+56ty^2$ for some
  even integer $s$ satisfying $14s^2+1\equiv 0\mod 5$ and
  $t=(14s^2+1)/5$. Here we choose $s=4$ and $t=45$. Set
  $$
    U=\begin{pmatrix}1&0&0\\22&1&0\\0&0&1\end{pmatrix}
  $$
  and
  $$
    M'=UMU^t=\begin{pmatrix}5&112&1\\112&2520&22\\1&22&1\end{pmatrix}.
  $$
  According to the proof of the proposition, the CM-point $z$ that has a
  Gram matrix equivalent to $M'$ is the one with $b_1\equiv
  22/2=11\mod 28$. In this case, it is the CM-point corresponding to
  $I-(J+IJ)/5=67\beta_1+2\beta_2-6\beta_3$. Indeed, by \eqref{equation: CM
    Gram}, the Gram matrix for this CM-point is precisely $M'$.
\end{Example}

\section{Evaluations of modular functions}
\label{section: evaluations}
In this work, we employ two methods to determine values of modular
functions on Shimura curves at CM-points. The first method uses the
theory of Borcherds forms, and the second method uses explicit Shimizu
lifting. We briefly describe the two methods in this section.

\subsection{Method of Borcherds forms}
In general, Borcherds forms are modular forms on orthogonal groups of
signature $(n,2)$, $n\ge 1$, arising from singular theta
correspondence. For details, we refer the reader to
\cite{Borcherds,Borcherds-Duke,Bruinier} for the classical setting and
\cite{Errthum,Kudla,Schofer} for the adelic setting.

In the case of Shimura curve, the set $L$ of trace-zero elements in an
Eichler order $\O$ in an indefinite quaternion algebra $\B$ over $\Q$ forms
a lattice of signature $(1,2)$ under the trace form, and the subgroup
$O_L^+$ of orientation-preserving elements in the orthogonal group
$O_L$ is essentially $N_\B^+(\O)$, the normalizer of $\O$ in $\B$ with
positive norm. Thus, Borcherds forms in this setting become modular
forms on some subgroups of $N_\B^+(\O)$. See the exposition in Section 2.2 of
\cite{Guo-Yang} for more details. (See also \cite{Errthum}.)

Borcherds forms on Shimura curves, by themselves, are not easy to work
with. What makes them useful in practice is a formula of Schofer
\cite{Schofer}, which expresses the absolute value of the product of
the values of a given Borcherds form at CM-points of the same
discriminant $d$ as a sum involving derivatives of Fourier
coefficients of certain incoherent Eisenstein series. The sum is still
very complicated, but nonetheless can be computed using formulas of
Kudla, Rapoport, and Yang \cite{Kudla-Rapoport-Yang}. We refer the
reader to \cite{Errthum,Guo-Yang} for strategies and examples of the
computation.

In \cite{Guo-Yang}, using Borcherds forms and Schofer's formula, we
determined equations of all hyperelliptic Shimura curves $X_0^D(N)$,
$D>1$. In order to do so, we devised a systematic method to construct
Borcherds forms on Shimura curves. See Section 3 of \cite{Guo-Yang}
for details. Note that all Borcherds forms constructed using our
method are modular on $X_0^D(N)/W_{D,N}$, the quotient of
$X_0^D(N)$ by the full Atkin-Lehner group. However, in the present
work, we are often required to work with modular functions on
$X_0^D(1)/w_D$. For example, for $D=14$, there is a unique Shimura
curve $\fX$ of discriminant $14$ in $\sA_2$ with quadratic form
$5x^2+4xy+12y^2$. Since the quadratic form is not ambiguous, the curve
$\fX$ is isomorphic to $X_0^{14}(1)/w_{14}$. Using the method in
\cite{Guo-Yang}, we can construct a Borcherds form $t$ on
$X_0^{14}(1)/W_{14}$ that takes values $0$, $\infty$, and $1$ at the
CM-points of discriminant $-8$, $-4$, and $-11$, respectively. From
the ramification of the covering $X_0^{14}(1)/w_{14}\to
X_0^{14}(1)/W_{14}$, we know that a Hauptmodul for
$X_0^{14}(1)/w_{14}$ can be taken to be $s=t^{1/2}$. Now
$X_0^{14}(1)/w_{14}$ has two CM-points of discriminant $-11$ and the
values of $s$ at these two points are $\pm 1$. In practice, we need to
know which CM-point of discriminant $-11$ takes value $1$ and which
takes value $-1$. (For example, exactly one of the CM-points of
discriminant $-11$ lies on the intersection of $\fX$ and the Humbert
surface $H_1$. To determine modular parameterization of $\fX$, we need
to know what the value of $s$ at that point is.) Therefore, we will
also need the method of explicit Shimizu lifting described in the next
section.

\subsection{Method of explicit Shimizu lifting}

To describe the method of explicit Shimizu lifting, we first recall
the Jacquet-Langlands correspondence.

\begin{Proposition}[{\cite{Jacquet-Langlands,Shimizu}}]
  \label{proposition: Jacquet-Langlands}
  Let $D$ be discriminant of an indefinite quaternion algebra $\B$ over
  $\Q$. Let $N$ be a positive integer relatively prime to $D$. For an
  Eichler order $\O$ of level $N$ in $\B$ and a positive even
  integer $k$, let $S_k(\O)$ denote the space of modular forms on
  $\O$. Then
  $$
    S_k(\O)\simeq
    S_k^{D\text{-\rm{new}}}(DN):=\bigoplus_{d|N}\bigoplus_{m|N/d}
    S_k^\new(dD)^{[m]}
  $$
  as Hecke modules. Here
  $$
    S_k^\new(dD)^{[m]}=\{f(m\tau):~f(\tau)\in S_k^\new(dD)\}
  $$
  and $S_k^\new(dD)$ denotes the newform subspace of
  cusp forms of weight $k$ on $\Gamma_0(dD)$. In other words, for each
  Hecke eigenform $f(\tau)$ in $S_k^{D\text{-\rm{new}}}(DN)$, there
  corresponds a Hecke eigenform $f_\B(\tau)$ in
  $S_k(\O)$ that shares the same Hecke eigenvalues. Moreover, for
  a prime divisor $p$ of $D$, if the Atkin-Lehner involution $W_p$
  acts on $f$ by $W_pf=\epsilon_p f$, then
  $$
    W_p\widetilde f_\B=-\epsilon_p\widetilde f_\B.
  $$
\end{Proposition}

Building on earlier works of Shimizu \cite{Shimizu} and Watson
\cite{Watson}, Nelson \cite{Nelson} obtained an explicit formula for
values of modular forms on Shimura curves, which we describe now.

\begin{Definition}[{\cite[Definition 2.2]{Nelson}}]
  Let $D$, $N$, $\B$, and $\O$ be as above. Let $f=\sum_na_nq^n$
  and $f_\B$ be Hecke eigenform in $S_k(DN)$ and $S_k(\O)$,
  respectively, that form a pair under the Jacquet-Langlands
  correspondence. We say $f$ and $f'$ are \emph{compatibly normalized}
  if $a_1=1$ and
  $$
    \int_{\Gamma_0(DN)\backslash\H}y^k|f(z)|^2\frac{dxdy}{y^2}
   =\int_{\O^1\backslash\H}y^k|f_\B(z)|^2\frac{dxdy}{y^2}
  $$
\end{Definition}

From now on we assume that the level $N$ is squarefree. Under this
condition, it is known \cite{Atkin-Lehner} that the Fourier
coefficient $a_m$ in $f=\sum_na_nq^n$ has value $\pm m^{k/2-1}$ for
any divisor $m$ of $DN$. Thus, the definition
\begin{equation} \label{equation: cm}
  c_m:=\frac{\mu((m,N))}{m\cdot a_m}
\end{equation}
is meaningful, where $\mu(\cdot)$ is the M\"obius function. In
addition, let
$$
  \O^\vee=\{\alpha\in\B:\tr(\alpha\ol\beta)\in\Z\text{ for all }\beta\in\O\}
$$
be the dual lattice of $\O$. For a positive divisor $m$ of $DN$, we let
$$
  \O^{(m)}:=\{\alpha\in\O^\vee:\alpha\in\O\otimes\Z_q
  \text{ for all primes }q\text{ not dividing }m\}.
$$
In particular, we have $\O^{(1)}=\O$ and $\O^{(DN)}=\O^\vee$.

\begin{Remark} In the case of $N=1$, if we let $\O$ be the order
  given in Part (1) of Lemma \ref{lemma: explicit order}, and
  $e_1,\ldots,e_4$ be given by \eqref{equation: basis for O}, then
  $$
    e_1=1,\quad e_2=\frac{1+J}2,\quad \frac{e_3}D=\frac{I+IJ}D,\quad
    \frac{e_4}D=\frac{sDJ+IJ}{pD},
  $$
  form a basis for $\O^\vee$ and hence a basis for $\O^{(m)}$ is
  $\{e_1,e_2,e_3/m,e_4/m\}$.
\end{Remark}

Given an element $\alpha=\SM abcd$ in $M(2,\R)$, we let
\begin{equation*}
  X(\alpha)=\frac{a+d+i(b-c)}2.
% \quad
%  Y(\alpha)=\frac{a-d+i(b+c)}2, \quad
%  P(\alpha)=\frac{a^2+b^2+c^2+d^2}2.
\end{equation*}
Also, for a positive number $x$, define
$$
  W_k(x):=2^{1-k}\sum_{n=1}^\infty n^k(nxK_{k-1}(nx)-K_k(nx)),
$$
where $K$ is the $K$-Bessel function. For a subset $L$ of
$M(2,\R)$ and a positive number $t$, set
$$
  V_k(L,t)=\sum_{\alpha\in L,~\det(\alpha)=t}
  \left(\frac{X(\alpha)}{|X(\alpha)|}\right)^kW_k(4\pi|X(\alpha)|).
$$
Finally, for $z=x+iy\in\H$, set
$$
  \sigma_z=\M{y^{1/2}}{xy^{-1/2}}0{y^{-1/2}}.
$$
Then Nelson's formula states as follows.

\begin{Proposition}[{\cite[Theorem 3.1]{Nelson}}] Let
  $f(z)=\sum_na_nq^n$ and $f_\B$ be compatibly normalized
  Jacquet-Langlands pair of Hecke eigenforms of weight $k$. For
  $z_1=x_1+iy_1,z_2=x_2+iy_2\in\H$, we have
  $$
    (y_1y_2)^{k/2}\ol{f_\B(z_1)}f_\B(z_2)
  =\sum_{n=1}^\infty a_n\sum_{m|DN}c_mV_k\left(\sigma_{z_1}^{-1}\O^{(m)}
   \sigma_{z_2},n/m\right),
  $$
  where $c_m$ is defined by \eqref{equation: cm}.
\end{Proposition}

Note that in the definition of compatibly normalized Jacquet-Langlands
pair, $f_\B$ is determined only up to a scalar of absolute value
$1$. Nelson's formula is independent of this scalar. We now give some
examples demonstrating how to compute values of modular functions at
CM-points using Nelson's formula.

\begin{Example} \label{example: Shimizu 14}
  Consider the case $D=14$ and $N=1$. Let
  $X=X_0^{14}(1)$. Represent $\B=\B_{14}$ by
  $\JS{-14,5}\Q$ and let $\O$ be the maximal order spanned by
  $e_1,\ldots,e_4$ given in Example \ref{example: D=14 order}.
  Let $j$ be the Hauptmodul of $X/w_{14}$ that takes values
  $0$, $\infty$, and $1$ at the CM-point of discriminant $-8$, the
  CM-point of discriminant $-4$, and the CM-point of discriminant
  $-11$ corresponding to $I-(J+IJ)/5$ in Example \ref{example: D=14
    -11}. Our goal here is to evaluate $j$ at CM-points using Nelson's
  formula.

  Let
  $$
    f(z)=\eta(z)\eta(2z)\eta(7z)\eta(14z)
  =q-q^2-2q^3+q^4+2q^6+q^7-q^8+q^9+\cdots
  $$
  be the unique newform in $S_2(14)$ and $f_\B$ be a compatibly
  normalized eigenform on $X$ corresponding to $f$. We have
  $$
    f|w_2=f, \quad f|w_7=-f.
  $$
  Hence, by Proposition \ref{proposition: Jacquet-Langlands},
  $$
    f_\B|w_2=-f_\B, \quad f_\B|w_7=f_\B.
  $$
  In other words, $f_\B$ is a modular form on $X/w_7$. Being
  a modular form of weight $2$ on $X/w_7$, $f_\B$ has a zero
  at the elliptic point of order $2$, which is the CM-point of
  discriminant $-4$, on $X/w_7$. Since a modular form of
  weight $k$ on $X/w_7$ has $k/4$ zeros, $f$ has no other zeros.
  Consequently, as a
  modular form on $X$, the divisor of $f_\B$ is
  $$
    \mathrm{div}_Xf=\frac12z_{-4}+\frac12z_{-4}',
  $$
  where $z_{-4}$ and $z_{-4}'$ are the two CM-points of discriminant
  $-4$ on $X$.

  We next let $g$ be the Hecke eigenform in $S_6(14)$ with
  $$
    g(z)=q+4q^2+8q^3+16q^4+10q^5+32q^6-49q^7+64q^8+\cdots,
  $$
  and $g_\B$ be a compatibly normalized Hecke eigenform on
  $X_0^{14}(1)$ corresponding to $g$. We have
  $g|w_2=-g$ and $g|w_7=g$, and hence
  $$
    g_\B|w_2=g, \quad g_\B|w_7=-g_\B.
  $$
  Therefore, $g_\B$ is a modular form of weight $6$ on $X/w_2$, and
  hence has zeros at elliptic points order $2$ and $4$ on $X/w_2$,
  i.e., the CM-points of discriminant $-8$ and $-4$. Thus, as a
  modular form on $X$, the divisor of $g_\B$ is
  $$
    \mathrm{div}_Xg_\B=\frac12z_{-4}+\frac12z_{-4}'+z_{-8}+z_{-8}',
  $$
  where $z_{-8}$ and $z_{-8}'$ are the CM-points of discriminant $-8$
  on $X$. It follows that
  $$
    \mathrm{div}_X\frac{g_\B}{f_\B^3}=z_{-8}+z_{-8}'-z_{-4}-z_{-4}'.
  $$
  Note that $g_\B/f_\B^3$ is invariant under $w_{14}$. As a modular
  function on $X/w_{14}$, we have
  $$
    \mathrm{div}_{X/w_{14}}\frac{g_\B}{f_\B^3}=z_{-8}-z_{-4}.
  $$
  Therefore, $j=cg_\B/f_\B^3$ for some constant $c$. To determine $c$,
  we use Nelson's formula.

  Pick a point $z'$ in the upper half-plane that is not equivalent to
  $z_{-4}$ or $z_{-8}$ on $X/w_{14}$ and let
  $$
    h(z)=\frac{\ol g_\B(z')g_\B(z)}{\ol{f_\B(z')}^3f_\B(z)^3}.
  $$
  We evaluate $h$ at $z_{-11}$ using Nelson's formula. Then using the
  relation $j(z)=h(z_{-11})^{-1}h(z)$, we may obtain values of $j$ at
  other CM-points. For instance, we choose $z'=(18\sqrt
  5+\sqrt{-74})/154$, the fixed point of $11I+18J$, expand $f$ and $g$
  for $200$ terms, sums over
  $\alpha\in\sigma_{z'}^{-1}\O^{(m)}\sigma_{z_{-11}}$ satisfying
  $|X(\alpha)|^2\le 5000$ in $V_k$, and cut off the sum in $W_k(x)$
  when the individual term becomes less thant $10^{-16}$. We get
  $$
    h(z_{-11})=-117.75021053378 + 122.65507878410i.
  $$
  Then evaluating at the CM-point $z_{-11}'$ corresponding to $2I+3J$,
  we get
  $$
    h(z_{-11})^{-1}h(z_{-11}')=-0.99999999999964 - 1.1411319536611\cdot10^{-12}i,
  $$
  which differs from the expected value $-1$ only by about $10^{-12}$.
  Also, there are two CM-points $z_{-43}$ and $z_{-43}'$ of
  discriminant $-43$ on $X_0^{14}(1)/w_{14}$ corresponding to
  $2I-3J/5+2IJ/5$ and $3I+17J/5-3IJ/5$, respectively. We find that
  \begin{equation*}
  \begin{split}
    h(z_{-11})^{-1}h(z_{-43})&=0.55555554974119 + 7.1800012303484\cdot
    10^{-10}i, \\
    h(z_{-11})^{-1}h(z_{-43}')&=-0.55555555732276 - 2.5838571396425\cdot10^{-9}i.
  \end{split}
  \end{equation*}
  Using Borcherds forms and Schofer's formula, we know that
  $j(z_{-43})^2=25/81$ (see also \cite[Table 5]{Elkies}). Therefore,
  we have $j(z_{-43})=5/9$ and $j(z_{-43}')=-5/9$.
\end{Example}

%\begin{Example} Consider the case $D=26$ and $N=1$. There is a
%  Hauptmodul $j$ on $X_0^{26}(1)/w_{26}$ that takes values $0$, $\infty$,
%  and $1$ at the CM-point of discriminant $-52$, the CM-point of
%  discriminant $-8$, and one of the two CM-points of discriminant
%  $-11$. In this example, we should construct such a modular function
%  that is suitable for application of Nelson's formula.

%  Let $f$ and $g$ be Hecke eigenforms in $S_2(26)$ with Fourier
%  expansions
%  \begin{equation*}
%  \begin{split}
%    f(z)&=q - q^2 + q^3 + \cdots + q^{12} + q^{13} +\cdots, \\
%    g(z)&=q + q^2 - 3q^3 + \cdots -  3q^{12} - q^{13} + \cdots,
%  \end{split}
%  \end{equation*}
%  and $f_\B$ and $g_\B$ be Hecke eigenforms in $S(X_0^{26}(1))$.
%  We have
%  $$
%    f_\B|w_2=-f_\B, \quad f_\B|w_{13}=f_\B, \quad
%    g_\B|w_2=g_\B, \quad g_\B|w_{13}=-g_\B.
%  $$
%  Reasoning as in the previous example, we find that as modular forms
%  on $X=X_0^{26}(1)$, we have
%  $$
%    \mathrm{div}_Xf_\B=z_{-52}+z_{-52}', \quad
%    \mathrm{div}_Xg_\B=z_{-8}+z_{-8}'.
%  $$
%  Therefore, as a modular function on $X/w_{26}$,
%  $$
%    \mathrm{div}_{X/w_{26}}\frac{f_\B}{g_\B}=z_{-52}-z_{-8}.
%  $$
%  We may apply Nelson's formula on $f_\B/g_\B$ to get values of $j$ at CM-points.
%\end{Example}

We next give a more complicated example.

\begin{Example} Consider the case $D=39$ and $N=1$. Let
  $X=X_0^{39}(1)$. According to the arXiv version of \cite{Guo-Yang}
  (arXiv:1510.06193), there is a Hauptmodul $x$ on $X/w_{39}$ that
  takes values $0$ and $\infty$ at the two CM-points of discriminant
  $-7$ and values $\pm1$ at the two CM-points of $-19$. Also, an
  equation for $X_0^{39}(1)$ is
  $$
    y^2=-(x^4-x^3-x^2+x+1)(7x^4-23x^3+5x^2+23x+7),
  $$
  on which the Atkin-Lehner involutions are given by
  $$
    w_{13}:(x,y)\longmapsto\left(-\frac1x,\frac{y}{x^4}\right), \quad
    w_{39}:(x,y)\longmapsto(x,-y).
  $$
  (The same equation also appeared in \cite{Molina-isogeny}, which in
  turn was obtained from an earlier work of Molina
  \cite{Molina-hyperelliptic} by a simple change of variables.)
  Since the covering $X\to X/w_{13}$ ramifies at CM-points of
  discriminant $-52$, we find that $x$ takes values $\pm i$ at the two
  CM-points of discriminant $-52$. Here in this example we shall
  indicate how to use Nelson's formula to evaluate the modular
  function $x$.

  Let $f$ and $g$ be Hecke eigenforms in $S_2(39)$ with
  \begin{equation*}
  \begin{split}
    f(z)&=q + q^2 - q^3 - q^4 + \cdots + q^{12} + q^{13} + \cdots, \\
    g(z)&=q+(\sqrt 2-1)q^2+q^3+\cdots +(1-2\sqrt
    2)q^{12}-q^{13}+\cdots,
  \end{split}
  \end{equation*}
  and $g'$ be the Galois conjugate of $g$. Let $f_\B$, $g_\B$, and
  $g_\B'$ be the corresponding Hecke eigenforms in $S_2(X)$. As a
  modular form on $X$, the divisor of $f_\B$ is given by
  $$
    \mathrm{div}_Xf_\B=z_{-52}^{(1)}+z_{-52}^{(2)}+z_{-52}^{(3)}+z_{-52}^{(4)}.
  $$
  Here we assume that the action of the Atkin-Lehner involution $w_3$
  on the CM-points of discriminant $-52$ is given by
  $w_3z_{-52}^{(1)}=z_{-52}^{(3)}$ and
  $w_3z_{-52}^{(2)}=z_{-52}^{(4)}$. Now there are constants $c$ and
  $c'$, not both $0$, such that $h_\B=cg_\B+c'g_\B'$ vanishes at
  $z_{-52}^{(1)}$. We claim that $h_\B$ actually has a double zero at
  $z_{-52}^{(1)}$.

  Here we first consider $h_\B$ as a modular form on $X/w_3$. Assume that
  $z$ is another zero of $h_\B$ different from $z_{-52}^{(1)}$. Then since
  $h_\B|w_{13}=-h_\B$, $w_{13}z$ is also a zero of $h_\B$. However, as a
  modular form on $X/w_3$, $h_\B$ can have only two zeros. Therefore,
  $z$ must be a fixed point of $w_{13}$, i.e., 
  $z=z_{-52}^{(2)}$. But this implies that as a modular form on $X$,
  the divisor of $h_\B$ is
  $$
    \mathrm{div}_Xh_\B=z_{-52}^{(1)}+z_{-52}^{(2)}+z_{-52}^{(3)}+z_{-52}^{(4)},
  $$
  which coincides with that of $f_\B$, and consequently, $h_\B/f_\B$
  is a constant function, but this is a contradiction as we have
  $(h_\B/f_\B)|w_3=-h_\B/f_\B$. We conclude that $h_\B$ has no other
  zeros, i.e.,
  $$
    \mathrm{div}_Xh_\B=2z_{-52}^{(1)}+2z_{-52}^{(3)},
  $$
  and we have
  $$
    \div_{X/w_{39}}\frac{h_\B}{f_\B}=z_{-52}^{(1)}-z_{-52}^{(2)}.
  $$
  If we assume that $x(z_{-52}^{(1)})=i$ and $x(z_{-52}^{(2)})=-i$, then
  $$
    \frac{h_\B}{f_\B}=c_0\frac{x-i}{x+i}
  $$
  for some constant $c_0$. In practice, we use Nelson's formula to
  determine the ratio between $c$ and $c'$ in $h_\B=cg_\B+c'g_\B'$ and
  evaluate $h_\B/f_\B$ at a particular point to determine $c_0$. We
  then can apply Nelson's formula to find values of $x$ at other
  points.
\end{Example}
\clearpage
\appendix
\section{List of Shimura curves and their quadratic forms}
\label{appendix: list of Shimura curves}
Here we list all Shimura curves of discriminant $<100$ or of genus $0$
on $\sA_2$,
characterized by their quadratic forms. Here $D$ is the discriminant
of the quaternion algebra, $k$ is the numbering of the curves, $W$ is
the stable group of the Shimura curve, as defined in Definition
\ref{definition: stable}, and $g$ is the genus of the Shimura
curve. We let $[a,b,c]$ represent the quadratic form $ax^2+bxy+cy^2$.
The quadratic forms are enumerated first in the ascending
order of $b$ and then in the ascending order of $a$. For convenience
of the reader, we  also list the smallest $p$ and $s$ satisfying the
assumptions of Lemma \ref{lemma: explicit order}.
$$ \extrarowheight3pt
\begin{array}{|c|c|c|c|c|c|} \hline\hline
D & k & Q & p,s & |W| & g \\ \hline
6 & 1 & [ 5, 2, 5 ] & 5, 2 & 4 & 0 \\
\hline
10 & 1 & [ 5, 0, 8 ] & 13, 10 & 4 & 0 \\
\hline
14 & 1 & [ 5, 4, 12 ] & 5, 4 & 2 & 0 \\
\hline
15 & 1 & [ 5, 0, 12 ] & 17, 14 & 4 & 0 \\
   & 2 & [ 8, 4, 8 ] & 17, 3 & 4 & 0 \\
\hline
21 & 1 & [ 5, 2, 17 ] & 5, 2 & 2 & 0 \\
\hline
22 & 1 & [ 8, 8, 13 ] & 13, 6 & 4 & 0 \\
\hline
26 & 1 & [ 8, 0, 13 ] & 149, 130 & 4 & 0 \\
   & 2 & [ 5, 2, 21 ] & 5, 2 & 2 & 0 \\
\hline
33 & 1 & [ 8, 4, 17 ] & 17, 16 & 2 & 0 \\
\hline
34 & 1 & [ 5, 4, 28 ] & 5, 4 & 2 & 0 \\
\hline
35 & 1 & [ 5, 0, 28 ] & 73, 68 & 4 & 0 \\
   & 2 & [ 12, 4, 12 ] & 13, 7 & 4 & 0 \\
   & 3 & [ 12, 8, 13 ] & 13, 6 & 2 & 0 \\
\hline
38 & 1 & [ 12, 4, 13 ] & 13, 12 & 2 & 0 \\
   & 2 & [ 8, 8, 21 ] & 37, 6 & 4 & 0 \\
\hline
39 & 1 & [ 5, 4, 32 ] & 5, 4 & 2 & 0 \\
   & 2 & [ 8, 4, 20 ] & 5, 1 & 2 & 0 \\
\hline
46 & 1 & [ 5, 2, 37 ] & 5, 2 & 2 & 0 \\
\hline
51 & 1 & [ 12, 0, 17 ] & 29, 24 & 4 & 0 \\
   & 2 & [ 5, 2, 41 ] & 5, 2 & 2 & 0 \\
   & 3 & [ 12, 12, 20 ] & 5, 3 & 4 & 0 \\
\hline
55 & 1 & [ 13, 2, 17 ] & 13, 2 & 2 & 0 \\
   & 2 & [ 8, 4, 28 ] & 13, 11 & 2 & 0 \\
\hline
57 & 1 & [ 8, 4, 29 ] & 29, 28 & 2 & 1 \\
\hline
58 & 1 & [ 8, 0, 29 ] & 37, 28 & 4 & 0 \\
\hline
62 & 1 & [ 12, 4, 21 ] & 29, 6 & 2 & 0 \\
   & 2 & [ 13, 10, 21 ] & 13, 10 & 2 & 0 \\
\hline\hline
\end{array}
\qquad
\begin{array}{|c|c|c|c|c|c|} \hline\hline
D & k & Q & p,s & |W| & g \\ \hline
65 & 1 & [ 5, 0, 52 ] & 97, 26 & 4 & 0 \\
   & 2 & [ 13, 0, 20 ] & 137, 14 & 4 & 0 \\
   & 3 & [ 8, 4, 33 ] & 37, 12 & 2 & 1 \\
\hline
69 & 1 & [ 5, 4, 56 ] & 5, 4 & 2 & 0 \\
   & 2 & [ 17, 16, 20 ] & 17, 4 & 2 & 0 \\
\hline
74 & 1 & [ 8, 0, 37 ] & 109, 100 & 4 & 0 \\
   & 2 & [ 5, 4, 60 ] & 5, 4 & 2 & 0 \\
   & 3 & [ 13, 8, 24 ] & 13, 6 & 2 & 0 \\
\hline
77 & 1 & [ 13, 4, 24 ] & 13, 12 & 2 & 1 \\
   & 2 & [ 17, 14, 21 ] & 17, 10 & 2 & 1 \\
\hline
82 & 1 & [ 13, 12, 28 ] & 13, 4 & 2 & 1 \\
\hline
85 & 1 & [ 5, 0, 68 ] & 73, 58 & 4 & 0 \\
   & 2 & [ 17, 0, 20 ] & 37, 26 & 4 & 0 \\
\hline
86 & 1 & [ 5, 2, 69 ] & 5, 2 & 2 & 0 \\
   & 2 & [ 12, 4, 29 ] & 29, 28 & 2 & 0 \\
   & 3 & [ 8, 8, 45 ] & 61, 10 & 4 & 0 \\
\hline
87 & 1 & [ 12, 0, 29 ] & 41, 34 & 4 & 0 \\
   & 2 & [ 8, 4, 44 ] & 17, 5 & 2 & 0 \\
   & 3 & [ 17, 6, 21 ] & 17, 12 & 2 & 0 \\
   & 4 & [ 12, 12, 32 ] & 41, 7 & 4 & 0 \\
\hline
91 & 1 & [ 13, 0, 28 ] & 41, 38 & 4 & 1 \\
   & 2 & [ 5, 2, 73 ] & 5, 2 & 2 & 2 \\
   & 3 & [ 20, 12, 20 ] & 5, 3 & 4 & 1 \\
\hline
93 & 1 & [ 17, 12, 24 ] & 17, 6 & 2 & 2 \\
\hline
94 & 1 & [ 13, 2, 29 ] & 13, 2 & 2 & 0 \\
   & 2 & [ 5, 4, 76 ] & 5, 4 & 2 & 0 \\
\hline
95 & 1 & [ 8, 4, 48 ] & 13, 9 & 2 & 0 \\
   & 2 & [ 12, 4, 32 ] & 37, 9 & 2 & 0 \\
   & 3 & [ 12, 8, 33 ] & 37, 28 & 2 & 0 \\
   & 4 & [ 13, 12, 32 ] & 13, 4 & 2 & 0 \\
\hline\hline
\end{array}
$$
$$ \extrarowheight3pt
\begin{array}{|c|c|c|c|c|c|} \hline\hline
D & k & Q & p,s & |W| & g \\ \hline
106 & 1 & [ 8, 0, 53 ] & 61, 46 & 4 & 0 \\
\hline
111 & 1 & [ 5, 2, 89 ] & 5, 2 & 2 & 0 \\
   & 2 & [ 8, 4, 56 ] & 17, 7 & 2 & 0 \\
   & 3 & [ 20, 12, 24 ] & 5, 3 & 2 & 0 \\
   & 4 & [ 17, 14, 29 ] & 17, 10 & 2 & 0 \\
\hline
115 & 1 & [ 5, 0, 92 ] & 97, 78 & 4 & 0 \\
   & 2 & [ 20, 20, 28 ] & 17, 9 & 4 & 0 \\
\hline
118 & 1 & [ 8, 8, 61 ] & 61, 30 & 4 & 0 \\
\hline
119 & 1 & [ 17, 0, 28 ] & 181, 142 & 4 & 0 \\
   & 2 & [ 5, 4, 96 ] & 5, 4 & 2 & 0 \\
   & 3 & [ 12, 4, 40 ] & 41, 21 & 2 & 0 \\
   & 4 & [ 20, 4, 24 ] & 5, 1 & 2 & 0 \\
   & 5 & [ 12, 8, 41 ] & 41, 20 & 2 & 0 \\
   & 6 & [ 24, 20, 24 ] & 181, 39 & 4 & 0 \\
\hline
122 & 1 & [ 8, 0, 61 ] & 349, 160 & 4 & 0 \\
\hline
134 & 1 & [ 5, 4, 108 ] & 5, 4 & 2 & 0 \\
   & 2 & [ 12, 4, 45 ] & 53, 32 & 2 & 0 \\
   & 3 & [ 8, 8, 69 ] & 229, 140 & 4 & 0 \\
   & 4 & [ 13, 12, 44 ] & 13, 4 & 2 & 0 \\
\hline
143 & 1 & [ 13, 0, 44 ] & 293, 60 & 4 & 0 \\
   & 2 & [ 24, 4, 24 ] & 293, 233 & 4 & 0 \\
\hline
146 & 1 & [ 5, 2, 117 ] & 5, 2 & 2 & 0 \\
   & 2 & [ 13, 2, 45 ] & 13, 2 & 2 & 0 \\
   & 3 & [ 21, 4, 28 ] & 53, 46 & 2 & 0 \\
   & 4 & [ 21, 10, 29 ] & 29, 12 & 2 & 0 \\
\hline\hline
\end{array}
\qquad
\begin{array}{|c|c|c|c|c|c|} \hline\hline
D & k & Q & p,s & |W| & g \\ \hline
159 & 1 & [ 12, 0, 53 ] & 101, 42 & 4 & 0 \\
   & 2 & [ 5, 4, 128 ] & 5, 4 & 2 & 0 \\
   & 3 & [ 8, 4, 80 ] & 41, 19 & 2 & 0 \\
   & 4 & [ 20, 4, 32 ] & 5, 1 & 2 & 0 \\
   & 5 & [ 12, 12, 56 ] & 101, 59 & 4 & 0 \\
   & 6 & [ 21, 12, 32 ] & 41, 22 & 2 & 0 \\
\hline
166 & 1 & [ 8, 8, 85 ] & 101, 84 & 4 & 0 \\
\hline
194 & 1 & [ 21, 2, 37 ] & 37, 2 & 2 & 0 \\
   & 2 & [ 5, 4, 156 ] & 5, 4 & 2 & 0 \\
   & 3 & [ 13, 4, 60 ] & 13, 12 & 2 & 0 \\
   & 4 & [ 28, 12, 29 ] & 29, 10 & 2 & 0 \\
   & 5 & [ 21, 16, 40 ] & 149, 120 & 2 & 0 \\
\hline
202 & 1 & [ 8, 0, 101 ] & 109, 82 & 4 & 0 \\
\hline
206 & 1 & [ 5, 2, 165 ] & 5, 2 & 2 & 0 \\
   & 2 & [ 12, 4, 69 ] & 109, 10 & 2 & 0 \\
   & 3 & [ 21, 8, 40 ] & 53, 50 & 2 & 0 \\
   & 4 & [ 24, 16, 37 ] & 37, 28 & 2 & 0 \\
   & 5 & [ 21, 20, 44 ] & 157, 4 & 2 & 0 \\
\hline
215 & 1 & [ 5, 0, 172 ] & 577, 436 & 4 & 0 \\
   & 2 & [ 20, 20, 48 ] & 577, 141 & 4 & 0 \\
\hline
314 & 1 & [ 8, 0, 157 ] & 229, 210 & 4 & 0 \\
\hline\hline
\end{array}
$$
\newpage

\section{Parameterizations of Shimura curves}
\label{appendix: parameterizations}
In this section, we list modular parameterizations for all Shimura
curves of genus zero on $\sA_2$. In addition, we also give
parameterizations for $\fY_N$ and $\fY_N'$ for the first few $N$,
where $\fY_N$ and $\fY_N'$ are the modular curves introduced in
Notation \ref{notation: YN}. Note that the curves $\fY_N$ lie on the
Humber surface of discriminant $1$ and hence vanish on $s_5$. The
choice of Hauptmoduls for $X_0(N)/w_N$ is given in Tabel
\ref{table: Hauptmoduls for modular curves}, while those for
$X_0^D(1)/W_D$ and $X_0^D(1)/w_D$ are given in 
\ref{table: Hauptmoduls for Shimura curves 1} and
\ref{table: Hauptmoduls for Shimura curves 2}, respectively. Note that
the description of Hauptmoduls in the case of Shimura curves is given
by specifying the values of the Hauptmodul at three certain
CM-points $z_{d_1}$, $z_{d_2}$, and $z_{d_3}$. In the case of
$X_0^D(1)/W_D$, this uniquely determines the Hauptmodul since for each
$d_i$, there is only one CM-point of discriminant $d_i$. In the case
of $X_0^D(1)/w_D$, we have two different CM-points of discriminant
$d_3$. Thus, there are two possible choices of Hauptmoduls. In the
table, we also describe how these two choices are related.

\begin{table}[!htbp] \extrarowheight2pt
%\begin{footnotesize}
\begin{tabular}{|c|l|}
\hline\hline
$N$ & $\fY_N$ \\ \hline
$1$ & $s_2=j^{2}$ \\ 
    & $s_3=j^{2}(j-1728)$ \\ 
    & $s_6=2^{12}3^{6}j^{4}$ \\ \hline
$2$ & $s_2=j(j+144)$ \\
    & $s_3=j^2(j-648)$ \\
    & $s_6=2^{12}3^6j^3$ \\ \hline
$3$ & $s_2=j(j+192)$ \\
    & $s_3=j(j^2-576j-1728)$ \\
    & $s_6=2^{12}3^6j^2$ \\ \hline
$4$ & $s_2=j^2+208j+16$ \\
    & $s_3=(j+8)(j^2-560j-8)$ \\
    & $s_6=2^{12}3^6j$ \\ \hline
$5$ & $s_2=j^2+216j+144$ \\
    & $s_3=j(j^2-540j-6480)$ \\
    & $s_6=2^{12}3^6$ \\ \hline
$6$ & $s_2=(j+1)(j^3 + 219j^2 - 21j + 1)$ \\
    & $s_3=(j^2+14j+1)(j^4 - 548j^3 + 198j^2 - 44j + 1)$ \\
    & $s_6=2^{12}3^6j^5$ \\ \hline
$7$ & $s_2=(j-1)^2(j^2+222j+225)$ \\
    & $s_3=(j-1)^2(j^4-532j^3-7434j^2+11340j-5103)$ \\
    & $s_6=2^{12}3^6(j-1)^4$ \\
\hline\hline 
\end{tabular}
\medskip
\tabcaption{Modular parameterization for $\fY_{N}$.}
%\end{footnotesize}
\end{table}

\begin{table}[!htbp] \extrarowheight2pt
%\begin{footnotesize}
\begin{tabular}{|c|l|}
\hline\hline
$N$ & $\fY_N'$ \\ \hline
$3$ & $s_2=j(j-48)$ \\
    & $s_3=j(j^2-72j-1728)$ \\
    & $s_5=-2^{10}3^5j^2$ \\
    & $s_6=2^{10}3^5j^2(j+12)$ \\ \hline
$5$ & $s_2=(j-12)^2$ \\
    & $s_3=j(j^2-36j-432)$ \\
    & $s_5=-2^{10}3^5j$ \\
    & $s_6=2^{10}3^5(j^2+12j+12)$ \\ \hline
$7$ & $s_2=(j-1)^2(j^2-18j+225)$ \\
    & $s_3=(j-1)^2(j^4-28j^3-378j^2+3780j-5103)$ \\
    & $s_5=-2^{10}3^5j(j-1)^4$ \\
    & $s_6=2^{10}3^5(j-1)^4(j^3+14j^2-15j+12)$ \\
\hline\hline 
\end{tabular}
\medskip
\tabcaption{Modular parameterization for $\fY_{N}'$.}
%\end{footnotesize}
\end{table}

\begin{table}[!htbp] \extrarowheight3pt 
%\begin{footnotesize}
\begin{tabular}{|c|l|}
\hline\hline
$N$ & Hauptmodul \\ \hline
$1$ & $\displaystyle\frac{E_4(z)^3}{\eta(z)^{24}}$ \\ \hline
$2$ & $\displaystyle\frac{\eta(z)^{24}}{\eta(2z)^{24}}+128
      +4096\frac{\eta(2z)^{24}}{\eta(z)^{24}}$ \\ \hline
$3$ & $\displaystyle\frac{\eta(z)^{12}}{\eta(3z)^{12}}+54+
      729\frac{\eta(3z)^{12}}{\eta(z)^{12}}$ \\ \hline
$4$ & $\displaystyle\frac{\eta(2z)^{48}}{\eta(z)^{24}\eta(4z)^{24}}$
  \\ \hline
$5$ & $\displaystyle\frac{\eta(z)^6}{\eta(5z)^6}+22+125
       \frac{\eta(5z)^6}{\eta(z)^6}$ \\ \hline
$6$ & $\displaystyle\frac{\eta(2z)^{12}\eta(3z)^{12}}{\eta(z)^{12}\eta(6z)^{12}}$
      \\ \hline
$7$ & $\displaystyle\frac{\eta(z)^4}{\eta(7z)^4}+14+49\frac{\eta(7z)^4}{\eta(z)^4}$
      \\ 
\hline\hline 
\end{tabular}
\medskip 
\tabcaption{Hauptmoduls for for $X_0(N)/w_N$.}
\label{table: Hauptmoduls for modular curves}
\end{table}

\clearpage

\begin{table} \extrarowheight3pt
\begin{footnotesize}
\begin{tabular}{|c|l|}
\hline\hline
\multicolumn{2}{|c|} {$\fX_{6}^1$ $\quad[ 5, 2, 5 ]$} \\ \hline
$s_2$
& $j$ \\
\hline
$s_3$
& $j^2$ \\
\hline
$s_5$
& $j^3$ \\
\hline
$s_6$
& $j^3(j+1)$ \\
\hline\hline
\multicolumn{2}{|c|} {$\fX_{10}^1$ $\quad[ 5, 0, 8 ]$} \\ \hline
$s_2$
& $5^2(j-1)^2$ \\
\hline
$s_3$
& $2^1(j-1)^2(49j+5)$ \\
\hline
\end{tabular}
\end{footnotesize}
\end{table}
\clearpage
\begin{table} \extrarowheight3pt
\begin{footnotesize}
\begin{tabular}{|c|l|}
\hline
$s_6$
& $3^5(3j^{12}-24j^{11}+22j^{10}+280j^9-651j^8-880j^7+3348j^6-848j^5
   -2819j^4+3336j^3$ \\
& $+2166j^2-34632j+42987)$ \\
\hline\hline
\multicolumn{2}{|c|} {$\fX_{33}^1$ $\quad[ 8, 4, 17 ]$} \\ \hline
$s_2$
& $288j^6-360j^5-295j^4+240j^3+238j^2+120j+25$ \\
\hline
$s_3$
& $(j-1)(j+1)(324j^8-4536j^7+9216j^6+4284j^5+3373j^4+4704j^3+3350j^2+924j
   +121)$ \\
\hline
$s_5$
& $2^6(j-1)j^2(j+1)(3j-1)^2(3j^2+1)^2(3j^3+7j^2-3j+1)^2$ \\
\hline
$s_6$
& $2^4(6561j^{20}+26244j^{19}+275562j^{18}+26244j^{17}+279693j^{16}
   +863136j^{15}-190440j^{14}$ \\
& $-779904j^{13}+700018j^{12}-123384j^{11}-72484j^{10}+20328j^9+35570j^8
   -28416j^7$ \\
\hline
\end{tabular}
\end{footnotesize}
\end{table}
\clearpage
\begin{table} \extrarowheight3pt
\begin{footnotesize}
\begin{tabular}{|c|l|}
\hline
$s_5$
& $-3^5(j-1)^2j(j+1)^2(3j+1)^2$ \\
\hline
$s_6$
& $3^5(3j^9+153j^8+1077j^7+2848j^6+4149j^5+2872j^4+987j^3+172j^2+24j+3)$ \\
\hline\hline
\multicolumn{2}{|c|} {$\fX_{39}^1$ $\quad[ 5, 4, 32 ]$} \\ \hline
$s_2$
& $(5j^4+4j^3-10j^2+4j+5)^2$ \\
\hline
$s_3$
& $189j^{12}-180j^{11}+514j^{10}-708j^9-45j^8+24j^7-68j^6+2328j^5-45j^4
   -2148j^3-158j^2$ \\
& $+684j+189$ \\
\hline
$s_5$
& $2^{10}(j-2)^2j^2(j^2-2j-1)^2(j^2-j-1)^2(j^3-j^2-2j-1)^2$ \\
\hline
$s_6$
& $2^{10}(4j^{24}-60j^{23}+363j^{22}-1056j^{21}+1088j^{20}+1680j^{19}
   -4884j^{18}+210j^{17}+8824j^{16}$ \\
\hline\hline
\end{tabular}
\end{footnotesize}
\end{table}
\clearpage
\begin{table} \extrarowheight3pt
\begin{footnotesize}
\begin{tabular}{|c|l|}
\hline\hline
\multicolumn{2}{|c|} {$\fX_{51}^3$ $\quad[ 12, 12, 20 ]$} \\ \hline
$s_2$
& $756j^5+7525j^4-3212j^3+1310j^2+16j+5$ \\
\hline
$s_3$
& $324j^8+175590j^7-12629j^6+808600j^5-198999j^4+446j^3+141j^2+676j-5$ \\
\hline
$s_5$
& $2^4(j-1)^2(j+3)(3j^2+14j-1)^2(9j^3+29j^2+27j-1)^2$ \\
\hline
$s_6$
& $2^4(6561j^{16}+159651j^{15}+2359287j^{14}+20593107j^{13}
   +111201451j^{12}+388751951j^{11}$ \\
& $+886832803j^{10}+1282535759j^9+1101478273j^8+485319401j^7+45568837j^6$ \\
& $-33682279j^5+4042729j^4-207643j^3+7777j^2-379j+10)$ \\
\hline\hline
\multicolumn{2}{|c|} {$\fX_{55}^1$ $\quad[ 13, 2, 17 ]$} \\
\hline
\end{tabular}
\end{footnotesize}
\end{table}
\clearpage
\begin{table} \extrarowheight3pt
\begin{footnotesize}
\begin{tabular}{|c|l|}
\hline
$s_6$
& $3^5(j+27)^4(3j^{14}+382j^{13}+20569j^{12}+614444j^{11}+11682731j^{10}
   +172489474j^9$ \\
& $+2509347785j^8+32632825960j^7+275862030289j^6+1190599936674j^5$ \\
& $+2104065334707j^4+64759585708j^3+2989599832449j^2+12785269126302j$ \\
& $+1883056329003)$ \\
\hline\hline
\multicolumn{2}{|c|} {$\fX_{62}^1$ $\quad[ 12, 4, 21 ]$} \\ \hline
$s_2$
& $36j^{10}+180j^9+61j^8+150j^7+125j^6-150j^5+312j^4-150j^3+17j^2-30j+25$ \\
\hline
$s_3$
& $1620j^{14}+1512j^{13}+3014j^{12}-1764j^{11}+8025j^{10}-2268j^9+822j^8
   +2520j^7-70j^6$ \\
& $+1008j^5+498j^4-756j^3+1545j^2-252j+98$ \\
\hline
\end{tabular}
\end{footnotesize}
\end{table}
\clearpage
\begin{table} \extrarowheight3pt
\begin{footnotesize}
\begin{tabular}{|c|l|}
\hline
$s_3$
& $(j-1)(j+1)(324j^{20}-4536j^{19}+21744j^{18}+125244j^{17}+936289j^{16}
   +3680880j^{15}$ \\
& $+6312080j^{14}+10965696j^{13}+29930140j^{12}+32824512j^{11}
   +5191488j^{10}+46870488j^9$ \\
& $+38695806j^8-92833776j^7+57165264j^6+127860768j^5-164864808j^4
   +56214648j^3$ \\
& $+142662384j^2-163683828j+64304361)$ \\
\hline
$s_5$
& $2^8(j-3)^2(j-1)j^2(j+1)(j^2-3)^2(3j^2-4j-3)^2(3j^3+j^2-3j-9)^2(j^4
   +6j^3-8j^2$ \\
& $-6j-9)^2(3j^4-10j^2-9)^2$ \\
\hline
$s_6$
& $2^4(6561j^{44}+26244j^{43}+4702050j^{42}-14355468j^{41}-65241369j^{40}
   +93776616j^{39}$ \\
& $+461753100j^{38}+2071294056j^{37}-5816799149j^{36}-35512227804j^{35}
   +89401438394j^{34}$ \\
\hline
\end{tabular}
\end{footnotesize}
\end{table}
\clearpage
\begin{table} \extrarowheight3pt
\begin{footnotesize}
\begin{tabular}{|c|l|}
\hline
$s_6$
& $3^5(3j^{18}-198j^{17}+11499j^{16}-331280j^{15}+4985820j^{14}
   -46503720j^{13}+325525004j^{12}$ \\
& $-1704531568j^{11}+4566552266j^{10}-1836897316j^9+3945589898j^8
   +55380735888j^7$ \\
& $+90674070924j^6+45211822808j^5+6752477788j^4+1366390768j^3
   +1245415371j^2$ \\
& $+266892858j+6223395)$ \\
\hline\hline
\multicolumn{2}{|c|} {$\fX_{74}^2$ $\quad[ 5, 4, 60 ]$} \\ \hline
$s_2$
& $(5j^6+24j^5+65j^4-48j^3+95j^2+408j-165)^2$ \\
\hline
$s_3$
& $2^1(7j^6+12j^5+43j^4+120j^3+85j^2-132j+441)(7j^{12}+132j^{11}
   +752j^{10}+1236j^9$ \\
& $+1095j^8+936j^7+1880j^6+28776j^5+66565j^4-79788j^3-36360j^2+48708j
   +21357)$ \\
\hline
\end{tabular}
\end{footnotesize}
\end{table}
\clearpage
\begin{table} \extrarowheight3pt
\begin{footnotesize}
\begin{tabular}{|c|l|}
\hline
$s_6$
& $2^43^5(1171875j^{32}+3125000j^{31}+475000000j^{30}-474125000j^{29}
   +5779025000j^{28}$ \\
& $-5742175000j^{27}-95391440000j^{26}+350930536600j^{25}
   -281328352460j^{24}$ \\
& $-962556109080j^{23}+2982091336640j^{22}-3507330305320j^{21}
   +1521095605400j^{20}$ \\
& $+788056060360j^{19}-1153308662400j^{18}+248480443960j^{17}
   +271162702898j^{16}$ \\
& $-175343002728j^{15}+13552752704j^{14}+17339705384j^{13}
   -4268145832j^{12}-501653448j^{11}$ \\
& $+56553600j^{10}+145104712j^9-40121228j^8-2435976j^7+2742080j^6
   -450104j^5+9768j^4$ \\
& $+4888j^3-384j^2-24j+3)$ \\
\hline\hline
\multicolumn{2}{|c|} {$\fX_{85}^2$ $\quad[ 17, 0, 20 ]$} \\
\hline
\end{tabular}
\end{footnotesize}
\end{table}
\clearpage
\begin{table} \extrarowheight3pt
\begin{footnotesize}
\begin{tabular}{|c|l|}
\hline
$s_6$
& $3^5(3j^{42}+36j^{41}+153j^{40}-2256j^{39}-3945j^{38}+89676j^{37}
   +119950j^{36}-1404192j^{35}$ \\
& $-2383032j^{34}+7508560j^{33}+22613599j^{32}+8283840j^{31}
   -68774916j^{30}-159748368j^{29}$ \\
& $-273874536j^{28}+120946816j^{27}+943496954j^{26}+1066628664j^{25}
   +3722087962j^{24}$ \\
& $-1912993632j^{23}+7195328626j^{22}-5606464856j^{21}+27529259172j^{20}
   +18915647040j^{19}$ \\
& $+89498935604j^{18}+87156343248j^{17}+169258399262j^{16}
   +150778256832j^{15}$ \\
& $+197442778380j^{14}+150265902000j^{13}+150291756088j^{12}
   +95606264448j^{11}$ \\
& $+76731518779j^{10}+40124067524j^9+26470916605j^8+11138352048j^7
   +6130361839j^6$ \\
& $+2010980460j^5+931867134j^4+224556000j^3+86256972j^2+12600000j
   +3583443)$ \\
\hline
\end{tabular}
\end{footnotesize}
\end{table}
\clearpage
\begin{table} \extrarowheight3pt
\begin{footnotesize}
\begin{tabular}{|c|l|}
\hline
$s_5$
& $2^4(j+1)^2(j+3)^2(j^2+3)^2(j^3+5j^2-j+3)^2(3j^4-10j^2-9)^2(3j^4+10j^3
   +20j^2+6j$ \\
& $+9)^2(3j^5+11j^4+14j^3-2j^2+15j-9)^2(j^6-7j^4+43j^2+27)$ \\
\hline
$s_6$
& $2^4(6561j^{56}+209952j^{55}+3107727j^{54}+27083808j^{53}
   +151649253j^{52}+575672832j^{51}$ \\
& $+1672925121j^{50}+5317158528j^{49}+22409317552j^{48}+83947238976j^{47}$ \\
& $+201055419218j^{46}+220206640704j^{45}+90128014654j^{44}
   +2406305458560j^{43}$ \\
& $+18279605142134j^{42}+69732828903552j^{41}+176393932898071j^{40}
   +342939555574752j^{39}$ \\
& $+626334259234949j^{38}+1375037361108960j^{37}+3638108496924127j^{36}$ \\
& $+9599672165937024j^{35}+22729530724785827j^{34}
   +47249650224708864j^{33}$ \\
\hline
\end{tabular}
\end{footnotesize}
\end{table}
\clearpage
\begin{table} \extrarowheight3pt
\begin{footnotesize}
\begin{tabular}{|c|l|}
\hline
$s_6$
& $2^4(6561j^{56}+52488j^{55}+8039412j^{54}-18528264j^{53}
   +298131354j^{52}+717336j^{51}$ \\
& $-4459857660j^{50}-2766012696j^{49}+114341029531j^{48}
   -249826493712j^{47}$ \\
& $-90414992776j^{46}+117762669840j^{45}+2477236048612j^{44}
   -3307030846320j^{43}$ \\
& $+17570787209528j^{42}-52974498085008j^{41}+85981469676457j^{40}
   -127493963548968j^{39}$ \\
& $+353500740375068j^{38}+59751119619624j^{37}+875675100692230j^{36}$ \\
& $+1160973488294088j^{35}+2755345462028396j^{34}+6874101522279480j^{33}$ \\
& $+14409011308121755j^{32}+30002705428157088j^{31}
   +56385633506364368j^{30}$ \\
& $+92060854455600480j^{29}+153246555369358072j^{28}
   +215939758436124000j^{27}$ \\
\hline
\end{tabular}
\end{footnotesize}
\end{table}
\clearpage
\begin{table} \extrarowheight3pt
\begin{footnotesize}
\begin{tabular}{|c|l|}
\hline
$s_6$
& $3^5(j-1)^4(j+1)^4(192j^{46}-1728j^{45}+1664j^{44}+31104j^{43}
   -93516j^{42}-218844j^{41}$ \\
& $+1245087j^{40}+374158j^{39}-9450619j^{38}+6055486j^{37}+48471934j^{36}
   -64951060j^{35}$ \\
& $-178127634j^{34}+372491412j^{33}+465823816j^{32}-1518266156j^{31}
   -751944434j^{30}$ \\
& $+4775193508j^{29}-6239914j^{28}-11974226954j^{27}+4697529653j^{26}
   +24169622966j^{25}$ \\
& $-18235172777j^{24}-38945443140j^{23}+45020268788j^{22}
   +48483389296j^{21}$ \\
& $-83287639424j^{20}-42418083712j^{19}+120682974208j^{18}
   +16526768640j^{17}$ \\
& $-138392494336j^{16}+19791160320j^{15}+124700978176j^{14}
   -47549837312j^{13}$ \\
& $-86241902592j^{12}+53196341248j^{11}+43562024960j^{10}-39506345984j^9
   -14288748544j^8$ \\
\hline
\end{tabular}
\end{footnotesize}
\end{table}
\clearpage
\begin{table} \extrarowheight3pt
\begin{footnotesize}
\begin{tabular}{|c|l|}
\hline
$s_6$
& $2^{10}3^5(12j^{72}+144j^{71}+801j^{70}+2577j^{69}+5481j^{68}
   +12222j^{67}+52500j^{66}$ \\
& $+230301j^{65}+659958j^{64}+1010439j^{63}+44550j^{62}-2788296j^{61}
   -3274811j^{60}$ \\
& $+5472903j^{59}+15453684j^{58}-4792677j^{57}-59488926j^{56}
   -57205071j^{55}+98169964j^{54}$ \\
& $+255403272j^{53}+59727178j^{52}-434867477j^{51}-479916765j^{50}
   +295817358j^{49}$ \\
& $+824532459j^{48}+48479037j^{47}-838398304j^{46}-111909442j^{45}
   +847823514j^{44}$ \\
& $-211963573j^{43}-1248242808j^{42}+575755135j^{41}+1966503137j^{40}
   -701753232j^{39}$ \\
& $-2621895701j^{38}+604437851j^{37}+2914517348j^{36}-407999189j^{35}
   -2753248863j^{34}$ \\
& $+224233121j^{33}+2270291570j^{32}-84775418j^{31}-1676650494j^{30}
   -1492531j^{29}$ \\
\hline
\end{tabular}
\end{footnotesize}
\end{table}
\clearpage
\begin{table} \extrarowheight3pt
\begin{footnotesize}
\begin{tabular}{|c|l|}
\hline
$s_3$
& $5103j^{36}+29160j^{35}-38448j^{34}-571428j^{33}-593928j^{32}
   +3258504j^{31}+7488926j^{30}$ \\
& $-531816j^{29}-16165560j^{28}-22856084j^{27}-9515496j^{26}
   +56894568j^{25}+94441569j^{24}$ \\
& $-73392528j^{23}-186291072j^{22}+78725112j^{21}+197095920j^{20}
   -108171888j^{19}$ \\
& $-129726748j^{18}+131740944j^{17}+49266672j^{16}-110375304j^{15}
   +14941008j^{14}$ \\
& $+63203664j^{13}-43577823j^{12}-12583896j^{11}+32355312j^{10}
   -14384980j^9-6393000j^8$ \\
& $+10660200j^7-4578850j^6-603432j^5+1735560j^4-984420j^3+307800j^2
   -56376j+5103$ \\
\hline
$s_5$
& $2^{10}3^5j^2(j^2+1)^2(j^3-j+1)^2(j^3+j^2-1)^2(j^3+2j^2+3j+1)^2(j^4-3j
   +3)^2(j^4$ \\
& $+j^2-2j+1)^2(j^4+j^3+j^2-j+1)^2(j^5+2j^4+3j^3-j^2-j+1)^2$ \\
\hline
\end{tabular}
\end{footnotesize}
\end{table}
\clearpage
\begin{table} \extrarowheight3pt
\begin{footnotesize}
\begin{tabular}{|c|l|}
\hline
$s_3$
& $108j^{13}-6866j^{12}+165921j^{11}-1559442j^{10}+8042330j^9-26507148j^8
   +60030444j^7$ \\
& $-96017120j^6+107800368j^5-81597888j^4+37802112j^3-8047104j^2-506880j
   +401408$ \\
\hline
$s_5$
& $3^5j(j+2)^2(j^2-2)^2(j^2+j-4)^2(j^2+2j-4)^2(2j^3-j^2-8j+8)^2$ \\
\hline
$s_6$
& $3^5(48j^{26}+368j^{25}-533j^{24}-10286j^{23}-3096j^{22}+143428j^{21}
   +92165j^{20}-1315962j^{19}$ \\
& $-665489j^{18}+8615204j^{17}+1337836j^{16}-40567600j^{15}
   +10590160j^{14}+133734208j^{13}$ \\
& $-91268288j^{12}-292190464j^{11}+330882304j^{10}+379026432j^9
   -677893120j^8$ \\
& $-202903552j^7+816771072j^6-118849536j^5-579534848j^4+273678336j^3
   +209715200j^2$ \\
& $-209715200j+50331648)$ \\
\hline
\end{tabular}
\end{footnotesize}
\end{table}
\clearpage
\begin{table} \extrarowheight3pt
\begin{footnotesize}
\begin{tabular}{|c|l|}
\hline
$s_6$
& $2^{10}(580j^{72}-5088j^{71}+19872j^{70}-65488j^{69}+214508j^{68}
   -409380j^{67}+257771j^{66}$ \\
& $-690236j^{65}+3746764j^{64}-2753337j^{63}-7997501j^{62}-3864637j^{61}
   +34853110j^{60}$ \\
& $+32014242j^{59}-94397332j^{58}-139122304j^{57}+164450219j^{56}
   +412919403j^{55}$ \\
& $-90421462j^{54}-870189638j^{53}-427299749j^{52}+1178195640j^{51}
   +1390336207j^{50}$ \\
& $-256705225j^{49}-1652462283j^{48}-3227365923j^{47}-1938223411j^{46}
   +8853522631j^{45}$ \\
& $+13996434451j^{44}-12170378145j^{43}-37409527603j^{42}
   +5173097872j^{41}$ \\
& $+69167410202j^{40}+18389414156j^{39}-99446150565j^{38}
   -56079801316j^{37}$ \\
& $+116681751516j^{36}+95204159674j^{35}-115368806298j^{34}
   -119843551709j^{33}$ \\
\hline
\end{tabular}
\end{footnotesize}
\end{table}
\clearpage
\begin{table} \extrarowheight3pt
\begin{footnotesize}
\begin{tabular}{|c|l|}
\hline
$s_3$
& $32832j^{36}+11520j^{35}-302720j^{34}-471168j^{33}+1042368j^{32}
   +3448960j^{31}+288736j^{30}$ \\
& $-10821696j^{29}-13156048j^{28}+12477856j^{27}+41935768j^{26}
   +15583240j^{25}-54678019j^{24}$ \\
& $-70800284j^{23}+11621554j^{22}+91215700j^{21}+49680155j^{20}
   -44469808j^{19}-47537564j^{18}$ \\
& $+2057200j^{17}-240709j^{16}-17322748j^{15}+11036578j^{14}
   +35742548j^{13}+13262693j^{12}$ \\
& $-16800568j^{11}-18306512j^{10}-4268704j^9+4080920j^8+3871200j^7
   +1460032j^6$ \\
& $+372032j^5+231264j^4+164736j^3+60160j^2+9984j+576$ \\
\hline
$s_5$
& $2^{10}(j^2-2j-1)^2(j^2-2j+2)^2(j^2+j-1)^2(j^3+j+1)^2(2j^3-2j^2+1)^2
   (j^4-2j^3+j$ \\
& $-1)^2(j^4-j^3-j^2+j+1)^2(j^4+1)^2(j^5-4j^4-j^3+4j^2-2j-3)^2$ \\
\hline
\end{tabular}
\end{footnotesize}
\end{table}
\clearpage
\begin{table} \extrarowheight3pt
\begin{footnotesize}
\begin{tabular}{|c|l|}
\hline
$s_3$
& $67500j^{22}-3815000j^{21}+184178550j^{20}-502278900j^{19}
   +124999867j^{18}-5015148066j^{17}$ \\
& $+26147456025j^{16}-42620924352j^{15}+27428803572j^{14}
   +1850252952j^{13}$ \\
& $-14419666584j^{12}+10855282008j^{11}-4431223686j^{10}+1138054452j^9
   -157885242j^8$ \\
& $+2729952j^7+2913704j^6+3064080j^5+160458j^4+155260j^3+28035j^2-1890j
   -135$ \\
\hline
$s_5$
& $2^43^5(j+1)^2(j+3)(j^3+7j^2-j+1)(5j^4+10j^3-4j^2+6j-1)^2(5j^5+13j^4
   +18j^3$ \\
& $-14j^2+9j+1)^2(5j^6+32j^5+45j^4-24j^3+15j^2-8j-1)^2$ \\
\hline
$s_6$
& $2^43^5(1171875j^{44}+45312500j^{43}+798109375j^{42}+8449250000j^{41}
   +60123932500j^{40}$ \\
& $+306396430000j^{39}+1169899059150j^{38}+3464705101800j^{37}
   +8073799986103j^{36}$ \\
\hline
\end{tabular}
\end{footnotesize}
\end{table}
\clearpage
\begin{table} \extrarowheight3pt
\begin{footnotesize}
\begin{tabular}{|c|l|}
\hline
$s_5$
& $2^{16}3^5(j-1)^2j(j+1)^2(j^2-10j-7)^2(j^3-19j^2-61j-49)^2(j^4+12j^3
   +86j^2+108j$ \\
& $+49)^2(j^8-8j^7-68j^6+72j^5-2202j^4-20024j^3-32772j^2-12936j+2401)^2$ \\
\hline
$s_6$
& $2^{12}3^5(3j^{48}-192j^{47}+15064j^{46}-221760j^{45}-8476868j^{44}
   +272558272j^{43}$ \\
& $-2049791576j^{42}-19022839488j^{41}+427148929350j^{40}
   -1325955093568j^{39}$ \\
& $-30004596433336j^{38}+222586485951808j^{37}+1152456051118540j^{36}$ \\
& $-10548683847579584j^{35}-28604431828938952j^{34}
   +13664386221900736j^{33}$ \\
& $-1903651701506426739j^{32}+6113559417344210048j^{31}
   +246356840492833153776j^{30}$ \\
& $+1613309123572867192192j^{29}+4928380987621787223416j^{28}$ \\
\hline
\end{tabular}
\end{footnotesize}
\end{table}
\clearpage
\begin{table} \extrarowheight3pt
\begin{footnotesize}
\begin{tabular}{|c|l|}
\hline
$s_6$
& $2^{12}3^5(3j^{96}-54j^{95}+87j^{94}+3528j^{93}-14399j^{92}
   -117306j^{91}+596211j^{90}$ \\
& $+2742708j^{89}-13972303j^{88}-51170598j^{87}+212136771j^{86}
   +756066080j^{85}$ \\
& $-2050721099j^{84}-7899070074j^{83}+9282395999j^{82}+44675853020j^{81}
   +43298131135j^{80}$ \\
& $+66930930038j^{79}-926074010239j^{78}-3097529528632j^{77}
   +6100591582183j^{76}$ \\
& $+16023087388506j^{75}-33567004955163j^{74}+41791468939260j^{73}
   +289770285256095j^{72}$ \\
& $-546165196097962j^{71}-1066045014962851j^{70}+441690264745968j^{69}$ \\
& $-11007668340857237j^{68}-9046801052949590j^{67}
   +77435831867824353j^{66}$ \\
& $+103677730362616564j^{65}+270748634398526004j^{64}
   +1205258580962063908j^{63}$ \\
\hline\hline
\end{tabular}
\end{footnotesize}
\end{table}
\clearpage
\begin{table} \extrarowheight3pt
\begin{footnotesize}
\begin{tabular}{|c|l|}
\hline\hline
\multicolumn{2}{|c|} {$\fX_{119}^3$ $\quad[ 12, 4, 40 ]$} \\ \hline
$s_2$
& $225j^{32}-2880j^{31}+7056j^{30}+22080j^{29}-148904j^{28}+530880j^{27}
   -1244432j^{26}$ \\
& $-1548480j^{25}+13544956j^{24}-16610880j^{23}+140929040j^{22}
   -140871360j^{21}$ \\
& $-324095576j^{20}-290498880j^{19}-4440944144j^{18}-1214692800j^{17}
   -1425835930j^{16}$ \\
& $-842888640j^{15}+71931400240j^{14}+9554247360j^{13}+211703302888j^{12}
   +8018944320j^{11}$ \\
& $+285345444944j^{10}+166231888320j^9+454026217660j^8+638996662080j^7$ \\
& $+670129521712j^6+797068009920j^5+477336682648j^4+384926015040j^3$ \\
& $+145826406096j^2+59295096000j+26201020545$ \\
\hline
\end{tabular}
\end{footnotesize}
\end{table}
\clearpage
\begin{table} \extrarowheight3pt
\begin{footnotesize}
\begin{tabular}{|c|l|}
\hline
$s_6$
& $2^{12}3^5(3j^{96}+36j^{95}-63j^{94}-2184j^{93}-63j^{92}+48132j^{91}
   -119679j^{90}-394200j^{89}$ \\
& $+8727477j^{88}+10507764j^{87}-213157203j^{86}-352574376j^{85}
   +2353538173j^{84}$ \\
& $+1834302036j^{83}-17988192579j^{82}+54848498056j^{81}
   +224586924603j^{80}$ \\
& $-611659340292j^{79}-1798477486729j^{78}+634665656520j^{77}
   -3288351640553j^{76}$ \\
& $+7613628298268j^{75}+118635729719255j^{74}-100497482592744j^{73}
   -977192469675093j^{72}$ \\
& $+482950732530700j^{71}+4230185731420355j^{70}+8599757329733416j^{69}$ \\
& $+26898717327422227j^{68}+13360413165239980j^{67}
   -146827390882419373j^{66}$ \\
& $-140588212516832136j^{65}+121149776402410728j^{64}
   -1353031088558056600j^{63}$ \\
\hline\hline
\end{tabular}
\end{footnotesize}
\end{table}
\clearpage
\begin{table} \extrarowheight3pt
\begin{footnotesize}
\begin{tabular}{|c|l|}
\hline\hline
\multicolumn{2}{|c|} {$\fX_{119}^4$ $\quad[ 20, 4, 24 ]$} \\ \hline
$s_2$
& $225j^{32}+2880j^{31}+14736j^{30}-68160j^{29}-471464j^{28}+221760j^{27}
   +8056048j^{26}$ \\
& $+5204160j^{25}-60951044j^{24}-211976640j^{23}-189671920j^{22}
   +473000640j^{21}$ \\
& $+2936233384j^{20}+8951987520j^{19}+22439386096j^{18}+46776937920j^{17}
   +89236072550j^{16}$ \\
& $+155913823680j^{15}+247910513200j^{14}+335515817280j^{13}
   +512974665448j^{12}$ \\
& $+414143021760j^{11}+807523336784j^{10}+158205195840j^9+991342334140j^8$ \\
& $-299972984640j^7+917475249712j^6-481680991680j^5+575693935768j^4$ \\
& $-278824096320j^3+204556786896j^2-59295096000j+26201020545$ \\
\hline
\end{tabular}
\end{footnotesize}
\end{table}
\clearpage
\begin{table} \extrarowheight3pt
\begin{footnotesize}
\begin{tabular}{|c|l|}
\hline
$s_6$
& $2^{12}3^5(3j^{96}-36j^{95}+57j^{94}+1512j^{93}-11039j^{92}+11676j^{91}
   +231161j^{90}$ \\
& $-1391976j^{89}+1974581j^{88}+15893324j^{87}-102841851j^{86}
   +205416008j^{85}$ \\
& $+536941853j^{84}-4582626804j^{83}+11265506325j^{82}+4622891768j^{81}
   -115426920901j^{80}$ \\
& $+340426263940j^{79}-200165940289j^{78}-1781170542760j^{77}
   +6141545123063j^{76}$ \\
& $-3592288893756j^{75}-33989639511521j^{74}+113834533924264j^{73}
   -54489765697493j^{72}$ \\
& $-699210894905292j^{71}+2519762500974187j^{70}-3045943689525576j^{69}$ \\
& $-5060869674391181j^{68}+21139998148080756j^{67}
   +17911638340098811j^{66}$ \\
& $-295189829685782136j^{65}+879127861205181672j^{64}
   -416938591081795944j^{63}$ \\
\hline\hline
\end{tabular}
\end{footnotesize}
\end{table}
\clearpage
\begin{table} \extrarowheight3pt
\begin{footnotesize}
\begin{tabular}{|c|l|}
\hline\hline
\multicolumn{2}{|c|} {$\fX_{119}^5$ $\quad[ 12, 8, 41 ]$} \\ \hline
$s_2$
& $225j^{32}+1440j^{31}-144j^{30}-54240j^{29}-98504j^{28}+500640j^{27}
   +633808j^{26}-845280j^{25}$ \\
& $+2132956j^{24}-19238880j^{23}+44946800j^{22}+113746080j^{21}
   +284343304j^{20}$ \\
& $+911746080j^{19}-1835298224j^{18}+3508935840j^{17}+2128193990j^{16}
   -20801675040j^{15}$ \\
& $+81618915280j^{14}-155127131040j^{13}+239889284488j^{12}
   -504583760160j^{11}$ \\
& $+335951681264j^{10}-769379636640j^9+489545251420j^8-473562959520j^7$ \\
& $+659357150032j^6-75410553120j^5+431199450808j^4-35214026400j^3
   +96978541296j^2$ \\
& $-32019351840j+1297080225$ \\
\hline
\end{tabular}
\end{footnotesize}
\end{table}
\clearpage
\begin{table} \extrarowheight3pt
\begin{footnotesize}
\begin{tabular}{|c|l|}
\hline
$s_6$
& $2^{12}3^5(3j^{96}-18j^{95}-81j^{94}+840j^{93}+1953j^{92}-9198j^{91}
   +50379j^{90}-527652j^{89}$ \\
& $-3657423j^{88}+34872702j^{87}+46659963j^{86}-824174736j^{85}
   +541913365j^{84}$ \\
& $+6409773426j^{83}-18975643881j^{82}+37721374132j^{81}
   +150053014719j^{80}$ \\
& $-682911644142j^{79}+73829825769j^{78}+2326580457768j^{77}
   -4317817695737j^{76}$ \\
& $-28513353882354j^{75}-20537749504819j^{74}+45481526577012j^{73}
   +381938972402463j^{72}$ \\
& $+3242074587512530j^{71}+2091133185961253j^{70}+2951504100805120j^{69}$ \\
& $-34372002246253301j^{68}-67305684975813410j^{67}
   -120815691395131607j^{66}$ \\
& $-274936025062707204j^{65}-205603411102825932j^{64}
   -3412374678592066292j^{63}$ \\
\hline
\end{tabular}
\end{footnotesize}
\end{table}
\clearpage
\begin{table} \extrarowheight3pt
\begin{footnotesize}
\begin{tabular}{|c|l|}
\hline
$s_3$
& $7^1(729j^{24}-100008j^{23}+5776452j^{22}-127088248j^{21}
   +451004034j^{20}+22158037896j^{19}$ \\
& $-289167727148j^{18}+538365671448j^{17}+6008497964487j^{16}
   -858178335504j^{15}$ \\
& $-95627649919992j^{14}-153054295716912j^{13}+374492117817308j^{12}$ \\
& $+2490216562698384j^{11}+6439869188898888j^{10}+11183702877720240j^9$ \\
& $+14185264236892407j^8+14537973766611768j^7+12914207968121524j^6$ \\
& $+10904434660312872j^5+8016773124634050j^4+4996764411370088j^3$ \\
& $+2580204489377508j^2+1304554312555704j+386794770831945)$ \\
\hline
$s_5$
& $2^{12}3^5(j+3)^2(j+7)^2(j^3-3j^2-13j-49)^2(j^4+12j^3+86j^2+108j+49)^2
   (j^5+3j^4$ \\
\hline
\end{tabular}
\end{footnotesize}
\end{table}
\clearpage
\begin{table} \extrarowheight3pt
\begin{footnotesize}
\begin{tabular}{|c|l|}
\hline
$s_3$
& $(13j^{10}+72j^9+891j^8+528j^7+3010j^6-3360j^5+6678j^4-5136j^3+7777j^2
   -1320j+63)$ \\
& $(27j^{22}-72j^{21}+1811j^{20}-1152j^{19}+7641j^{18}-68088j^{17}
   -228335j^{16}-321600j^{15}$ \\
& $+539262j^{14}+4311792j^{13}+14781486j^{12}+1727424j^{11}-5174622j^{10}
   -1061424j^9$ \\
& $+5457682j^8-112745664j^7+202388823j^6-130310952j^5+61890479j^4
   -16060992j^3$ \\
& $+3405717j^2-273240j+45549)$ \\
\hline
$s_5$
& $3^5(j+1)^2(j^2-4j-1)^2(j^2+8j+3)^2(5j^2-1)^2(j^3-2j^2+3j+2)^2(j^3-j^2
   +11j+5)^2$ \\
& $(j^4+8j^2-1)^2(j^4+22j^2-7)^2(j^6+j^4+75j^2-13)^2$ \\
\hline
$s_6$
& $3^5(3j^{66}+12j^{65}+1213j^{64}-1856j^{63}+240633j^{62}-121572j^{61}
   +12684416j^{60}$ \\
\hline
\end{tabular}
\end{footnotesize}
\end{table}
\clearpage
\begin{table} \extrarowheight3pt
\begin{footnotesize}
\begin{tabular}{|c|l|}
\hline
$s_6$
& $3^5(3j^{66}-24j^{65}+649j^{64}-4352j^{63}+64225j^{62}-367416j^{61}
   +3857036j^{60}$ \\
& $-18905664j^{59}+156649839j^{58}-654691928j^{57}+4542119562j^{56}
   -15955290816j^{55}$ \\
& $+97041093333j^{54}-278443484664j^{53}+1562481397016j^{52}
   -3466156456064j^{51}$ \\
& $+19418556770687j^{50}-29800233827832j^{49}+193513181528810j^{48}
   -159755429976960j^{47}$ \\
& $+1647766434723949j^{46}-313041784813208j^{45}+12873481916396916j^{44}$ \\
& $+2453094741963072j^{43}+93878175147916387j^{42}
   +25134645303512136j^{41}$ \\
& $+603409688801515386j^{40}+120405216460641728j^{39}
   +3187953664928511393j^{38}$ \\
& $+347249434708561704j^{37}+13328026258623636912j^{36}
   +7138012550191104j^{35}$ \\
\hline
\end{tabular}
\end{footnotesize}
\end{table}
\clearpage
\begin{table} \extrarowheight3pt
\begin{footnotesize}
\begin{tabular}{|c|l|}
\hline
$s_3$
& $2619j^{32}-56448j^{31}+794636j^{30}-8227296j^{29}+63883224j^{28}
   -390273408j^{27}$ \\
& $+2100991956j^{26}-8951920992j^{25}+35681977836j^{24}
   -113035966848j^{23}$ \\
& $+334312050828j^{22}-791850070368j^{21}+1662902925912j^{20}
   -2810798605440j^{19}$ \\
& $+3859458414644j^{18}-3894323346144j^{17}+3318430357962j^{16}
   -198617811840j^{15}$ \\
& $-1985646547068j^{14}+774942297696j^{13}+519657212808j^{12}
   -458205277824j^{11}$ \\
& $-221300519844j^{10}+256043356128j^9+36564821820j^8-71061024384j^7
   +5294568932j^6$ \\
& $+11265651936j^5-884598840j^4-760669056j^3+19426716j^2+12573792j
   +2869587$ \\
\hline
$s_5$
& $3^5(j+1)^2(j^2-5)^2(j^2+3)^2(j^2+4j-1)^2(j^2+4j+7)^2(j^3+2j^2+3j-2)^2
   (j^3+3j^2$ \\
\hline
\end{tabular}
\end{footnotesize}
\end{table}
\clearpage
\begin{table} \extrarowheight3pt
\begin{footnotesize}
\begin{tabular}{|c|l|}
\hline
$s_5$
& $2^{10}3^6(j^2-5j+13)^2(j^2-2j-26)^2(j^2+4j+13)^2(j^3-6j^2-9j-13)^2
   (4j^3-24j^2$ \\
& $-169)(j^4-10j^3+24j^2-22j+169)^2(j^{10}+5j^9-27j^8-894j^7+3711j^6
   -5535j^5$ \\
& $+45174j^4-29523j^3+118638j^2-101062j+28561)^2$ \\
\hline
$s_6$
& $2^{10}3^6(4j^{60}-224j^{59}+7456j^{58}-150604j^{57}+639224j^{56}
   +54582656j^{55}$ \\
& $-1743555417j^{54}+28623550536j^{53}-295026572448j^{52}
   +1836788957994j^{51}$ \\
& $-3985018461690j^{50}-47334134155074j^{49}+706134905469483j^{48}
   -6520008677154618j^{47}$ \\
& $+49895911778668485j^{46}-311109851182011732j^{45}
   +1679719424833977507j^{44}$ \\
& $-9609198948169922808j^{43}+60899375778442022922j^{42}
   -371599380578867190204j^{41}$ \\
\hline
\end{tabular}
\end{footnotesize}
\end{table}
\clearpage
\begin{table} \extrarowheight3pt
\begin{footnotesize}
\begin{tabular}{|c|l|}
\hline
$s_5$
& $2^{10}3^8(j-1)^2j^2(j^2-5j+13)^2(j^3-6j^2-9j-13)^2(j^5-20j^4+7j^3
   -127j^2+65j$ \\
& $-169)^2(j^5-20j^4+115j^3-100j^2+416j-169)(j^{10}+5j^9-27j^8-894j^7
   +3711j^6$ \\
& $-5535j^5+45174j^4-29523j^3+118638j^2-101062j+28561)^2$ \\
\hline
$s_6$
& $2^{10}3^6(4j^{60}-384j^{59}+28176j^{58}-1132631j^{57}+22693650j^{56}
   -86300484j^{55}$ \\
& $-8119428422j^{54}+275013854721j^{53}-5217877391751j^{52}
   +68894985466830j^{51}$ \\
& $-660617034960609j^{50}+4531407918925095j^{49}-19664599117084224j^{48}$ \\
& $+5519329209840087j^{47}+910402762448649834j^{46}
   -11589536410908128385j^{45}$ \\
& $+100431522287671167495j^{44}-706128415655992617105j^{43}$ \\
\hline
\end{tabular}
\end{footnotesize}
\end{table}
\clearpage
\begin{table} \extrarowheight3pt
\begin{footnotesize}
\begin{tabular}{|c|l|}
\hline
$s_6$
& $3^5(3j^{72}+132j^{71}+2710j^{70}+34380j^{69}+300607j^{68}
   +1910428j^{67}+9047397j^{66}$ \\
& $+32065924j^{65}+83428484j^{64}+149487084j^{63}+148691309j^{62}
   -17715096j^{61}$ \\
& $-212079861j^{60}+91104728j^{59}+1100674336j^{58}+1539111176j^{57}
   -217340254j^{56}$ \\
& $-2076823036j^{55}+1121265442j^{54}+7052781256j^{53}+2917009283j^{52}
   -11505766128j^{51}$ \\
& $-9462617967j^{50}+19362459684j^{49}+27950213847j^{48}
   -18126057268j^{47}$ \\
& $-50313941103j^{46}+8521187544j^{45}+77091027126j^{44}
   +17276460628j^{43}$ \\
& $-96528581839j^{42}-50686333296j^{41}+111474674726j^{40}
   +88856088300j^{39}$ \\
& $-120844840493j^{38}-118730514656j^{37}+130939773850j^{36}
   +132460366428j^{35}$ \\
\hline
\end{tabular}
\end{footnotesize}
\end{table}
\clearpage
\begin{table} \extrarowheight3pt
\begin{footnotesize}
\begin{tabular}{|c|l|}
\hline
$s_3$
& $539j^{36}+7098j^{35}+33885j^{34}+49220j^{33}-90450j^{32}-62580j^{31}
   +1571773j^{30}$ \\
& $+3043386j^{29}-4797138j^{28}-17318150j^{27}+855021j^{26}
   +47255172j^{25}+29209501j^{24}$ \\
& $-81884430j^{23}-95704338j^{22}+106887278j^{21}+175449684j^{20}
   -121512714j^{19}$ \\
& $-214035218j^{18}+130875018j^{17}+179396004j^{16}-123797486j^{15}
   -96629682j^{14}$ \\
& $+86347854j^{13}+30470005j^{12}-42435924j^{11}-3011163j^{10}
   +15569270j^9-3718578j^8$ \\
& $-2737962j^7+1821253j^6-315420j^5-55170j^4+30412j^3-3915j^2-42j+35$ \\
\hline
$s_5$
& $3^5(j^2+j+1)^2(j^3+2j-1)^2(j^3+j^2+j-2)^2(j^3+2j^2+2j-1)^2(j^4+j^3-j
   +1)^2(j^4$ \\
& $+j^3+2j^2-j+1)^2(j^4+2j^3+3j^2-2j+1)^2(j^6+3j^5+2j^4+j^3-2j^2+3j-1)^2$ \\
\hline
\end{tabular}
\end{footnotesize}
\end{table}
\clearpage
\begin{table} \extrarowheight3pt
\begin{footnotesize}
\begin{tabular}{|c|l|}
\hline
$s_6$
& $3^5(3j^{72}+108j^{71}+1882j^{70}+21052j^{69}+169007j^{68}
   +1030968j^{67}+4931257j^{66}$ \\
& $+18785940j^{65}+57157984j^{64}+137234544j^{63}+250540641j^{62}
   +314067344j^{61}$ \\
& $+172847111j^{60}-216472592j^{59}-463842996j^{58}+189311512j^{57}
   +1800052602j^{56}$ \\
& $+2481702740j^{55}-167581258j^{54}-4224569160j^{53}-2165324405j^{52}
   +8578258148j^{51}$ \\
& $+13742323197j^{50}-3262102980j^{49}-25501021861j^{48}-9515363360j^{47}
   +39163582433j^{46}$ \\
& $+39423415736j^{45}-39163393590j^{44}-75475994480j^{43}
   +26750577313j^{42}$ \\
& $+115447902512j^{41}+2479248818j^{40}-146604053640j^{39}
   -39064617501j^{38}$ \\
& $+164031879632j^{37}+72888760326j^{36}-164559596976j^{35}
   -90714563653j^{34}$ \\
\hline
\end{tabular}
\end{footnotesize}
\end{table}
\clearpage
\begin{table} \extrarowheight3pt
\begin{footnotesize}
\begin{tabular}{|c|l|}
\hline
$s_6$
& $2^4(12157665459056928801j^{52}+997829135454450156408j^{51}$ \\
& $+60491217744244517522742j^{50}+1976975853614917158990456j^{49}$ \\
& $+41323696367986190053000773j^{48}+607927217613902092980236448j^{47}$ \\
& $+6623224727788956620278034712j^{46}
   +55344313019809562103807796512j^{45}$ \\
& $+364836624728008914186834900006j^{44}
   +1937758042803079602554702623920j^{43}$ \\
& $+8395510141005773383510424096940j^{42}
   +29757133931474582876914398382896j^{41}$ \\
& $+85745738145105316912646975455062j^{40}
   +197911603333065953943106195674528j^{39}$ \\
& $+357430757332329048762167702655672j^{38}
   +489006945816582386192450703824160j^{37}$ \\
\hline
\end{tabular}
\end{footnotesize}
\end{table}
\clearpage
\begin{table} \extrarowheight3pt
\begin{footnotesize}
\begin{tabular}{|c|l|}
\hline
$s_6$
& $2^4(12157665459056928801j^{104}-324204412241518101360j^{103}$ \\
& $+4456009532697309904248j^{102}-36678325838257081200528j^{101}$ \\
& $+215225578407544475351094j^{100}-1029879713848907139080400j^{99}$ \\
& $+4431438444546965432216952j^{98}-18177063749703714411047088j^{97}$ \\
& $+71206956583420550005913157j^{96}-258731788884934765844001600j^{95}$ \\
& $+856928078838517271496255648j^{94}-2586547222017741227108088000j^{93}$ \\
& $+7224111899660664565502244120j^{92}
   -18992584300525068947241786048j^{91}$ \\
& $+47685099265511793923653761312j^{90}
   -115155320884964136504960061248j^{89}$ \\
\hline\hline
\end{tabular}
\end{footnotesize}
\end{table}
\clearpage
\begin{table} \extrarowheight3pt
\begin{footnotesize}
\begin{tabular}{|c|l|}
\hline\hline
\multicolumn{2}{|c|} {$\fX_{159}^3$ $\quad[ 8, 4, 80 ]$} \\ \hline
$s_2$
& $61647156j^{34}+56687040j^{33}+240637797j^{32}-1056058560j^{31}
   +11980409328j^{30}$ \\
& $-28716068160j^{29}+59864359608j^{28}-32687478720j^{27}
   +49415028768j^{26}$ \\
& $+148098144960j^{25}+135875404908j^{24}+301829198400j^{23}
   +507967336880j^{22}$ \\
& $+476655483840j^{21}+679225134152j^{20}+632918894400j^{19}
   +582391858344j^{18}$ \\
& $+417768909120j^{17}+251245951358j^{16}+129337583040j^{15}
   +54053528272j^{14}$ \\
& $+16971693120j^{13}+4173441416j^{12}+451673280j^{11}-15055232j^{10}
   -42724800j^9$ \\
& $-17259220j^8-1661760j^7-249968j^6+25920j^5+36088j^4+960j^3-636j^2+5$ \\
\hline
\end{tabular}
\end{footnotesize}
\end{table}
\clearpage
\begin{table} \extrarowheight3pt
\begin{footnotesize}
\begin{tabular}{|c|l|}
\hline
$s_6$
& $2^4(12157665459056928801j^{104}-129681764896607240544j^{103}$ \\
& $+9054759063562065972567j^{102}+7363042429129588879776j^{101}$ \\
& $+423457719645847938880113j^{100}-133985989839595318448256j^{99}$ \\
& $+6600922267585021459409997j^{98}-2933296590811632925750272j^{97}$ \\
& $+101947218901026349272162534j^{96}+74058284811542540735046528j^{95}$ \\
& $+1341957194226731384006549364j^{94}+1643004533738004313455736704j^{93}$ \\
& $+12320278360388137540910306556j^{92}
   +15380174598227275212075958272j^{91}$ \\
& $+81236379820411383399804302316j^{90}
   +89727956216833813682023208448j^{89}$ \\
\hline\hline
\end{tabular}
\end{footnotesize}
\end{table}
\clearpage
\begin{table} \extrarowheight3pt
\begin{footnotesize}
\begin{tabular}{|c|l|}
\hline\hline
\multicolumn{2}{|c|} {$\fX_{159}^4$ $\quad[ 20, 4, 32 ]$} \\ \hline
$s_2$
& $146677716j^{34}+1927359360j^{33}+13904313957j^{32}+43942953600j^{31}
   +127377026928j^{30}$ \\
& $+286422583680j^{29}+478304605368j^{28}+728240699520j^{27}
   +802075461408j^{26}$ \\
& $+805816022400j^{25}+657145120108j^{24}+379454666880j^{23}
   +230168723760j^{22}$ \\
& $-83344755840j^{21}-39117675448j^{20}-104187150720j^{19}
   +44316220904j^{18}$ \\
& $+21238508160j^{17}+19885527678j^{16}-16750293120j^{15}
   +1820115792j^{14}-1266019200j^{13}$ \\
& $+481207176j^{12}+91693440j^{11}-49464192j^{10}-1584000j^9+4596780j^8
   -420480j^7$ \\
& $-259568j^6+36480j^5+10488j^4+1920j^3+164j^2+5$ \\
\hline
\end{tabular}
\end{footnotesize}
\end{table}
\clearpage
\begin{table} \extrarowheight3pt
\begin{footnotesize}
\begin{tabular}{|c|l|}
\hline
$s_6$
& $2^4(12157665459056928801j^{104}-259363529793214481088j^{103}$ \\
& $+3107409234553772802063j^{102}-26310989189022757914816j^{101}$ \\
& $+176329815577075761247665j^{100}-986843720215183923853056j^{99}$ \\
& $+4792119325794118605694869j^{98}-20647328639976167889959424j^{97}$ \\
& $+80280927525529667518454502j^{96}-284563990634715416986064640j^{95}$ \\
& $+927118491859581182557310484j^{94}-2791556776622050413302668032j^{93}$ \\
& $+7809558699987792268262150652j^{92}
   -20384795222054897580233879040j^{91}$ \\
& $+49869434521270108172226358476j^{90}
   -114764863437035758227051154944j^{89}$ \\
\hline\hline
\end{tabular}
\end{footnotesize}
\end{table}
\clearpage
\begin{table} \extrarowheight3pt
\begin{footnotesize}
\begin{tabular}{|c|l|}
\hline\hline
\multicolumn{2}{|c|} {$\fX_{159}^5$ $\quad[ 12, 12, 56 ]$} \\ \hline
$s_2$
& $-(23383404j^{17}-3255548517j^{16}-8457414768j^{15}-30351640248j^{14}
   -159551078688j^{13}$ \\
& $-281397751148j^{12}+2330444240j^{11}+247439295928j^{10}-23510915944j^9$ \\
& $-27342311038j^8+13246038448j^7-4829364616j^6+853124992j^5-78119980j^4$ \\
& $+4065008j^3-115448j^2+1436j-5)$ \\
\hline
$s_3$
& $-(13947137604j^{26}+4777066816254j^{25}-67979945662101j^{24}
   +1101744960466752j^{23}$ \\
& $+10373809920600420j^{22}+31109806297146420j^{21}
   +26259013327880358j^{20}$ \\
& $-56213874342944208j^{19}-97218146970658136j^{18}
   +86846581660072522j^{17}$ \\
\hline
\end{tabular}
\end{footnotesize}
\end{table}
\clearpage
\begin{table} \extrarowheight3pt
\begin{footnotesize}
\begin{tabular}{|c|l|}
\hline
$s_3$
& $13947137604j^{52}-911212990128j^{51}+10708991063496j^{50}
   -25346024116056j^{49}$ \\
& $+84779631097353j^{48}-17460659864592j^{47}+581413762915272j^{46}
   -814327355133936j^{45}$ \\
& $-2288888390766852j^{44}-7696904073079536j^{43}-4845451662847512j^{42}$ \\
& $-23639386716288j^{41}+18998054809724610j^{40}+57582916373686032j^{39}$ \\
& $+88402549098090456j^{38}+69624708421037616j^{37}
   -15076542500349536j^{36}$ \\
& $+55708912618750272j^{35}+124328925153276544j^{34}
   +44128291118946216j^{33}$ \\
& $-30585108562546665j^{32}-165446161456687776j^{31}
   -113361126681830832j^{30}$ \\
& $-104241837197454432j^{29}+64549737811647832j^{28}
   +67722170833307424j^{27}$ \\
& $+38107690888961488j^{26}-9134274785528256j^{25}-8283796714026660j^{24}$ \\
& $-7929586147910112j^{23}+889211703141936j^{22}+513755470630752j^{21}$ \\
& $+664906713581668j^{20}+77026560122064j^{19}-39376876679224j^{18}
   -27635171544456j^{17}$ \\
& $-9092346907577j^{16}+4904250288j^{15}+931208745576j^{14}
   +306218831184j^{13}$ \\
& $+53731892268j^{12}-7808169264j^{11}-7102019384j^{10}-1309769664j^9
   -52093694j^8$ \\
& $+68697552j^7+24953208j^6+815472j^5-827592j^4-112224j^3+3952j^2+1848j
   +121$ \\
\hline
$s_5$
& $2^8j^2(j^2+1)^2(3j^3-j^2+5j+1)^2(3j^3+7j^2-3j+1)^2(3j^4+14j^2-1)^2
   (9j^4-8j^3$ \\
& $+22j^2-8j+1)^2(9j^4-6j^3+8j^2+6j-1)^2(9j^4+22j^2+1)^2(9j^5-7j^4+22j^3
   +6j^2+j$ \\
& $+1)^2(27j^6+43j^4-7j^2+1)^2(9j^7+87j^6+29j^5+163j^4+91j^3+5j^2-j+1)^2$ \\
\hline
\end{tabular}
\end{footnotesize}
\end{table}
\clearpage
\begin{table} \extrarowheight3pt
\begin{footnotesize}
\begin{tabular}{|c|l|}
\hline
$s_6$
& $2^4(12157665459056928801j^{104}+226943088569062670952j^{103}$ \\
& $+6176094053200919830908j^{102}+20205139425140833672536j^{101}$ \\
& $+215822087842578149191086j^{100}+1051651203827603678542872j^{99}$ \\
& $+5076452027355668033712156j^{98}+20643635837086655126472360j^{97}$ \\
& $+93388318192932284698580997j^{96}+221335634765598508826905824j^{95}$ \\
& $+1235528684609107924558578960j^{94}+1528375146458725934659895328j^{93}$ \\
& $+11709312214620513657996280440j^{92}
   +7441384992504614545975083936j^{91}$ \\
& $+82225265481587880641187198288j^{90}
   +27812785703398632827048812896j^{89}$ \\
\hline\hline
\end{tabular}
\end{footnotesize}
\end{table}
\clearpage
\begin{table} \extrarowheight3pt
\begin{footnotesize}
\begin{tabular}{|c|l|}
\hline\hline
\multicolumn{2}{|c|} {$\fX_{166}^2$ $\quad[ 8, 8, 85 ]$} \\ \hline
$s_2$
& $j(73984j^{13}-988608j^{12}+5923968j^{11}-21008540j^{10}+49154381j^9
   -80690978j^8$ \\
& $+98201753j^7-95417444j^6+80325756j^5-59034192j^4+33751568j^3
   -12439168j^2$ \\
& $+2184448j-36864)$ \\
\hline
$s_3$
& $j^2(1769472j^{20}-14483456j^{19}-77349888j^{18}+1748275200j^{17}
   -12201096576j^{16}$ \\
& $+49092462336j^{15}-124547974212j^{14}+179392084272j^{13}
   -5149393269j^{12}$ \\
& $-689415050772j^{11}+1966674660380j^{10}-3362773381632j^9
   +4081222723932j^8$ \\
& $-3669658527152j^7+2455549115760j^6-1200922252800j^5+412851699200j^4$ \\
\hline
\end{tabular}
\end{footnotesize}
\end{table}
\clearpage
\begin{table} \extrarowheight3pt
\begin{footnotesize}
\begin{tabular}{|c|l|}
\hline
$s_5$
& $3^5(j^2-j+1)^2(j^2+j-1)^2(j^3-2j^2+4j-1)^2(j^3-j^2+j+2)^2(j^3+2j+1)^2
   (j^4-2j^3$ \\
& $+3j^2+2j+1)^2(j^4-j^3+j+1)^2(j^4-j^3+2j^2+j+1)^2(3j^4-2j^3+5j^2+2j
   +3)^2(2j^5$ \\
& $-3j^4+4j^3+j^2-2j-1)^2(j^6-j^5+2j^4+j^3-2j^2-j-1)^2$ \\
\hline
$s_6$
& $3^5(435j^{96}-9000j^{95}+101076j^{94}-786420j^{93}+4710238j^{92}
   -23001908j^{91}$ \\
& $+95422283j^{90}-347796260j^{89}+1149359335j^{88}-3541739868j^{87}
   +10383079570j^{86}$ \\
& $-29149976668j^{85}+77836465668j^{84}-194797542028j^{83}
   +449747321059j^{82}$ \\
& $-945036624672j^{81}+1787756533594j^{80}-3016287324000j^{79}
   +4496326477751j^{78}$ \\
& $-5856096479852j^{77}+6561913547688j^{76}-6157114619404j^{75}
   +4490023976796j^{74}$ \\
\hline
\end{tabular}
\end{footnotesize}
\end{table}
\clearpage
\begin{table} \extrarowheight3pt
\begin{footnotesize}
\begin{tabular}{|c|l|}
\hline
$s_6$
& $3^5(435j^{96}-7992j^{95}+87924j^{94}-699468j^{93}+4439614j^{92}
   -23748540j^{91}$ \\
& $+110550147j^{90}-457039124j^{89}+1699021191j^{88}-5722957956j^{87}
   +17550656418j^{86}$ \\
& $-49144172612j^{85}+125850025188j^{84}-294850413428j^{83}
   +631336038299j^{82}$ \\
& $-1231710101144j^{81}+2175858599514j^{80}-3439058133696j^{79}
   +4749139425535j^{78}$ \\
& $-5427629544956j^{77}+4325482123240j^{76}-32855049956j^{75}
   -8453430204420j^{74}$ \\
& $+20584441267052j^{73}-33010938818279j^{72}+39114414206948j^{71}
   -30557131450349j^{70}$ \\
& $+1620402549420j^{69}+44680314225222j^{68}-92172023250604j^{67}
   +114198176941626j^{66}$ \\
& $-86455527400424j^{65}+5634538474360j^{64}+97834761515240j^{63}
   -168594665689154j^{62}$ \\
\hline
\end{tabular}
\end{footnotesize}
\end{table}
\clearpage
\begin{table} \extrarowheight3pt
\begin{footnotesize}
\begin{tabular}{|c|l|}
\hline
$s_6$
& $3^5(3j^{96}-144j^{95}+3360j^{94}-50748j^{93}+557458j^{92}
   -4744540j^{91}+32545747j^{90}$ \\
& $-184785344j^{89}+884847031j^{88}-3622093452j^{87}+12797007322j^{86}
   -39272033892j^{85}$ \\
& $+105037361688j^{84}-244858163148j^{83}+495387442391j^{82}
   -860235215652j^{81}$ \\
& $+1251249341302j^{80}-1438884435872j^{79}+1084751042179j^{78}
   +65463498616j^{77}$ \\
& $-1828933647992j^{76}+3233016334596j^{75}-2471968519804j^{74}
   -2539055576412j^{73}$ \\
& $+13010295509501j^{72}-27901102914908j^{71}+42742983601011j^{70}
   -48859389115160j^{69}$ \\
& $+34069108294878j^{68}+11503916108916j^{67}-81899147451442j^{66}
   +143153209435352j^{65}$ \\
& $-144575243776880j^{64}+61209502815808j^{63}+68563188620154j^{62}
   -161721921260500j^{61}$ \\
\hline
\end{tabular}
\end{footnotesize}
\end{table}
\clearpage
\begin{table} \extrarowheight3pt
\begin{footnotesize}
\begin{tabular}{|c|l|}
\hline
$s_6$
& $3^5(435j^{96}-7200j^{95}+69648j^{94}-503988j^{93}+2962606j^{92}
   -14698272j^{91}$ \\
& $+63430315j^{90}-243443204j^{89}+844509275j^{88}-2679236140j^{87}
   +7842200698j^{86}$ \\
& $-21313186740j^{85}+54032954348j^{84}-128187761296j^{83}
   +285184243407j^{82}$ \\
& $-595654603632j^{81}+1168470474502j^{80}-2151820618916j^{79}
   +3715427265515j^{78}$ \\
& $-6000644653924j^{77}+9030713698668j^{76}-12585361231876j^{75}
   +16074129498932j^{74}$ \\
& $-18467506702100j^{73}+18363787404257j^{72}-14258220081736j^{71}
   +5049992244059j^{70}$ \\
& $+9348622699260j^{69}-27408288932518j^{68}+45684933987652j^{67}
   -59133602624818j^{66}$ \\
& $+62157504203824j^{65}-50299047353600j^{64}+22643023586424j^{63}
   +16928056552762j^{62}$ \\
\hline
\end{tabular}
\end{footnotesize}
\end{table}
\clearpage
\begin{table} \extrarowheight3pt
\begin{footnotesize}
\begin{tabular}{|c|l|}
\hline
$s_6$
& $3^5(3j^{96}-72j^{95}+1020j^{94}-10428j^{93}+84994j^{92}-578824j^{91}
   +3399835j^{90}$ \\
& $-17570360j^{89}+81094171j^{88}-337768444j^{87}+1279931266j^{86}
   -4438319836j^{85}$ \\
& $+14146164224j^{84}-41565138592j^{83}+112803744283j^{82}
   -282962567116j^{81}$ \\
& $+655865381314j^{80}-1402404086964j^{79}+2757960529599j^{78}
   -4962710812936j^{77}$ \\
& $+8104620627660j^{76}-11850490887076j^{75}+15144545775148j^{74}
   -16076230045540j^{73}$ \\
& $+12241849858677j^{72}-1799087176816j^{71}-14924052427717j^{70}
   +34102103186288j^{69}$ \\
& $-48497206943342j^{68}+49796095521580j^{67}-32672643659422j^{66}
   -1648167706304j^{65}$ \\
& $+44244919804184j^{64}-81481718202656j^{63}+98975976552534j^{62}
   -86965598905204j^{61}$ \\
\hline\hline
\end{tabular}
\end{footnotesize}
\end{table}
\clearpage
\begin{table} \extrarowheight3pt
\begin{footnotesize}
\begin{tabular}{|c|l|}
\hline\hline
\multicolumn{2}{|c|} {$\fX_{206}^1$ $\quad[ 5, 2, 165 ]$} \\ \hline
$s_2$
& $(192j^{17}-560j^{16}-216j^{15}-1680j^{14}-1890j^{13}-2985j^{12}
   -3141j^{11}-4650j^{10}-2061j^9$ \\
& $-3635j^8-1386j^7-1260j^6-588j^5-495j^4-105j^3-90j^2-21j-5)^2$ \\
\hline
$s_3$
& $(1072j^{16}+672j^{15}+3420j^{14}+1176j^{13}+5118j^{12}+2364j^{11}
   +4179j^{10}+2940j^9+2431j^8$ \\
& $+1416j^7+1674j^6+480j^5+444j^4+156j^3+87j^2+12j+7)(110592j^{34}
   -258048j^{33}$ \\
& $+664064j^{32}-634368j^{31}+1899072j^{30}+514176j^{29}+3453072j^{28}
   +5616672j^{27}$ \\
& $+7476636j^{26}+18230280j^{25}+23254189j^{24}+37276452j^{23}
   +51056883j^{22}+52895556j^{21}$ \\
& $+66814065j^{20}+55975644j^{19}+57809277j^{18}+43321860j^{17}
   +38634168j^{16}+24631848j^{15}$ \\
\hline\hline
\end{tabular}
\end{footnotesize}
\end{table}
\clearpage
\begin{table} \extrarowheight3pt
\begin{footnotesize}
\begin{tabular}{|c|l|}
\hline\hline
\multicolumn{2}{|c|} {$\fX_{206}^2$ $\quad[ 12, 4, 69 ]$} \\ \hline
$s_2$
& $36864j^{34}+153600j^{33}+3087616j^{32}+1562880j^{31}-1692864j^{30}
   +6204480j^{29}$ \\
& $+6647616j^{28}+11649360j^{27}+15990708j^{26}+8429700j^{25}
   +12257693j^{24}-4601550j^{23}$ \\
& $+495621j^{22}-12420810j^{21}-2196612j^{20}-8352090j^{19}+765549j^{18}
   -2438970j^{17}$ \\
& $+1681575j^{16}-285420j^{15}+680922j^{14}+49500j^{13}+52488j^{12}
   +25020j^{11}-52794j^{10}$ \\
& $+20400j^9-6253j^8+2010j^7+801j^6+2190j^5-180j^4-210j^3-39j^2-90j+25$ \\
\hline
$s_3$
& $267190272j^{50}+516096000j^{49}-498442240j^{48}+4786790400j^{47}
   +28805317632j^{46}$ \\
& $+15760926720j^{45}+37825325568j^{44}+16770926592j^{43}
   +15409611648j^{42}$ \\
& $-20667515904j^{41}-30532376928j^{40}-63721238112j^{39}
   -32673871260j^{38}$ \\
& $-33883196760j^{37}+35813874870j^{36}+42130596564j^{35}
   +94740255297j^{34}$ \\
& $+59915652300j^{33}+75549175446j^{32}+16212924000j^{31}
   +26812350648j^{30}$ \\
& $-17023612032j^{29}+2860838676j^{28}-16269929424j^{27}+1635501924j^{26}
   -5093037936j^{25}$ \\
& $+2980280972j^{24}-34447392j^{23}+1727176608j^{22}+449748432j^{21}
   +374257584j^{20}$ \\
& $+121250808j^{19}-82300194j^{18}+23618952j^{17}-76168248j^{16}
   +3876768j^{15}-19952568j^{14}$ \\
& $+1576512j^{13}-607284j^{12}-648144j^{11}+1431828j^{10}-291312j^9
   +431124j^8-104832j^7$ \\
& $-12420j^6+35784j^5-28998j^4+2772j^3+3177j^2-756j+98$ \\
\hline
$s_5$
& $3^5j^2(j^2+1)^2(j^2+3)^2(2j^3-j^2+2j+1)^2(2j^4+3j^2-1)^2(4j^4+7j^2
   +1)^2(4j^4$ \\
& $+j^3+7j^2+3j+1)^2(2j^5-j^4+3j^3-j+1)^2(4j^5-2j^4+7j^3-j^2+j-1)^2(4j^6
   +9j^4$ \\
& $+2j^2+1)^2(8j^6+9j^4-2j^2+1)^2$ \\
\hline
\end{tabular}
\end{footnotesize}
\end{table}
\clearpage
\begin{table} \extrarowheight3pt
\begin{footnotesize}
\begin{tabular}{|c|l|}
\hline
$s_6$
& $3^5(206158430208j^{102}-515396075520j^{101}+8005819039744j^{100}
   -17849884082176j^{99}$ \\
& $+146373827624960j^{98}-292471569383424j^{97}+1675707377778688j^{96}$ \\
& $-3005862541000704j^{95}+13470664329527296j^{94}
   -21671004166160384j^{93}$ \\
& $+80827031582932992j^{92}-116169667302653952j^{91}
   +375725455460794368j^{90}$ \\
& $-478931174259425280j^{89}+1386856855828365312j^{88}
   -1549001512804876288j^{87}$ \\
& $+4136094546305728512j^{86}-3970616996679892992j^{85}
   +10099172984627249152j^{84}$ \\
& $-8078449449647910912j^{83}+20417827111865170944j^{82}
   -12920326810543176704j^{81}$ \\
& $+34568724392682288896j^{80}-15754586376985016832j^{79}
   +49675551692736283904j^{78}$ \\
\hline
\end{tabular}
\end{footnotesize}
\end{table}
\clearpage
\begin{table} \extrarowheight3pt
\begin{footnotesize}
\begin{tabular}{|c|l|}
\hline
$s_3$
& $30081024j^{50}-45416448j^{49}-436723712j^{48}-917102592j^{47}
   -10117739520j^{46}$ \\
& $-4042192896j^{45}-69057778176j^{44}+871041024j^{43}-289182556800j^{42}
   +30764095488j^{41}$ \\
& $-881547871200j^{40}+44486505504j^{39}-1967762312388j^{38}
   -22701055608j^{37}$ \\
& $-3145344401826j^{36}-99510970860j^{35}-3564606792609j^{34}
   -54816969564j^{33}$ \\
& $-2837699267826j^{32}+58658843376j^{31}-1546818869700j^{30}
   +80326446480j^{29}$ \\
& $-539046052128j^{28}+8200769472j^{27}-89128901232j^{26}
   -34826714496j^{25}$ \\
& $+24130412008j^{24}-16166057040j^{23}+30301273188j^{22}
   +5481272160j^{21}$ \\
& $+14242064220j^{20}+5007949128j^{19}+1455264486j^{18}-96754392j^{17}
   -1457579100j^{16}$ \\
\hline\hline
\end{tabular}%
\end{footnotesize}
\tabcaption{Modular parameterizations of Shimura curves}
\end{table}
\clearpage
\begin{table}[!htbp] \extrarowheight3pt
%
\tabcaption{Choice of Hauptmoduls for $X_0^D(1)/W_D$}
\label{table: Hauptmoduls for Shimura curves 1}
\end{table}
\clearpage
\begin{table}[!htbp] \extrarowheight3pt
\begin{tabular}{|c|l|l|} \hline\hline
$D$ & Characterizing properties & $w_p$ \\ \hline
$14$ & $j(z_{-8})=0,~j(z_{-4})=\infty,~j(z_{-11})=1$ & $z\mapsto-z$ \\ \hline
$21$ & $j(z_{-28})=0,~j(z_{-7})=\infty,~j(z_{-4})=1$ & $z\mapsto-z$ \\ \hline
$26$ & $j(z_{-52})=0,~j(z_{-8})=\infty,~j(z_{-11})=1$ & $z\mapsto-z$ \\ \hline
$33$ & $j(z_{-12})=0,~j(z_{-3})=\infty,~j(z_{-4})=1$ & $z\mapsto-z$ \\ \hline
$34$ & $j(z_{-12})=0,~j(z_{-3})=\infty,~j(z_{-11})=1$ & $z\mapsto-1/z$ \\ \hline
$35$ & $j(z_{-28})=0,~j(z_{-7})=\infty,~j(z_{-8})=1$ & $z\mapsto-z$ \\ \hline
$38$ & $j(z_{-19})=0,~j(z_{-4})=\infty,~j(z_{-11})=1$ & $z\mapsto-z$ \\
\hline
\hline
\end{tabular}
\tabcaption{Choice of Hauptmoduls for $X_0^D(1)/w_D$}
\label{table: Hauptmoduls for Shimura curves 2}
\end{table}
\section{Mestre obstructions}
\label{appendix: Mestre}
In \cite{Mestre}, Mestre gave an algorithm to generate a hyperelliptic
curve of genus $2$ with given Igusa invariants
$J=[J_2,J_4,J_6,J_{10}]$. In the process, he considered a certain
ternary quadratic form
$$
  L(J)=\sum_{1\le i,j\le 3}A_{ij}(J)x_ix_j,
$$
constructed from $J$. He showed that if $L(J)$ is degenerate, then
there is always a curve $C$ of genus $2$ defined over
$\Q(J):=\Q(J_2,J_4,J_6,J_{10})$. In such a case, the curve $C$ has a
nontrivial automorphism different from the hyperelliptic involution.
(For the case of curves over $\C$, this means that the Jacobian of $C$
lies on the Humbert surface of discriminant $4$. See \cite{Bolza}.)
Furthermore, he showed that when $L(J)$ is nondegenerate, there is a
curve $C$ of genus $2$ defined over a field $K$ with Igusa invariants
$J$ if and only if $L(J)$ is isotropic over $K$. In other words, there
is a quaternion algebra $\B$ over $\Q(J)$ such that there is a curve
over $K$ with Igusa invariants $J$ if and only if $K$ splits $\B$. (In
the case when $L(J)$ is diagonal, say,
$L(J)=a_1x_1^2+a_2x_2^2+a_3x_3^2$, $\B$ is simply
$\JS{-a_1a_3,-a_2a_3}{\Q(J)}$). In literature, this quaternion
algebra $\B$ is called the \emph{Mestre obstruction} for $J$. In the
case of the unique Shimura curve $\fX_6$ of discriminant $6$, using
the parameterization given in \eqref{equation: Baba}, Baba and Granath
\cite{Baba-Granath} exhibited a matrix $M\in M(3,\Z[j])$ such that
$$
  M^t(A_{ij}(J))|_{\fX_6} M=-2^{15}3j^4(64j-81)^2
  \begin{pmatrix}1&0&0\\0&6j&0\\0&0&2(27j+16)\end{pmatrix}.
$$
Therefore, for a point on $\fX_6$ that is not on $H_4$, the Mestre
obstruction is $\JS{-6j,-2(27j+16)}{\Q(j)}$.
For the case of $D=10$, Baba and Granath \cite{Baba-Granath} found
that the Mestre obstruction is $\JS{-10j,-5(2j+25)}{\Q(j)}$. (Note
that their choice of Hauptmodul is different from ours.) In this
section, we conduct a similar computation and determine Mestre
obstructions for Shimura curves of genus $0$ and discriminant less
than $80$. The results are given in Table \ref{table: Mestre
  obstructions}.

\begin{Remark}
Note that when $X_0^D(1)/w_D$ has genus $0$ and $j$ is a Hauptmodul
for $X_0^D(1)/w_D$, a canonical model for $X_0^D(1)$ has the form
$y^2=f(x)$, where the roots of $f(x)$ are the values of $j$ at the
fixed points of the Atkin-Lehner involution $w_D$. Our computation
shows that in all cases where the Shimura curve $\fX$ is isomorphic to
$X_0^D(1)/w_D$, the Mestre obstruction for $\fX$ is given by
$$
  \JS{-D,mf(j)}{\Q(j)},
$$
where $m$ is an integer such that $r^2m$ is representable by the
quadratic form associated to $\fX$ for some rational number $r$.
For instance, the quadratic form
for $\fX_{51}^2$ is $5x^2+2xy+41y^2$, which clearly represents $5$.
Also, a canoncal model for $X_0^{51}(1)$ is $y^2=f(x)$, where
$f(x)=-(x^2+3)(243x^6+235x^4-31x^2+1)$ (see
\cite{Molina-isogeny,Guo-Yang}). Our computation shows that the Mestre
obstruction for $\fX_{51}^2$ is $\JS{-51,5f(j)}{\Q(j)}$. If this
phenomenon holds in general, there should be a deep arithmetic
meaning.

Note also that it is well-known that the Shimura curve $X_0^D(1)$ has
no real points when $D>1$ (see \cite{Ogg}). In other words, the
polynomial $f(x)$ above is negative for any real $x$. It follows that
the quaternion algebra 
$\JS{-D,mp(j)}\Q$ always ramifies at the infinite place whenever
$j\in\R$. Therefore, we have the following proposition.
\end{Remark}

\begin{Proposition} Let $\fX$ be a Shimura curve in Table \ref{table:
    Mestre obstructions} that is isomorphic to
  $X_0^D(1)/w_D$. (I.e., the quadratic form associated to $\fX$ is not
  ambiguous.) Then there are only a finite number of isomorphism
  classes of genus $2$ curves over $\R$ such that their
  Jacobians lie on $\fX$. To be more precise, these exceptional moduli
  points are the real points that lie on the intersection of $\fX$ and
  $H_4$, but not on $H_1$.
\end{Proposition}

\begin{Example} \label{example: Mestre 14}
  Consider $\fX=\fX_{14}^1$. Using \eqref{equation: H4}
  and the modular parameterization in Appendix \ref{appendix:
    parameterizations}, we find that the only real points lying on
  $\fX\cap H_4$ have $j$-values $j=\infty$,$0$,$\pm1$,$5/9$,$\pm\sqrt
  5$. Among these points, the points corresponding to $j=\infty,0,1$
  also lie on $H_1$. Thus, there are precisely four genus $2$ curves
  over $\R$ whose Jacobians belong to $\fX$, two of them defined over
  $\Q$ and the other two defined over $\Q(\sqrt 5)$. The point with
  $j=-1$ is a CM-point of discriminant $-11$ with
  $$
    [s_2,s_3,s_5,s_6]=[4,35,972,4617],
  $$
  or equivalently, 
  $$
    [J_2,J_4,J_6,J_{10}]=[-76, 198, 188, -4096].
  $$
  Its ternary quadratic form represents $4x^2+4xy+4y^2$. Thus, it also
  lies on the modular curve $\fY_3'$, defined in Example \ref{example:
    YN'}. Since $\fY_3'$ parameterizes curves of genus $2$ with an
  automorphism group containing the dihedral group $D_6$ of order $12$
  (see \cite[Theorem 2.1]{Cardona}), a curve of genus $2$ with these
  invariants can be obtained using a result of Cardona and Quer
  \cite[Proposition 2.2]{Cardona-Quer-TAMS}. We find that
  $$
    y^2=11x^6+11x^3-4
  $$
  is a curve with $[s_2,s_3,s_5,s_6]=[4,35,972,4617]$.

  The point with $j=5/9$ is a CM-point of discriminant $-43$ with
  $$
    [s_2,s_3,s_5,s_6]=[400/9, 889/3, 400/27, 144001/729]
  $$
  or equivalently,
  $$
    [J_2,J_4,J_6,J_{10}]=\left[-\frac{144001}{10800},
    \frac{15552288001}{2799360000},
    \frac{46655567999}{544195584000000}, -\frac{25}{419904}\right].
  $$
  A curve of genus $2$ over $\Q$ with these invariants is
  $$
    y^2=ax^6+bx^5+cx^4+dx^3-43cx^2-43^2bx-43^3a,
  $$
  where
  \begin{equation*}
  \begin{split}
    a&=55263257981868963587850953171983151, \\
    b&=17933188094622164053062876274057062, \\
    c&=2377689006459293602182305511758203269, \\
    d&=1274547562528177370745959571323659332.
  \end{split}
  \end{equation*}
  An extra involution is given by $(x,y)\mapsto(-43/x,\sqrt{-43^3}y/x^3)$.
\end{Example}

\begin{table}[!htbp] \extrarowheight3pt
\begin{tabular}{|c|l|l|} \hline\hline
curve & $a(j)$ & $b(j)$ \\ \hline
$\fX_6^1$   & $-6j$ & $-2(27j+16)$ \\
$\fX_{10}^1$ & $-10j$ & $-5(2j+25)$ \\
$\fX_{14}^1$ & $-14$  & $-3(16j^4 - 13j^2 + 8)$ \\
$\fX_{15}^1$ & $-15j$ & $-3(j+3)(27j+1)$ \\
$\fX_{15}^2$ & $-15j$ & $-3(j+3)(27j+1)$ \\
$\fX_{21}^1$ & $-21$  & $-6(7j^4 - 94j^2 + 343)$ \\
$\fX_{22}^1$ & $-22j$ & $-11(16j+11)$ \\
$\fX_{26}^1$ & $-26j$ & $-13(2j^3 - 19j^2 + 24j + 169)$ \\
$\fX_{26}^2$ & $-26$  & $-5(2j^6 - 19j^4 + 24j^2 + 169)$ \\
$\fX_{33}^1$ & $-33$  & $-2(243j^4 + 10j^2 + 3)$ \\
$\fX_{34}^1$ & $-34$  & $-5(27j^4 + 136j^3 + 122j^2 - 136j + 27)$ \\
$\fX_{35}^1$ & $-35j$ & $-5(j+7)(7j^3 + 51j^2 + 197j + 1)$ \\
$\fX_{35}^2$ & $-35j$ & $-5(j+7)(7j^3 + 51j^2 + 197j + 1)$ \\
$\fX_{35}^3$ & $-35$  & $-3(j^2+7)(7j^6 + 51j^4 + 197j^2 + 1)$ \\
$\fX_{38}^1$ & $-38j$ & $-2(16j^3 + 59j^2 + 82j + 19)$ \\
$\fX_{39}^1$ & $-39$  & $-15(j^4 - j^3 - j^2 + j + 1)(7j^4 - 23j^3 + 5j^2 + 23j + 7)$\\
$\fX_{39}^2$ & $-39$  & $-2(j^4 - j^3 - j^2 + j + 1)(7j^4 - 23j^3 + 5j^2 + 23j + 7)$\\
$\fX_{46}^1$ & $-46$  & $-10(64j^4 - 45j^2 + 8)$ \\
$\fX_{51}^1$ & $-51j$ & $-3(j+3)(243j^3 + 235j^2 - 31j + 1)$ \\
$\fX_{51}^2$ & $-51$  & $-5(j^2+3)(243j^6 + 235j^4 - 31j^2 + 1)$ \\
$\fX_{51}^3$ & $-51j$ & $-3(j+3)(243j^3 + 235j^2 - 31j + 1)$ \\
$\fX_{55}^1$ & $-55$  & $-13(j^4-5j^3+7j^2+5j+1)(27j^4-19j^3+13j^2+19j+27)$\\
$\fX_{55}^1$ & $-55$  & $-2(j^4-5j^3+7j^2+5j+1)(27j^4-19j^3+13j^2+19j+27)$\\
$\fX_{58}^1$ & $\ 29j$  & $29(j^3 + 39j^2 + 431j + 841)$ \\
$\fX_{62}^1$ & $-62$  & $-3(64j^8 + 99j^6 + 90j^4 + 43j^2 + 8)$ \\
$\fX_{62}^2$ & $-62$  & $-6(64j^8 + 99j^6 + 90j^4 + 43j^2 + 8)$ \\
$\fX_{65}^2$ & $-65\alpha(j)$ & $-13(j^2-4j-1)(7j^4+22j^3-104j^2-46j-7)$\\
$\fX_{65}^3$ & $-65\alpha(j)$ & $-5(j^2-4j-1)(7j^4+22j^3-104j^2-46j-7)$\\
$\fX_{69}^1$ & $-69$  & $-5(243j^8-1268j^6+666j^4+2268j^2+2187)$  \\
$\fX_{69}^2$ & $-69$  & $-21(243j^8-1268j^6+666j^4+2268j^2+2187)$  \\
$\fX_{74}^1$ & $-74j$ & $-37(2j^5 - 47j^4 + 328j^3 - 946j^2 + 4158j + 1369)$\\
$\fX_{74}^2$ & $-74$  & $-15(2j^{10} - 47j^8 + 328j^6 - 946j^4 + 4158j^2 + 1369)$\\
$\fX_{74}^3$ & $-74$  & $-3(2j^{10} - 47j^8 + 328j^6 - 946j^4 + 4158j^2 + 1369)$\\
\hline\hline
\end{tabular}
\tabcaption{Mestre obstruction for $\fX$
 (Here $\alpha(j)=(j^2-4j-1)(j^2+4j-9)$)}
\label{table: Mestre obstructions}
\end{table}


\begin{thebibliography}{10}

\bibitem{Atkin-Lehner}
A.~Oliver~L. Atkin and J.~Lehner.
\newblock Hecke operators on {$\Gamma \sb{0}(m)$}.
\newblock {\em Math. Ann.}, 185:134--160, 1970.

\bibitem{Baba-Granath}
Srinath Baba and H{\aa}kan Granath.
\newblock Genus 2 curves with quaternionic multiplication.
\newblock {\em Canad. J. Math.}, 60(4):734--757, 2008.

\bibitem{Bolza}
Oskar Bolza.
\newblock On binary sextics with linear transformations into themselves.
\newblock {\em Amer. J. Math.}, 10(1):47--70, 1887.

\bibitem{Borcherds}
Richard~E. Borcherds.
\newblock Automorphic forms with singularities on {G}rassmannians.
\newblock {\em Invent. Math.}, 132(3):491--562, 1998.

\bibitem{Borcherds-Duke}
Richard~E. Borcherds.
\newblock Reflection groups of {L}orentzian lattices.
\newblock {\em Duke Math. J.}, 104(2):319--366, 2000.

\bibitem{Bruinier}
Jan~H. Bruinier.
\newblock {\em Borcherds products on {O}(2, {$l$}) and {C}hern classes of
  {H}eegner divisors}, volume 1780 of {\em Lecture Notes in Mathematics}.
\newblock Springer-Verlag, Berlin, 2002.

\bibitem{Cardona}
Gabriel Cardona.
\newblock {$\Bbb Q$}-curves and abelian varieties of {$\rm GL_2$}-type from
  dihedral genus 2 curves.
\newblock In {\em Modular curves and abelian varieties}, volume 224 of {\em
  Progr. Math.}, pages 45--52. Birkh\"auser, Basel, 2004.

\bibitem{Cardona-Quer-TAMS}
Gabriel Cardona and Jordi Quer.
\newblock Curves of genus 2 with group of automorphisms isomorphic to {$D_8$}
  or {$D_{12}$}.
\newblock {\em Trans. Amer. Math. Soc.}, 359(6):2831--2849, 2007.

\bibitem{Cox}
David~A. Cox.
\newblock {\em Primes of the form {$x^2 + ny^2$}}.
\newblock A Wiley-Interscience Publication. John Wiley \& Sons, Inc., New York,
  1989.
\newblock Fermat, class field theory and complex multiplication.

\bibitem{Elkies}
Noam Elkies.
\newblock Shimura curve computations.
\newblock In {\em Algorithmic number theory ({P}ortland, {OR}, 1998)}, volume
  1423 of {\em Lecture Notes in Comput. Sci.}, pages 1--47. Springer, Berlin,
  1998.

\bibitem{Elkies-Kumar}
Noam Elkies and Abhinav Kumar.
\newblock K3 surfaces and equations for {H}ilbert modular surfaces.
\newblock {\em Algebra Number Theory}, 8(10):2297--2411, 2014.

\bibitem{Errthum}
Eric Errthum.
\newblock Singular moduli of {S}himura curves.
\newblock {\em Canad. J. Math.}, 63(4):826--861, 2011.

\bibitem{Gonzalez-Guardia}
Josep Gonz\'alez and Jordi Gu\`ardia.
\newblock Genus two curves with quaternionic multiplication and modular
  {J}acobian.
\newblock {\em Math. Comp.}, 78(265):575--589, 2009.

\bibitem{Molina-isogeny}
Josep Gonz\'alez and Santiago Molina.
\newblock The kernel of {R}ibet's isogeny for genus three {S}himura curves.
\newblock {\em J. Math. Soc. Japan}, 68(2):609--635, 2016.

\bibitem{Guo-Yang}
Jia-Wei Guo and Yifan Yang.
\newblock Equations of hyperelliptic {S}himura curves.
\newblock {\em Compos. Math.}, 153(1):1--40, 2017.

\bibitem{Hashimoto-Murabayashi}
Ki-ichiro Hashimoto and Naoki Murabayashi.
\newblock Shimura curves as intersections of {H}umbert surfaces and defining
  equations of {QM}-curves of genus two.
\newblock {\em Tohoku Math. J. (2)}, 47(2):271--296, 1995.

\bibitem{Hashimoto}
Ko-ichiro Hashimoto.
\newblock Explicit form of quaternion modular embeddings.
\newblock {\em Osaka J. Math.}, 32:533--546, 1995.

\bibitem{Igusa-1960}
Jun-ichi Igusa.
\newblock Arithmetic variety of moduli for genus two.
\newblock {\em Ann. of Math. (2)}, 72:612--649, 1960.

\bibitem{Igusa-1962}
Jun-ichi Igusa.
\newblock On {S}iegel modular forms of genus two.
\newblock {\em Amer. J. Math.}, 84:175--200, 1962.

\bibitem{Igusa-1967}
Jun-ichi Igusa.
\newblock Modular forms and projective invariants.
\newblock {\em Amer. J. Math.}, 89:817--855, 1967.

\bibitem{Jacquet-Langlands}
Herv\'e Jacquet and Robert~P. Langlands.
\newblock {\em Automorphic forms on {${\rm GL}(2)$}}.
\newblock Lecture Notes in Mathematics, Vol. 114. Springer-Verlag, Berlin,
  1970.

\bibitem{Kudla}
Stephen~S. Kudla.
\newblock Integrals of {B}orcherds forms.
\newblock {\em Compositio Math.}, 137(3):293--349, 2003.

\bibitem{Kudla-Rapoport-Yang}
Stephen~S. Kudla, Michael Rapoport, and Tonghai Yang.
\newblock On the derivative of an {E}isenstein series of weight one.
\newblock {\em Internat. Math. Res. Notices}, (7):347--385, 1999.

\bibitem{Lang}
Serge Lang.
\newblock {\em Introduction to algebraic and abelian functions}, volume~89 of
  {\em Graduate Texts in Mathematics}.
\newblock Springer-Verlag, New York-Berlin, second edition, 1982.

\bibitem{Mestre}
Jean-Fran\c{c}ois Mestre.
\newblock Construction de courbes de genre {$2$} \`a partir de leurs modules.
\newblock In {\em Effective methods in algebraic geometry ({C}astiglioncello,
  1990)}, volume~94 of {\em Progr. Math.}, pages 313--334. Birkh\"auser Boston,
  Boston, MA, 1991.

\bibitem{Molina-hyperelliptic}
Santiago Molina.
\newblock Equations of hyperelliptic {S}himura curves.
\newblock {\em Proc. Lond. Math. Soc. (3)}, 105(5):891--920, 2012.

\bibitem{Mumford}
David Mumford.
\newblock {\em Abelian varieties}, volume~5 of {\em Tata Institute of
  Fundamental Research Studies in Mathematics}.
\newblock Published for the Tata Institute of Fundamental Research, Bombay; by
  Hindustan Book Agency, New Delhi, 2008.
\newblock With appendices by C. P. Ramanujam and Yuri Manin, Corrected reprint
  of the second (1974) edition.

\bibitem{Nelson}
Paul~D. Nelson.
\newblock Evaluating modular forms on {S}himura curves.
\newblock {\em Math. Comp.}, 84(295):2471--2503, 2015.

\bibitem{Ogg}
Andrew~P. Ogg.
\newblock Real points on {S}himura curves.
\newblock In {\em Arithmetic and geometry, {V}ol. {I}}, volume~35 of {\em
  Progr. Math.}, pages 277--307. Birkh\"auser Boston, Boston, MA, 1983.

\bibitem{Pollak}
Barth Pollak.
\newblock The equation {$\bar tat=b$} in a composition algebra.
\newblock {\em Duke Math. J.}, 29:225--230, 1962.

\bibitem{Rotger2002}
Victor Rotger.
\newblock {\em Abelian varieties with quaternionic multiplication and their
  moduli}.
\newblock 2002.
\newblock Thesis (Ph.D.)--Universitat de Barcelona.

\bibitem{Rotger-Crelle}
Victor Rotger.
\newblock Quaternions, polarization and class numbers.
\newblock {\em J. Reine Angew. Math.}, 561:177--197, 2003.

\bibitem{Rotger-TAMS}
Victor Rotger.
\newblock Modular {S}himura varieties and forgetful maps.
\newblock {\em Trans. Amer. Math. Soc.}, 356(4):1535--1550, 2004.

\bibitem{Rotger-Igusa}
Victor Rotger.
\newblock Shimura curves embedded in {I}gusa's threefold.
\newblock In {\em Modular curves and abelian varieties}, volume 224 of {\em
  Progr. Math.}, pages 263--276. Birkh\"auser, Basel, 2004.

\bibitem{Runge}
Bernhard Runge.
\newblock Endomorphism rings of abelian surfaces and projective models of their
  moduli spaces.
\newblock {\em Tohoku Math. J. (2)}, 51(3):283--303, 1999.

\bibitem{Schofer}
Jarad Schofer.
\newblock Borcherds forms and generalizations of singular moduli.
\newblock {\em J. Reine Angew. Math.}, 629:1--36, 2009.

\bibitem{Shimizu}
Hideo Shimizu.
\newblock Theta series and automorphic forms on {${\rm GL}_{2}$}.
\newblock {\em J. Math. Soc. Japan}, 24:638--683, 1972.

\bibitem{Shimura-CM1}
Goro Shimura.
\newblock Construction of class fields and zeta functions of algebraic curves.
\newblock {\em Ann. of Math. (2)}, 85:58--159, 1967.

\bibitem{Streng}
Marco Streng.
\newblock Computing {I}gusa class polynomials.
\newblock {\em Math. Comp.}, 83(285):275--309, 2014.

\bibitem{Vigneras}
Marie-France Vign{\'e}ras.
\newblock {\em Arithm\'etique des alg\`ebres de quaternions}, volume 800 of
  {\em Lecture Notes in Mathematics}.
\newblock Springer, Berlin, 1980.

\bibitem{Watson}
Thomas~Crawford Watson.
\newblock {\em Rankin triple products and quantum chaos}.
\newblock ProQuest LLC, Ann Arbor, MI, 2002.
\newblock Thesis (Ph.D.)--Princeton University.

\end{thebibliography}
\end{document}